\documentclass[12pt]{amsart}
\usepackage[colorlinks=true,pagebackref,hyperindex,citecolor=Green,linkcolor=Blue]{hyperref}
\usepackage[top=1in, bottom=1in, left=1.2in, right=1.2in]{geometry}
\usepackage{amsmath}
\usepackage{amsfonts}
\usepackage{amssymb}
\usepackage{amsthm}
\usepackage{stmaryrd}
\usepackage[all]{xy} 
\usepackage{mathrsfs}
\usepackage{enumitem}  
\usepackage[T1]{fontenc}
\usepackage{color}
\usepackage{xcolor}
\usepackage{manfnt}
\usepackage{tikz-cd}
\usepackage{blkarray}
\usepackage[nonewpage]{imakeidx}

\indexsetup{noclearpage}

\setlist[enumerate]{leftmargin=16 pt,align=left,labelwidth=\parindent,labelsep=4pt}
\setlist[itemize]{leftmargin=16 pt,align=left,labelwidth=\parindent,labelsep=4pt}
\makeindex

\makeatletter
\def\namedlabel#1#2{\begingroup
#2%
\def\@currentlabel{#2}%
\phantomsection\label{#1}\endgroup
}
\makeatother

\theoremstyle{theorem} 
\newtheorem{theorem}{Theorem}[section]

\newtheorem{corollary}[theorem]{Corollary}

\newtheorem{lemma}[theorem]{Lemma}
\newtheorem{setup}[theorem]{Setup}
\newtheorem{proposition}[theorem]{Proposition}

\theoremstyle{definition} 
\newtheorem{definition}[theorem]{Definition}
\newtheorem{question}[theorem]{Question}
\newtheorem{example}[theorem]{Example}
\newtheorem{remark}[theorem]{Remark}
\numberwithin{equation}{subsection}

\renewcommand{\(}{\left(}
\renewcommand{\)}{\right)}
\newcommand{\NN}{\mathbb{N}}
\newcommand{\RR}{\mathbb{R}}
\newcommand{\ZZ}{\mathbb{Z}}
\newcommand{\QQ}{\mathbb{Q}}
\newcommand{\FF}{\mathbb{F}}
\newcommand{\CC}{\mathbb{C}}
\newcommand{\PP}{\mathbb{P}}

\newcommand{\m}{\mathfrak{m}}
\newcommand{\n}{\mathfrak{n}}

\newcommand{\cJ}{{\mathcal J}}

\newcommand{\cB}{\mathscr{B}}
\newcommand{\cA}{\mathscr{A}}

\newcommand{\cP}{{\mathcal P}}

\newcommand{\cO}{{\mathcal O}}
\newcommand{\cC}{{\mathcal C}}
\newcommand{\cG}{{\mathcal G}}

\newcommand{\kk}{\Bbbk}

\newcommand*\cube{\mbox{\mancube}}

\newcommand{\shL}{\mathcal L}
\newcommand{\shF}{\mathcal F}

\newcommand{\degtotal}{m}

\newcommand{\degsym}{ n }
\newcommand{ \degoperator}{\ell}

\makeatletter
\def\@tocline#1#2#3#4#5#6#7{\relax
  \ifnum #1>\c@tocdepth % then omit
  \else
    \par \addpenalty\@secpenalty\addvspace{#2}%
    \begingroup \hyphenpenalty\@M
    \@ifempty{#4}{%
      \@tempdima\csname r@tocindent\number#1\endcsname\relax
    }{%
      \@tempdima#4\relax
    }%
    \parindent\z@ \leftskip#3\relax \advance\leftskip\@tempdima\relax
    \rightskip\@pnumwidth plus4em \parfillskip-\@pnumwidth
    #5\leavevmode\hskip-\@tempdima
      \ifcase #1
       \or\or \hskip 1.9em \or \hskip 2em \else \hskip 3em \fi%
      #6\nobreak\relax
    \dotfill\hbox to\@pnumwidth{\@tocpagenum{#7}}\par
    \nobreak
    \endgroup
  \fi}
\makeatother

\newcommand{\Spec}{\operatorname{Spec}}
\newcommand{\Hom}{\operatorname{Hom}}

\newcommand{\Frac}{\operatorname{Frac}}

\newcommand{\End}{\operatorname{End}}
\newcommand{\Ker}{\operatorname{Ker}}

\newcommand{\IM}{\operatorname{Im}}

\newcommand{\gr}{\operatorname{gr}}
\newcommand{\e}{\operatorname{e}}
\newcommand{\Der}{\operatorname{Der}}
\newcommand{\Rank}{\operatorname{rank}}

\newcommand{\sep}[1]{{#1}^{\mathrm{sep}}}
\newcommand{\dif}[2]{^{{\langle #1\rangle}_{\! #2}}}

\newcommand{\Fdif}[1]{{ \llbracket #1 \rrbracket }}
\newcommand{\dm}[1]{\operatorname{s}_{#1}^{\operatorname{diff}}}
\newcommand{\pps}[1]{\operatorname{s}_{#1}^{\operatorname{pp}}}

\newcommand{\Sym}{\operatorname{Sym} }

\newcommand{\Syz}{\operatorname{Syz} }

\newcommand{\Proj}{\operatorname{Proj} }

\newcommand{\tensor}{\otimes}

\newcommand{\mondeg}[1] { \deg \, ( #1) }

\newcommand{\p}{\mathfrak{p}}
\newcommand{\q}{\mathfrak{q}}

\newcommand{\Q}{\mathbb{Q}}

\renewcommand{\a}{\mathfrak{a}}
\renewcommand{\b}{\mathfrak{b}}
\newcommand{\OO}{\mathcal{O}}

\newcommand{\vol}{\mathrm{vol}}
\newcommand{\Vol}{\mathrm{Vol}}
\newcommand{\frk}{\mathrm{freerank}}

\newcommand{\ModDif}[3]{{P}^{#1}_{{#2}|{#3}}}
\newcommand{\wModDif}[3]{\widehat{{P}}^{#1}_{{#2}|{#3}}}

\newcommand{\ddlambda}{{\frac{1}{\lambda !}} {{\partial^{\lambda}}} }

\newcommand{\NF}[1]{{#1}^{\bowtie}}
\newcommand{\lra}{\longrightarrow}

\definecolor{blue-violet}{rgb}{0.54, 0.17, 0.89}
\definecolor{Blue}{rgb}{0.01, 0.28, 1.0}
\definecolor{Green}{rgb}{0.04, 0.85, 0.32}

\begin{document}

\title{Quantifying Singularities With Differential Operators}

\author[H. Brenner]{Holger Brenner}
\address{Holger Brenner, Institut f\"ur Mathematik, Universit\"at Osnabr\"uck, Albrechtstrasse 28a, 49076 Osnabr\"uck, Germany}
\email{holger.brenner@uni-osnabrueck.de}

\author[J. Jeffries]{Jack Jeffries{$^1$}}
\address{Jack Jeffries, Centro de Investigaci\'on en Matem\'aticas, Guanajuato, Gto., M\'exico}
\email{jeffries@cimat.mx}

\author[L. N\'u\~nez-Betancourt]{Luis N\'u\~nez-Betancourt${^2}$}
\address{Luis N\'u\~nez-Betancourt, Centro de Investigaci\'on en Matem\'aticas, Guanajuato, Gto., M\'exico}
\email{luisnub@cimat.mx}

\thanks{{$^1$}The second author was partially supported by the NSF Grant DMS~\#1606353}
\thanks{{$^2$}The third author was partially supported by the NSF Grant DMS~\#1502282 and CONACYT Grant \#284598}


\maketitle

\begin{center}
\dedicatory{\emph{Dedicated to Professor Gennady Lyubeznik on the occasion of his sixtieth birthday}}
\end{center}

\begin{abstract}
The $F$-signature of a local ring of prime characteristic is a
numerical invariant that detects many interesting properties. For example,
this invariant detects (non)singularity and strong $F$-regularity. However,
it is very difficult to compute.

Motivated by different aspects of the $F$-signature, we  define a numerical invariant for rings of characteristic zero or $p>0$
that exhibits many of the useful properties of the $F$-signature. We also
compute many examples of this invariant, including cases where the
$F$-signature is not known.

We also obtain a number of results on symbolic powers and Bernstein-Sato polynomials.
\end{abstract}

\setcounter{tocdepth}{2}
\tableofcontents

%%%%%%%%%%%%%%%%%%%%%%%%%%%%%%%%%%%%%%%%%%%%%%%%%%%%%%%%%%%%%

\section{Introduction}
The $F$-signature of a local ring $(R,\m,K)$ is an important numerical invariant for rings of prime characteristic $p>0$. 
Suppose that $R$ is a reduced ring, and the Frobenius map is a finite morphism. Then, we can work with the ring consisting of its $p^e$-roots, denoted by $R^{1/p^e}$, which is a module-finite extension of $R$. The free rank\index{free rank} of a finitely generated $R$-module $M$ is the largest rank of a free $R$-linear direct summand of $M$. 
Kunz \cite{KunzReg} showed that if $R$ is local, then $R$ is regular if and only $R^{1/p}$ is free (equivalently $R^{1/p^e}$ is free for every or some $e\geq 1$).
One may then expect that the free rank  $R^{1/p^e}$ quantifies the singularity of $R$. This information is encoded in the following invariant.

The $F$-signature of $R$ is defined by
$$
s(R)=\lim\limits_{e\to\infty }\frac{\frk_R(R^{1/p^e}) }{p^{e(d+\alpha)}},
$$
where $d=\dim(R)$, and $p^\alpha=\dim_K K^{1/p}$.
This number was implicitly introduced in the work on rings of differential operators  by Smith and Van den Bergh~\cite{SmithVDB} (see also \cite{SeibertHilbertKunz}) and formally defined by Huneke and Leuschke \cite{HLMCM}, provided that the limit exists. After partial results \cite{HLMCM,YaoObsFsig,AE,SinghSemigroup}, Tucker \cite{TuckerFSig}  showed that  the $F$-signature exists as a limit, rather than just a limsup, in general. 

Yao \cite{YaoObsFsig} gave an important characterization of the $F$-signature in terms of ideals defined by Cartier maps.
Let 
\[I_e=\{r\in R\ |\ \phi(r^{1/p^e})\in\m \hbox{ for every }\phi\in\Hom_R(R^{1/p^e},R)\}.\]
Then, 
\[
s(R)=\lim\limits_{e\to\infty }\frac{\lambda_R(R/I_e)}{p^{ed}}, \]
where $\lambda_R(M)$ denotes the length of $M$ as an $R$-module.

The $F$-signature has a number of properties that relate to different aspects of the singularity of $R$. For example,
\begin{itemize}
	\item[(i)] $s(R)\in [0,1]$.
	\item[(ii)] $R$ is regular if and only if $s(R)=1$ \cite{HLMCM}.
	\item[(iii)] $R$ is strongly $F$-regular if and only if $s(R)>0$ \cite{AL}.
\end{itemize}

Unfortunately, computing the $F$-signature is a very difficult task. Some notable examples include particular rings of invariants. In particular, $s(R^G)=\frac{1}{|G|}$ when $G$ is a finite subgroup of $\mathrm{SL}_n(K)$ of invertible order acting linearly on a polynomial ring $R$. 
In addition, the $F$-signature of an affine toric variety can be computed as the volume of a certain polytope \cite{WatanabeYoshida,VonKorff} (see also \cite{SinghSemigroup}).

By its nature, there is not a notion of $F$-signature in characteristic zero. However, there has been work trying to find an appropriate analogue. Some approaches to this are via reduction to positive characteristic, symmetric powers of syzygies and of K\"{a}hler differentials \cite{SymSig,BCHigh}.

Suppose that $K$ is a perfect field and $R$ is the localization of a finitely generated $K$-algebra  at a maximal  ideal.
As $R^{1/p}$ detects singularity in prime characteristic, the module of  K\"{a}hler differentials, $\Omega_{R|K}$, detects singularity.
As a characteristic-free analogue for $R^{1/p^e}$, we consider ``higher''  K\"{a}hler differentials.
%We also give a description of the differential signature in terms of the free rank of modules.
Let  $\ModDif{n}{R}{K} = (R \otimes_K R) /  \Delta_{R|K}^{n+1}$,
where $\Delta_{R|K}$ is the kernel of the multiplication map $\mu:R\otimes_K R \rightarrow R$. These modules are known as the \emph{modules of principal parts}, introduced by Grothendieck \cite[D\'{e}finition 16.3.1]{EGAIV}.
We note that the module of  $K$-linear K\"{a}hler differentials of $R$ is isomorphic to 
$\Delta_{R|A}/\Delta^2_{R|A}$ and that $D^n_{R|K}\cong \Hom_R(\ModDif{n}{R}{K},R)$, where $ D^n_{R|K} $ denotes the $R$-module of all differential operators on $R$ of order at most $n$. In Theorem~\ref{Mel}, 
we give a  characteristic-free characterization of regularity: $R$ is regular if and only if 
$\ModDif{n}{R}{K}$ is free for every (some) $n\geq 1$.
As with the $F$-signature, one can expect that the free rank of $\ModDif{n}{R}{K}$  quantifies singularity. 
In this work, we introduce a numerical invariant that does this.
We define the  $K$-principal parts signature  of $R$ by
$$\pps{K}(R)=\limsup\limits_{n\to\infty}\frac{\frk_R(\ModDif{n}{R}{K} )}{\Rank( \ModDif{n}{R}{K} ) }.$$
The free rank of $\ModDif{n}{R}{K}$ has an easy interpretation in terms of partial differential equations on $R$. It is the maximal $t$ such that there exists a surjection $\ModDif{n}{R}{K} \rightarrow R^t$. Such a surjection is the same as an independent collection of $t$ differential operators $\delta_1, \ldots, \delta_t$ of order $\leq n$ with the property that the algebraic partial differential equation $\delta_j(-) = 1$ has a solution in $R$, see Lemma \ref{unitarysystem}. Hence the differential signature is a measure for the asymptotic size of such ``unitary'' operators.

We observe that if ``free rank'' is replaced by ``minimal number of generators'' in the previous definition, then one obtains the Hilbert-Samuel multiplicity of $R$. This characterization motivates the following analogy: the principal parts  signature is to the $F$-signature as Hilbert-Samuel multiplicity is to the Hilbert-Kunz multiplicity  (see Remark~\ref{analogy}).

Unfortunately, the module of principal parts might not be a finitely generated $R$-module for rings that are not $K$-algebras essentially of finite type. 
In order to have a definition of a signature for any local $K$-algebra, we make use of the action of differential operators on the ring $R$. 
%This is motivated by the differential nature of the $F$-signature and Yao's characterization, we define an invariant that works in all equal characteristics.
Suppose that $(R,\m)$ is a local ring containing a field $K$ of any characteristic.
We consider the \emph{differential powers} of $\m$, which are given by
$$\m\dif{n}{K} =\{ r\in R\ |\ \delta(r)\in\m \hbox{ for every }\delta\in D^{n-1}_{R|K} \}.$$ 
For prime ideals in polynomial rings, the differential powers coincide with symbolic powers, by the Zariski-Nagata Theorem \cite{ZariskiHolFunct,Nagata} (see also \cite{SurveySP}).
We define the $K$-differential signature of $R$ by
\[
\dm{K}(R)=\limsup\limits_{n\to\infty}\frac{\lambda_R(R/\m\dif{n}{K})}{n^d / d!}.\]

In Theorem~\ref{ThmDiffSigRanks},  we show that 
if $R$ is the localization of a finitely generated $K$-algebra at a maximal ideal, with $K=\overline{K}$, then $\dm{K}(R)=\pps{K}(R)$. In fact, we show this equality in a more general and technical setting.

We are  able to show that the differential signature shares several features with the $F$-signature. 
%In the following, and throughout the paper, we  say that $(R,\m,K)$ is an \emph{algebra with coefficient field $K$} if $R$ is a $K$-algebra that is either $\NN$-graded with $R_0=K$ or local and essentially of finite type over $K$, and $R/\m\cong K$; see Definition~\ref{def-pseudocoefficient} for a slightly more general setup.
Let $R$ be a reduced ring that is the localization of a finitely generated $K$-algebra at a maximal ideal, with $K$ perfect; see Definition~\ref{def-pseudocoefficient} for a slightly more general setup.
\begin{itemize}
	%\item[(i)] If $K$ is perfect and $R$ is a domain, then  $\dm{K}(R)\in [0,1]$. (Corollary~\ref{leq-1})
	\item[(i)] If $R$ is a domain, then  $\dm{K}(R)\in [0,1]$. (Corollary~\ref{leq-1})
	%\item[(ii)] If $K$ is perfect and $R$ regular, then $\dm{K}(R)=1$. (Example~\ref{reg-1})
	\item[(ii)] If $R$ is regular, then $\dm{K}(R)=1$. (Example~\ref{reg-1})
	\item[(iii-a)] If $\dm{K}(R)>0$, then $R$ is a simple $D_{R|K}$-module. (Theorem~\ref{ThmDifMultDsimple})
	\item[(iii-b)] If $R$ is a graded Gorenstein domain in characteristic zero with an isolated singularity and  $\dm{K}(R)>0$, then $R$ has negative $a$-invariant. 
   (Theorem~\ref{Possiganeg})
	\item[(iii-c)] If $K$ has positive characteristic and $R$ is $F$-pure, then $\dm{K}(R)>0$  if and only if $R$ is  strongly $F$-regular. (Theorem~\ref{ThmFregPos})
	\item[(iii-d)] If $K$ has characteristic zero, $R$ has dense $F$-pure type, the anticanonical cover of $R$ is finitely generated, and  $\dm{K}(R)>0$, then $R$ is log-terminal. (Theorem~\ref{ThmKLTPos})
	\item[(iii-e)] If $R$ is a direct summand of a regular ring, then $\dm{K}(R)>0$. (Theorem~\ref{ThmDirSumPos})
\end{itemize}

The behavior of the $F$-signature in a relative setting, say over $\Spec \ZZ$, is not well understood, since it is not possible to compare the splitting behavior of the Frobenius morphisms for different prime characteristics. An advantage of the differential signature is that its definition refers to the module $P^n_{R|K}$ which behaves nicely in a relative setting. See Subsection \ref{SubSecrelative} and in particular Corollaries~\ref{freerankprincipalrelative}	and \ref{freerankprincipalrelativesequence}.

We are able to compute many examples of differential signature of interesting rings. For instance,  $\dm{K}(R^G)=\frac{1}{|G|}=s(R^G)$ when $G$ is a finite subgroup of $\mathrm{SL}_n$ of invertible order acting linearly on a polynomial ring $R$ in Theorem~\ref{group-formula}. In Theorem~\ref{quadricsignature}, we show that $\dm{K}(R)=\left(\frac{1}{2} \right)^{d-1}$ for a quadric hypersurface of dimension $d\geq 2$. We also give a formula for the differential signature of a normal affine toric ring in terms of the volume of a certain polytope in Theorem~\ref{ThmToric}. The values we obtain are positive and rational, but may differ from the values of the $F$-signature. However, the formulas are highly analogous; roughly, up to a common scaling factor, the $F$-signature is the volume of the intersection of the $d$-dimensional unit cube with a linear subspace and the differential signature is $d!$ times the volume of the $d$-dimensional unit simplex intersected with the same linear subspace.
 {In Theorem~\ref{ThmDet}, we  compute the differential signature of $K[X]/I_t(X)$, where $K$ is a field of characteristic zero, $X$ is a generic matrix, and $I_t(X)$ is the ideal generated by the $t\times t$-minors of $X$.} Formulas for symmetric and skew-symmetric rank varieties appear in the same section. We point out that the $F$-signature for these classes of rings is not known for $t>2$.

We expect that the differential signature will find applications to geometry and singularity theory. One such application is in the forthcoming work \cite{JS}, where differential signature is applied to give a characteristic-free approach to bounding \'etale fundamental groups of singularities.

Unfortunately, we are not able to show in general that the differential signature exists as a limit rather than just a limsup. In Theorem~\ref{existenceandraionality}, we show that the limit exists and is rational when $R$ is an algebra with coefficient field $K$ and the associated graded ring of the differential operators with respect to the order filtration is finitely generated. Moreover, every statement about the differential signature in this paper, including (non)vanishing, bounds, and computations, is equally valid for the liminf definition as for the limsup. However, the simple example $\CC[x^2,x^3]$ in Example~\ref{example-not-D-graded} illustrates why recent advances in convergence of numerical limits in Commutative Algebra do not apply to differential signature. 
%The same example, in Example~\ref{example-equals-1}, also illustrates the necessity of the hypotheses in many of the statements listed above. 

For our work on the differential signature, we develop new tools to study differential operators. We introduce an algorithmic framework for the module of principal parts, differential operators, and their induced action on symmetric powers of the module of K\"ahler differentials (Subsections \ref{SubSecJacobi} and \ref{SubAlg}). This is in particular important for computing the differential signature of quadric hypersurfaces.
We define the differential core of an ideal, which emulates the splitting prime of a ring, see Section~\ref{SecDiffPrimes}. 
We define and work with ring extensions that are differentially extensible, a notion implicit in the work of Levasseur-Stafford  \cite{LS} and Schwarz \cite{Schwarz}.
In Theorem~\ref{BS-Thm}, we
use  these extensions to reduce the computation of the new notion of  Bernstein-Sato polynomials in certain  singular rings \cite{AMHNB} to the classical  Bernstein-Sato theory. As a consequence, we obtain a method for computing the Bernstein-Sato polynomials of elements of determinantal rings and other rings of invariants, see Remark~\ref{Rem-com-BS}. Our approach to Bernstein-Sato polynomials is a generalization of the methods of Hsiao and Matusevich \cite{Hsiao-Matusevich}.

% We also introduce the notions of $D$-graded and $D$-reduced rings to study the  differential signature as the multiplicity of a certain graded algebra in Section~\ref{SecGoodDprop}. For these classes of rings, which include all of the aforementioned examples, we are able to show stronger results about the behavior of differential signature and differential powers.

We also establish some new results and new proofs of old results that are of independent interest from differential signature. In Theorem~\ref{Mel}, we give a  characteristic-free characterization of regularity that can be interpreted as a converse to the Zariski-Nagata Theorem (see \cite{ZariskiHolFunct,Nagata,SurveySP}). In  Remark~\ref{Kunz}, we compare our characterization with Kunz's criterion for regularity in prime characteristic \cite{KunzReg}. We provide a Fedder/Glassbrenner-type criterion for $D$-simplicity in Subsection~\ref{FGCriterion}.  
We also give a new description of $F$-signature in Remark~\ref{analogy}, and a simplified proof of the polytope formula for $F$-signature of toric varieties in Theorem~\ref{WYVKS}.

An index with notation and new or uncommon terminology is provided at the end for the reader's convenience.

%%%%%%%%%%%%%%%%%%%%%%%%%%%%%%%%%%%%%%%%%%%%%%%%%%%%%%%%%%%%%

%%%%%%%%%%%%%%%%%%%%%%%%%%%%%%%%%%%%%%%%%%%%
\section{Differential operators and $n$-differentials}\label{basics}
%%%%%%%%%%%%%%%%%%%%%%%%%%%%%%%%%%%%%%%%%%%%%5

In this section, we recall the definition of the ring of differential operators and different characterizations. We often work in the following setting.

\begin{definition}\label{def-pseudocoefficient} A ring $(R,\m,\kk)$ is an \emph{algebra with pseudocoefficient field $K$}\index{algebra with pseudocoefficient field} if $R$ is a commutative $K$-algebra with $1\neq 0$ that is either $\NN$-graded  {and finitely generated over} $R_0=K$ or local and essentially of finite type over $K$, $\m$ is the homogeneous (respectively, the  unique) maximal ideal of $R$, and the inclusion map from $K\rightarrow \kk=R/\m$ is a finite separable extension of fields. If $K=\kk$, which is automatic in the graded case, then $R$ is an \emph{algebra with coefficient field $K$}.
\end{definition}

We note that if $K$ is perfect, and $R$ is a finitely generated $K$-algebra, then $R_{\m}$ is an algebra with pseudocoefficient field $K$ for any maximal ideal $\m \subset R$. That is, coordinate rings of closed points of varieties over perfect fields are of this form. 

\subsection{Differential operators}

\begin{definition}
Let $R$ be a commutative ring and $A$ be a subring, both with $1\neq 0$.
The \textit{$A$-linear differential operators of $R$ of order $n$},\index{$D^n_{R \vert A}$}\index{$D_{R \vert A}$} $D^{n}_{R|A}\subseteq \Hom_A(R,R)$, are defined inductively as follows:
\begin{itemize} 
\item[(i)]  $D^{0}_{R|A} =\Hom_R(R,R).$ 
 \item[(ii)]  $D^{n}_{R|A} = \{\delta\in \Hom_A(R,R)\;|\; [\delta,r]\in D^{n-1}_{R|A} \;\forall \; r \in R \}.$ 
\end{itemize} 
The ring of $A$-linear differential operators is defined by $D_{R|A}=\displaystyle\bigcup_{n\in\NN}D^{n}_{R|A}$.

Throughout, when we discuss rings of differential operators $D_{R|A}$,  the rings $A$ and $R$ are assumed to be commutative with $1\neq 0$.

More generally, if $M$ and $N$ are $R$-modules, one defines $D^{n}_{R|A}(M,N)$ and $D_{R|A}(M,N)$ as submodules of $\Hom_A(M,N)$ by the similar rules:
\begin{itemize} 
\item[(i)]  $D^{0}_{R|A}(M,N) =\Hom_R(M,N).$ 
 \item[(ii)] If $r_M$ and $r_N$ denote the multiplication by $r$ in the modules $M$ and $N$ respectively, then  
$$D^{n}_{R|A} (M,N)= \{\delta\in \Hom_A(M,N)\;|\; \delta r_M- r_N\delta \in D^{n-1}_{R|A}(M,N) \;\forall \; r \in R \}.$$ 
\end{itemize} 
\end{definition}

 These are $R$-modules, where $R$ acts by postcomposition of maps.

The ring structure on $D_{R|A}$ is given by composition and satisfies $D^m_{R|A}D^n_{R|A} \subseteq D^{m+n}_{R|A}$.

 {
\begin{remark}
In this note, we often say that a local $K$-algebra essentially of finite type is smooth without assuming that it is of finite type; in many sources, smooth entails finite type. Here, by smooth we mean that $R$ is flat over $K$ and that
$\Omega_{R|K}$ is projective as an $R$-module.
\end{remark}
}

\begin{example}\label{example-regular-D} Let  $(R,\m,\kk)$ be an algebra with pseudocoefficent field $K$. We recall that a graded or local ring essentially of finite type over $K$ is smooth if and only if $R\otimes_K L$ is regular for some (equivalently, every) perfect field $L/K$; in particular, if $K$ is perfect, this is equivalent to $R$ being regular.
	Set $d=\dim(R)$. In this case,  $R$ is differentially smooth over $K$ \cite[16.10.2]{EGAIV}. We then have the following description of the ring of differential operators.
	
Let $\m=(x_1,\dots,x_d)$, where $x_i$ are homogeneous in the graded case. Then, for each $\alpha=(a_1,\dots,a_d) \in \NN^d$, there is a differential operator $\delta_{\alpha}$ such that 
\[\delta_{\alpha} (x_1^{b_1} \cdots x_d^{b_d}) = \binom{b_1}{a_1} \cdots \binom{b_d}{a_d}  x_1^{b_1-a_1} \cdots x_d^{b_d-a_d},\]
and $D^n_{R|K}=R\langle \, \delta_\alpha \ | \ |\alpha|\leq n \, \rangle$ \cite[16.11.2]{EGAIV}.

In the graded case, where $R$ is a polynomial ring, we  write $ \delta_{\alpha}=\frac{1}{a_1 !} \cdots \frac{1}{a_d !} \frac{\partial^{a_1}}{\partial x_1^{a_1}} \cdots \frac{\partial^{a_d}}{\partial x_d^{a_d}}$, where $\frac{\partial^{a_i}}{\partial x_i^{a_i}}$ is the $a_i$-iterate of the usual partial derivative $\frac{\partial}{\partial x_i}$. Since $\frac{1}{a_1 !} \cdots \frac{1}{a_d !} \in R$ for all $\alpha$ only for $R$ of characteristic zero, we take this as an honest equality of operators there, and merely a formal one in characteristic $p$. By a standard abuse of notation, we write $\frac{1}{\alpha !} \partial^{\alpha}$\index{$\ddlambda$} for the operator $\delta_{\alpha}$ in any characteristic.

In particular, $D_{R|K}$ is generated by $R$ and the partial derivatives $\frac{\partial}{\partial x_i}$ in characteristic zero, and is a Noetherian noncommutative ring. In positive characteristic, $D_{R|K}$ is not Noetherian for $d>0$.
\end{example}

Even for nice rings of characteristic zero, the ring of differential operators may fail to be Noetherian.

\begin{example}\label{example-BGG} Let $R=\CC[x,y,x]/(x^3+y^3+z^3)$, and $\m=(x,y,z)$ be the homogeneous maximal ideal. The following facts hold about $D_{R|\CC}$ \cite{DiffNonNoeth}:
\begin{itemize}
\item $D_{R|\CC}$ is a graded ring.
\item $[D_{R|\CC}]_{<0} = 0$.
\item $[D_{R|\CC}]_{0} = \CC\langle 1, E, E^2, \dots \rangle$, where $E=x \frac{\partial}{\partial x} +y \frac{\partial}{\partial y}+z \frac{\partial}{\partial z}$.
\item For every $n$, $\displaystyle \frac{[D^n_{R|\CC}]_{1}}{[D^{n-1}_{R|\CC}]_{1} +E [D^{n-1}_{R|\CC}]_{1}}\cong \CC^3$ as $\CC$-vector spaces.
\end{itemize}
From this, it follows that $D_{R|\CC}$ is not a finitely generated $\CC$-algebra, and neither left- nor right-Noetherian.
\end{example}

 We note that $R$ is  a left $D_{R|A}$-module, as elements of $D_{R|A}$ are endomorphisms of $R$. We call every  $D_{R|A}$-submodule of $R$ a \emph{$D_{R|A}$-ideal}.\index{D-ideal}

\begin{lemma}[{\cite[Lemma~4.1]{Switala}}]\label{diff-ops-are-cts}
	 Let $R$ be a ring, $A$ be a subring, and $I\subseteq R$ be an ideal. Then every differential operator $\delta\in D_{R|A}$ is $I$-adically continuous.
\end{lemma}

\subsection{Modules of principal parts}

A key description of the differential operators comes from the fact that they are represented by an $R$-module analogously to how derivations are represented by the K\"ahler differentials.

\begin{definition} 
Let $R$ be a ring and $A$ be a subring. The \textit{module of $n$-differentials}, or \textit{principal parts},\index{principal parts}\index{$\ModDif{n}{R}{A}$} of $R$ over $A$, is 
\[ \ModDif{n}{R}{A} = (R \otimes_A R) /  \Delta_{R|A}^{n+1}\]
where $\Delta_{R|A}$\index{$\Delta_{R  \vert A}$} is the kernel of the multiplication map $\mu:R\otimes_A R \rightarrow R$.
\end{definition}

We warn the reader that even if $R$ is Noetherian, $\ModDif{}{R}{A}:=R\otimes_A R$\index{$\ModDif{}{R}{A}$} may not be, and $\ModDif{n}{R}{A}$ may fail to be finitely generated as an $R$-module. However, both of these finiteness conditions hold if $R$ is a essentially of finite type over $A$.

Note that for $A\subseteq R$ and $R$-modules $M$ and $N$, 
there is an $(R\otimes_A R)$-module structure on $\Hom_A(M,N)$ given by the rule
\[ ((a \otimes b) \cdot \phi) (m) = a \phi(b m). \]
We also endow  $\Hom_R( R \otimes_A M, N)$ with an $(R\otimes_A R)$-module structure by the rule
\[ ((a \otimes b) \cdot \phi) (r \otimes m) = \phi(ar \otimes bm). \]
The natural isomorphism
\begin{equation}\label{eq:Hom} \Hom_R(R\otimes_A M, N) \rightarrow  \Hom_A(M,N)
\end{equation}
given by composing the adjunction isomorphism and the evaluation isomorphism is given by $\Phi(\phi)(m)=\phi(1\otimes m)$. One has that $(a\otimes b)(\phi)(r\otimes m) = \phi(ar \otimes bm)= a\phi(r\otimes bm)$, so that $\Phi( (a\otimes b)(\phi))(m)=a \phi(1\otimes bm)$. On the other hand, $(a\otimes b)(\Phi(\phi))(m)=a \phi(1 \otimes bm)$, so the natural isomorphism is an isomorphism of $(R\otimes_A R)$-modules via the structures given above.

The following characterization of differential operators is useful.

\begin{lemma}[{\cite[2.2.3]{HeynemannSweedler}}]\label{diff-ann} A map $\delta\in \Hom_A(M,N)$ is a differential operator of order $\leq n$ if and only if $\Delta_{R|A}^{n+1} \cdot \delta=0$ under the $(R\otimes_A R)$-module action described above.
\end{lemma}

The modules of $n$-differentials represent the differential operators in the same way that the  {K\"ahler} differentials represent the modules of derivations. Namely, let $d^n$\index{$d^n$}\index{universal differential} be the \emph{universal differential} $d^n:R\rightarrow \ModDif{n}{R}{A}$ given by $d^n(x)=\overline{1\otimes x}$. Then one has the following:

\begin{proposition}[{\cite[16.8.8]{EGAIV}, \cite[2.2.6]{HeynemannSweedler}}]\label{universaldifferential} Let $R$ be a ring and $A$ be a subring. For all $R$-modules $M,N$ and $\delta\in D^n_{R|A}(M,N)$, the map $d^*:\Hom_R(\ModDif{n}{R}{A}\otimes_R M,N)\rightarrow D^n_{R|A}(M,N)$ given by $\phi\mapsto \phi\circ d^n$ is an $R$-module isomorphism.
\end{proposition}

We  find it useful to compare different filtrations on rings of differential operators. To this end, and motivated by the previous proposition, we make the following definition. 
\begin{definition} \label{DefIDef}
Let $I$ be an ideal of $R\otimes_A R$, and $M$ and $N$ be $R$-modules. 
\begin{enumerate}
 \item[(i)] The \emph{$I$-differential operators} of $M$ into $N$ are
\[ D^I_{R|A}(M,N)= ( 0 :_{\Hom_A(M,N)} I). \]\index{$D^I_{R \vert A}$}
\item[(ii)] The \emph{module of $I$-differentials} of $M$ is
\[ \ModDif{I}{R}{A}(M) = \frac{R \otimes_A M}{ I \cdot (R\otimes_A M)}.\]\index{$\ModDif{I}{R}{A}$}
\item[(iii)] The \emph{universal $I$-differential}\index{universal differential}\index{$d^n$} on $M$ is the map 
$d^I_{M|A}: M \rightarrow \ModDif{I}{R}{A}(M)$ given by $d^I_{M|A}(m)= \overline{1 \otimes m}$.
\end{enumerate}
\end{definition}

\begin{remark} We also define $\ModDif{}{R}{A}(M)={R \otimes_A M}$ and $\ModDif{n}{R}{A}(M)=\frac{R \otimes_A M}{ \Delta_{R|A}^n \cdot (R\otimes_A M)}$. Note that $\ModDif{}{R}{A}$ is an algebra, and that every $\ModDif{I}{R}{A}$ is a quotient of this ring and $\ModDif{I}{R}{A}(M)$ is a module over this ring.
\end{remark}

\begin{remark}
	Each of the modules of differentials $\ModDif{}{R}{A}$, $\ModDif{I}{R}{A}$, $\ModDif{}{R}{A}(M)$, and $\ModDif{I}{R}{A}(M)$ can be considered as an $R$-module in multiple ways. However, unless specified otherwise, when we consider any of these as an $R$-module, we mean the $R$-module structure coming from the left copy of $R$.
\end{remark}

The following proposition is an extension of Proposition~\ref{universaldifferential}.

\begin{proposition}\label{representing-differential} With the notation in Definition~\ref{DefIDef}, the map 
\[\psi: \Hom_{R}(\ModDif{I}{R}{A}(M),N) \rightarrow D^I_{R|A}(M,N)\] given by
 $\psi(\phi)=\phi \circ d^I_{M|A}$
 is an isomorphism of $R$-modules.
\end{proposition}
\begin{proof} The map $\psi$ is clearly additive, and the action of $R$ on $\ModDif{I}{R}{A}(M)$ corresponds to postcomposition of maps in the adjunction isomorphism~\eqref{eq:Hom}, so it is $R$-linear.
	
It follows from the definitions that $d^I_{M|A} \in D^I(M, \ModDif{I}{R}{A}(M))$. 
It is a routine verification that the image of $\psi$ consists of $I$-differential operators. Since the target is generated as an $R$-module 
by elements in the image of $D^I_{M|A}$, $\psi$ is injective.
Now, by the discussion preceding Equation~\ref{eq:Hom}, we have an $(R\otimes_A R)$-isomorphism from $\Hom_R(R\otimes_A M, N) \rightarrow \Hom_A(M,N)$.
 Thus, given $\delta\in D^I_{R|A}(M,N)\subseteq \Hom_A(M,N)$, $\delta$ factors as an $R$-linear map through $R \otimes_A M$, 
 and since $I \cdot \delta = 0$, $\delta$ factors through $\ModDif{I}{R}{A}(M)$. Thus, $\psi$ is surjective.
 \end{proof}
 
 We note that not every $I$-differential operator is a differential operator. However, one has the following, which follows immediately from Lemma~\ref{diff-ann}.

 \begin{lemma}\label{I-diff-ops}
 With the notation in Definition~\ref{DefIDef}, we have 
  \begin{enumerate}
   \item[(i)] If $I \subseteq R\otimes_A R$ contains $\Delta^{n+1}_{R|A}$ for some $n$, then every $I$-differential operator from $M$ to $N$
   is a differential operator from $M$ to $N$ of order $\leq n$.
   \item[(ii)] If $I_n \subseteq R\otimes_A R$ form a system of ideals cofinal with the powers of $\Delta_{R|A}$ as $n$ varies, then $\delta \in \Hom_A(M,N)$
   is an element of $D_{R|A}$ if and only if $\delta$ is an $I_n$-differential operator for some $n$.
  \end{enumerate}
 \end{lemma}

\subsection{Behavior of differential operators under localization and completion}

All differential operators on a localization occur as localizations of differential operators on the original ring. Namely,

 \begin{proposition}[{\cite[2.2.2 \& 2.2.10]{Masson}}]\label{localization1} Let $K$ be a field, $R$ be a $K$-algebra, $R\rightarrow S$ be formally \'etale, and suppose that  {$\ModDif{n}{R}{K}$} is finitely presented for all $n$. Then the natural maps
 \[ S \otimes_R \ModDif{n}{R}{K} \to \ModDif{n}{S}{K} \qquad \text{and} \qquad S \otimes_R D^n_{R|K} \rightarrow D^n_{S | K}\] are isomorphisms for all~$n$. In particular, formation of differential operators commutes with localization. Additionally, $$D^n_{R|K}=\{\delta \in D^n_{W^{-1}R | K} \ | \ \delta(R) \subseteq R \}$$ for a multiplicative system $W$, if $R$ has no $W$-torsion (so that $R\subseteq W^{-1}R$).
\end{proposition}

Formation of the module of differentials also commutes with localization. We provide a proof to fill a gap in the proof found in the standard reference, and because our statement is somewhat more general. 

\begin{proposition}[{\cite[Theorem~16.4.14]{EGAIV}}]\label{diffmod-localize} Let $A$ be a subring of $R$, and $W$ be a multiplicatively closed subset of $R$. Let $I \subseteq \ModDif{}{R}{A}$ contain $\Delta^n_{R|A}$ for some $n$, and set $I'$ to be the image of $I$ in $\ModDif{}{W^{-1}R}{A}\cong \ModDif{}{W^{-1}R}{(W\cap A)^{-1}A}$. Then, there are isomorphisms of $R$-modules 
	\[W^{-1}\ModDif{I}{R}{A} \cong \ModDif{I'}{W^{-1}R}{A} \cong \ModDif{I'}{W^{-1}R}{(W\cap A)^{-1}A} \]
given by the natural maps. In particular, 
\[W^{-1}\ModDif{n}{R}{A} = \ModDif{n}{W^{-1}R}{A} = \ModDif{n}{W^{-1}R}{(W\cap A)^{-1}A}. \]
\end{proposition}
\begin{proof}
Since 
$$\ModDif{}{W^{-1}R}{A}\cong \ModDif{}{W^{-1}R}{(W\cap A)^{-1}A},$$
 the isomorphism 
 $$\ModDif{I'}{W^{-1}R}{A} \cong \ModDif{I'}{W^{-1}R}{(W\cap A)^{-1}A}$$ is immediate from the definitions. Now, $\ModDif{I'}{W^{-1}R}{A}$ is the localization of $W^{-1}\ModDif{I}{R}{A}$ at the image of $(1\otimes W)$. We remind the reader that, since $W^{-1}\ModDif{I}{R}{A}$ is to be interpreted as a localization as an $R$-module, it is the localization of $\ModDif{I}{R}{A}$ at the image of $(W \otimes 1)$. To show that the map given by localization at the image of $(1\otimes W)$ is an isomorphism, it suffices to show that each element in this multiplicative set is already a unit in $W^{-1}\ModDif{I}{R}{A}$. Let $w\in W\subseteq R$. By assumption, $(w\otimes 1 - 1 \otimes w)^{n+1}=0$ in $W^{-1}\ModDif{I}{R}{A}$, so we can write $w^{n+1}\otimes 1 = (1\otimes w)\cdot \alpha$ for some $\alpha \in W^{-1}\ModDif{I}{R}{A}$. But, $w\otimes 1$ is a unit, so $w^{n+1}\otimes 1$ is a unit, and $1 \otimes w$ is as well. This concludes the proof of the first series of isomorphisms. For the second, we observe that $\Delta^{n+1}_{W^{-1}R|A}$ is the image of $\Delta^{n+1}_{R|A}$ under localization.
	\end{proof}

The following proposition is an immediate consequence of Propositions~\ref{diffmod-localize} and \ref{universaldifferential} combined with the fact that Hom commutes with localization for finitely presented modules. In contrast with Proposition~\ref{localization1}, we do not assume that $A$ is a field.

\begin{proposition}\label{localization2}
	Let $A$ be a subring of $R$, $W$ be a multiplicatively closed subset of $R$, and suppose that $\ModDif{n}{R}{A}$ is finitely presented for all $n$.  Then the natural maps 
	\[W^{-1}R \otimes_R D^n_{R|A} \rightarrow D^n_{W^{-1}R | A} \rightarrow D^n_{W^{-1}R | {(W\cap A)^{-1}A}}\] are isomorphisms for all $n$. More generally, if $I \subseteq \ModDif{}{R}{A}$ contains $\Delta^n_{R|A}$ for some $n$, and $I'$ is the image of $I$ in $\ModDif{}{W^{-1}R}{A}$, then the natural maps
		\[W^{-1}R \otimes_R D^I_{R|A} \rightarrow D^{I'}_{W^{-1}R | A} \rightarrow D^{I'}_{W^{-1}R | {(W\cap A)^{-1}A}}\] are isomorphisms.
\end{proposition}

\begin{remark} The isomorphisms above may be interpreted more concretely as follows. An $A$-linear differential operator $\delta$ on $R$ extends to a differential operator $\tilde{\delta}$ on $W^{-1}R$ by the rule $\tilde{\delta} (\frac{r}{w})=\frac{ \delta(r)}{w}$ if $\delta$ has order zero. Assume for the sake of induction that the action of every element in
$D_{R|A}^{n-1}$ on $W^{-1}R$ is defined. Take $\delta\in D^{n}_{R|A}.$ Then,
\[
\tilde{\delta} \left(\frac{r}{w}\right) = \frac{\delta(r) - \widetilde{[\delta,w]}(\frac{r}{w})}{w},
\]
which is well defined since the order of $[\delta,w]$ is as most $n-1$. Note that one has the equality  $\tilde{\delta} (\frac{r}{1})=\frac{\delta(r)}{1}$ by induction. Then, the previous proposition can be interpreted as saying that, when the modules of principal parts are finitely presented, every $A$-linear differential operator on $W^{-1}R$ of order at most $n$ can be written in the form $\frac{1}{w} \tilde{\delta}$ for some $\delta\in D^n_{R|A}$; one checks easily that this does not depend on the choice of representatives.
\end{remark}

	We need a generalization of Proposition~\ref{localization1} to $I$-differential operators.

\begin{lemma}\label{etalemap}
	Let $A\subseteq R\to S$ be maps of rings. If $I\subseteq \ModDif{}{R}{A}$ and $J\subseteq \ModDif{}{S}{A}$ are such that $I \ModDif{}{S}{A} \subseteq J$, then there is an $S$-module homomorphism $\alpha:S \otimes_R \ModDif{I}{R}{A} \to \ModDif{J}{S}{A}$.
	
	If $I=\Delta_{R|A}^n$ and $J=\Delta_{S|A}^n$, then $\alpha$ is an isomorphism on the \'etale locus of the map $R\to S$. Similarly, in characteristic $p>0$, $I=\Delta_{R|A}^{[p^e]}$ and $J=\Delta_{S|A}^{[p^e]}$, then $\alpha$ is an isomorphism on the \'etale locus of the map $R\to S$.
\end{lemma}
\begin{proof}
	The  map $\alpha$ is just the map given by $(S \otimes_R R\otimes_A R)/I^e \to (S \otimes_R R\otimes_A R \otimes_R S)/I^e \to (S \otimes_R R\otimes_A R \otimes_R S)/J$.
	
	To verify that $\alpha$ is an isomorphism when stated, we use the local structure theorem for \'etale maps \cite[{Tag~025A}]{stacks-project}. Write $S$ as a localization (at $W^{-1}$) of $R[\theta]/f(\theta)$ in which $f'(\theta)$ is invertible. In the case of powers, the map $\alpha$ takes the form
	\[ \frac{W^{-1} R[\theta]\otimes_A R }{\Delta^n_{R|A} + (f(\theta))} \to \frac{W^{-1} R[\theta]\otimes_A W^{-1} R[\overline{\theta}] }{(\Delta_{R|A} + (\overline{\theta}-\theta))^n + (f(\theta),f(\overline{\theta}))}.\]
	By the same argument as Proposition~\ref{diffmod-localize}, we can rewrite the right-hand side as
	\[\frac{((W^{-1} R)\otimes_A R)[\theta,\overline{\theta}] }{(\Delta_{R|A} + (\overline{\theta}-\theta))^n + (f(\theta),f(\overline{\theta}))}.\]
	We write $\theta'=\overline{\theta}-\theta$ and use the Taylor expansion of $f(\theta+\theta')$ to rewrite the target module as
	\[\frac{((W^{-1} R)\otimes_A R)[\theta,{\theta'}] }{(\Delta_{R|A} + \theta')^n + (f(\theta),\theta' + \frac{\theta'^2 f''(\theta)}{2! f'(\theta)}+\cdots))}.\]
	Given an element in this module, we can expand as a polynomial expression in $\theta'$. If there is a term of the form  $B \theta'^i$ , with $1\leq i <n$, we can subtract off $B (\theta'^i + \frac{\theta'^{i+1} f''(\theta)}{2! f'(\theta)}+\cdots)$ to obtain an expression where the least such $i$ for which $B\neq 0$ is larger. Iterating this, we obtain an expression for the element with no $\theta'$ term. That is, the map $\alpha$ is an isomorphism.
	
	The argument is entirely analogous in the case of Frobenius powers.
\end{proof}

\begin{proposition}\label{RankDiff} Let $K$ be a field, and $(R,\m,\kk)$ be a local or graded domain that is essentially of finite type over $K$. Suppose that $\Frac(R)$ is separable over $K$. Let $d=\dim R$ and $t$ be the transcendence degree of $\kk$ over $K$. Then  { ${\Rank_R(\ModDif{n}{R}{K})=\binom{d+t+n}{d+t}}$}. 
\end{proposition}
\begin{proof} Let $F=\Frac(R)$, and $e=d+t$, which, by standard dimension theory, is the transcendence degree of $F$ over $K$. Since $F$ is separable over $K$, we can write $F=K(x_1,\dots,x_e)(\alpha)$, where $x_1,\dots,x_e$ are a transcendence basis for $F$ over $K$, and $F$ is a separable algebraic extension of $L=K(x_1,\dots,x_e)$.
It follows from  {Proposition~\ref{diffmod-localize}} that $\Rank_R(\ModDif{n}{R}{K})=\Rank_F(\ModDif{n}{F}{K})$. Then, since $F$ is a separable extension of $L$, by Lemma~\ref{localization1}, we have  {$\Rank_F(\ModDif{n}{F}{K}) = \Rank_L(\ModDif{n}{L}{K})$}. Applying  {Proposition~\ref{diffmod-localize}} again, this is equal to  {$\Rank_R(\ModDif{n}{K[x_1,\dots,x_e]}{K})$}. In this case, we compute  {$\displaystyle \ModDif{n}{K[x_1,\dots,x_e]}{K}\cong \frac{K[x_1,\dots,x_e,z_1,\dots,z_e]}{(z_1,\dots,z_e)^n}$} (see \S\ref{SubSecJacobi} below) from which the claim follows.
\end{proof}

\begin{definition} Let $(R,\m,\kk)$ be a complete local ring, and $A\subseteq R$ be a subring. The \emph{complete module of $n$-differentials} or \emph{complete module of principal parts} of $R$ over $A$ is
	$\wModDif{n}{R}{A}$\index{$\wModDif{n}{R}{A}$}, the $\m$-adic completion of $\ModDif{n}{R}{A}$.
\end{definition}

\begin{lemma}\label{separatedquot} Let $(R,\m,\kk)$ be a complete local ring, and $K\cong \kk$ be a coefficient field. Then,
	\begin{enumerate}
		\item 	$\wModDif{n}{R}{K}$ is finitely generated, and
		\item  $\wModDif{n}{R}{K}\cong \sep{(\ModDif{n}{R}{K})}$, where $\sep{M}=M/(\cap_{n=1}^\infty\m^n M)$\index{$\sep{M}$}, the maximal separated quotient of $M$.
	\end{enumerate}	
\end{lemma}
\begin{proof}
	Each $\sep{(\ModDif{n}{R}{K})}$ is finitely generated over $R$, hence is complete \cite[Remark~4.7]{Switala}. The isomorphism in (2) then follows from the universal properties of the two modules, and the first statement is then immediate.
\end{proof}

The following proposition is an analogue of Proposition~\ref{representing-differential}  for complete rings.

\begin{proposition}\label{represent-complete} Let $(R,\m,\kk)$ be a complete local ring, and $A\subseteq R$ be a subring. Then $D^n_{R|A} {\cong}\Hom_R(\wModDif{n}{R}{A},R)$.
\end{proposition}
\begin{proof}
	By Proposition~\ref{representing-differential}, we have $D^n_{R|A}\cong \Hom_R(\ModDif{n}{R}{A},R)$. Since $R$ is complete, a map from $\ModDif{n}{R}{A}$ to $R$ factors uniquely through $\wModDif{n}{R}{A}$.
\end{proof}

The analogue of Proposition~\ref{localization2} holds as well.

\begin{proposition}[{\cite[2.3.3]{LyuUMC}}]\label{diff-ops-completion} Let $(R,\m,K)$ be an $A$-algebra essentially of finite type, with $A$ Noetherian. Then there are isomorphisms
	\[ \widehat{R} \otimes_R D^n_{R|A} \rightarrow D^n_{\widehat{R}|A}. \]
\end{proposition}

We note that for an algebra with a pseudocoefficient field $K$, the modules of principal parts $\ModDif{n}{R}{K}$ are all finitely presented, so Proposition~\ref{localization2} applies.

%%%%%%%%%%%%%%%%%%%%%%%%%
\subsection{The Jacobi-Taylor matrices}\label{SubSecJacobi}
%%%%%%%%%%%%%%%%%%%%%%%%%

In this subsection we introduce a family of matrices that give a presentation for the modules of principal parts and a computationally easy description of differential operators. We use these matrices for algorithmic aspects of the differential signature in Subsection~\ref{SubAlg}. In particular, we compute the differential signature for quadrics in Subsection~\ref{SubQuadrics}.

We point out that a closely related version of the Jacobi-Taylor matrices were independently and simultaneously introduced by Barajas and Duarte \cite{BarajasDuarte} under the name of higher Jacobian matrices. We point out that the hypersurface case was already studied by Duarte \cite{Duarte}. 
 
For polynomials $f_i$, $1 \leq i \leq m$, in $k$ variables the usual Jacobian matrix $J=\left( \partial_j f_i  \right)$ provides a representation
\[ R^m \stackrel{J^\text{tr} }{\longrightarrow} R^k  \longrightarrow \Omega_{R|K} \longrightarrow 0, \]
 {where $J^\text{tr}$ denotes the transpose matrix of $J$, }
of the module of K\"ahler differentials. In this subsection we provide a similar description of the module of principal parts. For a $k$-tuple
$\lambda \in \NN^k$ we define the operators $ \frac{ 1 }{ \lambda! } \partial^\lambda $ as in Example~\ref{example-regular-D}.

\begin{lemma}\label{principalpartdescription}
	Let $f_1 , \ldots , f_m \in K[x_1 , \ldots , x_k]$ denote polynomials with residue class ring
	$ R  =  K[x_1 , \ldots , x_k]/ \left( f_1 , \ldots , f_m \right) $. Then
	\[ R \otimes_K R  \cong  R[y_1 , \ldots , y_k]/ \left( g_1 , \ldots , g_m \right) , \]
	where $ g_i   = \sum_\lambda g_ {i, \lambda } y^\lambda $ and $ g_{i, \lambda }   =  \frac{ 1 }{ \lambda ! } \partial^\lambda (f_i).$
\end{lemma}
\begin{proof}
	We work with the description 
	\[	\begin{aligned}		
	R \otimes_{ K } R  
	&=  { K[x_1 , \ldots , x_k ] / \left( f_1 , \ldots , f_m \right) \otimes_{ K } K[x_1 , \ldots , x_k ] /\left( f_1 , \ldots , f_m \right) } \\
	& = { K[x_1 , \ldots , x_k , \tilde{x}_1 , \ldots , \tilde{x}_k ] /\left( f_1 , \ldots , f_m , \tilde{f}_1 , \ldots , \tilde{f}_m\right) } ,
	\end{aligned} \]
	where $\tilde{f_i}$ comes from $f_i$ by replacing $x_j$ by $\tilde{x}_j$. We put $ y_j 
	= \tilde{x_j} - x_j $ and write the ring as
	\[  K[x_1 , \ldots , x_k, y_1 , \ldots , y_k ]/ \left( f_1 , \ldots , f_m, g_1 , \ldots , g_m \right)   =  R[ y_1 , \ldots , y_k ]/ \left( g_1 , \ldots , g_m \right) , \]
	where
	\[ 	 g_i = \tilde{f_i}  =  f_i \left( \tilde{x}_1 , \ldots , \tilde{x}_k \right) 
	=   f_i \left( x_1+y_1 , \ldots , x_k+y_k \right). \] Consider a monomial $x_1^{\nu_1} \cdots x_k^{\nu_k}$ in some $f$. This corresponds to a term in $g$ of the form
	\[ ( x_1+y_1)^{\nu_1} \cdots (x_k+y_k)^{\nu_k} .\]
	Multiplying out yields
{\small	
	\[   \sum_{\lambda \leq \nu} \binom { \nu_1 } { \lambda_1} \cdots \binom { \nu_k } { \lambda_k} x_1^{\nu_1- \lambda_1} y_1^{\lambda_1} \cdots x_k^{\nu_k- \lambda_k} y_k^{\lambda_k} 
	=  \sum_{\lambda \leq \nu} \binom { \nu_1 } { \lambda_1} \cdots \binom { \nu_k } { \lambda_k} x_1^{\nu_1- \lambda_1} \cdots x_k^{\nu_k- \lambda_k} y_1^{\lambda_1} \cdots y_k^{\lambda_k} . \]
}	
	Thus, the term for the monomial $y^\lambda$ in $g$ is
	$  {\binom { \nu_1 } { \lambda_1} \cdots \binom { \nu_k } { \lambda_k} x_1^{\nu_1- \lambda_1} \cdots x_k^{\nu_k- \lambda_k}}$, which coincides with $\frac{ 1 }{ \lambda ! } \partial^\lambda (x^\nu) $.
\end{proof}

\begin{definition} Let $f_1 , \ldots , f_m \in K[x_1 , \ldots ,x_k]$ be polynomials. For
	$n \in \NN$, let \[ \mathcal{A}   =  \left\{ (\mu,i) \ | \ \mu \in \NN^k \text{ such that} \mondeg {\mu} \leq n-1 , \, 1 \leq i \leq m \right\}\] and \[ \mathcal{B}  = \left\{ \nu \ | \ \nu \in \NN^ k \text{ such that } \mondeg { \nu } \leq n \right\} .\] Then the $\mathcal{A} \times \mathcal{B}$ matrix with entries
	\[ a_{( \mu,i ; \nu)} = \frac{ 1 }{ (\nu - \mu )! } \partial^{\nu - \mu} (f_i)\]
	is called the $n$-th \emph {Jacobi-Taylor matrix}.\index{Jacobi-Taylor matrix}
\end{definition}

We denote these matrices by $J_n$\index{$J_n$}. We may consider them over the polynomial ring or over the residue class ring. To give an example, in three variables and one equation $f$, the transposed second Jacobi-Taylor matrix is given as
\[ \begin{blockarray}{ccccc}
 & 1 & a & b & c \\ 
 \begin{block}{c[cccc]}
 1 & 0 & 0 & 0 & 0 \\ 
 a & \partial_x (f) & 0 & 0 & 0 \\
 b & \partial_y (f) & 0 & 0 & 0 \\
 c & \partial_z (f) & 0 & 0 & 0 \\
 a^2 & \frac{ 1 }{ 2 } \partial_x \partial_x (f) & \partial_x (f) & 0 & 0 \\
 ab & \partial_x \partial_y (f) & \partial_y (f) & \partial_x (f) & 0 \\ 
 ac & \partial_x \partial_z (f) & \partial_z (f) & 0 & \partial_x (f) \\ 
 b^2 & \frac{ 1 }{ 2 } \partial_y \partial_y (f) & 0 & \partial_y (f) & 0 \\
 bc & \partial_y \partial_z (f) & 0 & \partial_z (f) & \partial_y (f) \\  
 c^2 & \frac{ 1 }{ 2 } \partial_z \partial_z (f) & 0 & 0 & \partial_z (f) \\
   \end{block}
    \end{blockarray} , \]
where $1,a,b,c$ and $1,a,\dots, c^2$ indicate which column (respectively, row) corresponds with which indexing element of $J$ (respectively, $I$).
Note that in all $J_n$, for varying $n$, only a finite number of distinct entries occur, namely all partial derivatives of the $f_i$.

We now prove that this matrix  {gives} a presentation for the module of principal parts.
\begin{corollary}
	\label{principalpartsrepresentation}
	Let $f_1 , \ldots , f_m \in K[x_1 , \ldots , x_k]$ denote polynomials with residue class ring
	$R  =  K[x_1 , \ldots , x_k]/ \left( f_1 , \ldots , f_m \right)$. Then the module of principal parts $P^n_{R | K}$ has the presentation 
	\[ \bigoplus_{\substack{ \mondeg {\mu} \leq n - 1 \\ 1 \leq i \leq m }} R e_{\mu ,i} \stackrel{J_n^\mathrm{tr}  } { \longrightarrow }\bigoplus_{ \mondeg {\lambda} \leq n } R e_{\lambda } \longrightarrow P^n_{R | K} \longrightarrow 0 . \]
\end{corollary}
\begin{proof}
	Due to Lemma~\ref{principalpartdescription} we have
	\[ P^n_{R{{|}} K}  = (R \otimes_K R) /\Delta^{n+1} \cong  R[y_1 , \ldots , y_k]/ \left( g_1 , \ldots , g_m , y^\lambda , \mondeg {\lambda} \geq n+1 \right) \, \]
	In particular, the monomials
	$y^\lambda$, $ \mondeg  {\lambda } \leq n$, give an $R$-module generating system for $P^n_{R {{|}} K}$ and a surjective mapping
	\[ \bigoplus_{ \mondeg  {\lambda} \leq n } R e_\lambda \longrightarrow P^n_{R | K}  ,\, e_\lambda \longmapsto  y^\lambda .\]
	The part of the ideal generated by $g_i$ of degree $\leq n$ is generated as an $R$-module by
	\[  y^\mu g_i = y^\mu \left( \sum_\lambda g_{i, \lambda} y^\lambda \right) , \, \mondeg {\mu} \leq n-1, \, 1 \leq i \leq m  \,  . \]
	Hence the kernel of the mapping is generated by all $\lambda$-tuples  {
	\[ C_{\mu, i} = \left( C_{\nu; \mu,i } \right) \text{ with } C_{\nu; \mu,i} = g_{i, \nu-\mu } ,\, \mondeg {\mu}  \leq n-1, \, 1 \leq i \leq m .   \]}
	So the kernel is the image of the map
	\[ \bigoplus_{\substack{ \mondeg {\mu}  \leq n - 1 \\ 1 \leq i \leq m }} R e_{\mu,i} \longrightarrow \bigoplus_{\mondeg {\lambda} \leq n } R e_\lambda , \, e_{ \mu,i} \longmapsto  C_{\mu, i} .\]
	The entry of this matrix in row index $\nu$ and column index $(\mu,i)$ is
	\[ g_{i, \nu - \mu} = \frac{ 1 }{ (\nu-\mu)! } \partial^{\nu - \mu} (f_i) , \] 
	so this is the transposed Jacobi-Taylor matrix.
\end{proof}

\begin{remark}
	\label{JacobiTaylorrelation}
	Every Jacobi-Taylor matrix evolves from the previous one in the block matrix form
	\[ J_n^\text{tr} = 	\begin{pmatrix} J_{n-1}^\text{tr} & 0 \\ S_n & T_n \end{pmatrix} \]
	In the matrix $T_n$ we only have first partial derivatives of the $f_i$, this matrix sends $e_{\mu,i}$ to $ \sum_j \partial_j(f_i) \cdot e_{\mu +e_j}$.
	We have a commutative diagram with exact rows	
		\[\xymatrix{ 0 \ar[r] &\bigoplus\limits_{ \mondeg {\mu}  = n-1, i } R e_{ \mu,i }
			\ar[r]\ar[d]_-{T_n} &\bigoplus\limits_{ \mondeg {\mu}  \leq n-1, i } R e_{ \mu,i } \ar[r]\ar[d]_-{J^{\mathrm{tr}}_n} &\bigoplus\limits_{ \mondeg {\mu}  \leq n-2, i } R e_{ \mu,i } \ar[r]\ar[d]_-{J^{\mathrm{tr}}_{n-1}} & 0 \\
			0 \ar[r] & \bigoplus\limits_{ \mondeg {\lambda}  = n } R {e_\lambda } \ar[r]\ar[d] & \bigoplus\limits_{ \mondeg {\lambda} \leq n } R {e_\lambda } \ar[r]\ar[d] & \bigoplus\limits_{ \mondeg {\lambda} \leq n-1 } R {e_\lambda } \ar[r]\ar[d] & 0\\
			0 \ar[r] & \Delta^{n}/\Delta^{n+1} \ar[r] \ar[d] & \ModDif{n}{R}{K}\ar[r]\ar[d] & \ModDif{n-1}{R}{K}\ar[r]\ar[d] & 0\\ & 0 &0& \, 0 . }  \]
		The first two rows split. The columns in the middle and on the right are also exact. In the column on the left the second map is surjective and it is exact provided that $J^{\mathrm{tr}}_{n-1}$ is injective. This is not always the case, e.g. in positive characteristic it might be that all first partial derivatives and hence $T_n$ is $0$.
\end{remark}

\begin{corollary}
	\label{JacobiTayloroperators}
	Let $f_1 , \ldots , f_m \in K[x_1 , \ldots , x_k]$ denote polynomials with residue class ring
	$R  =  K[x_1 , \ldots , x_k]/ \left( f_1 , \ldots , f_m \right)$.
	Then differential operators on $R$ of order $\leq n $ correspond to elements in the kernel of the $n$-th Jacobi-Taylor matrix. A $\lambda$-tuple $\left( a_\lambda \right)$ in the kernel corresponds to the operator that is represented on the level of the polynomial ring by
	\[  \sum_\lambda a_\lambda \frac{ 1 }{ \lambda! } \partial^\lambda .\]
\end{corollary}
\begin{proof}
	We work with the presentation
	\[ \bigoplus_{\substack{ \mondeg {\mu}  \leq n - 1 \\ 1 \leq i \leq m }} R e_{\mu ,i} \stackrel{  J_n^{ \mathrm{tr} } } {\longrightarrow}  \bigoplus_{  \mondeg {\lambda}  \leq n } R e_{\lambda } \longrightarrow P^n_{R {{|}} K} \longrightarrow 0 \]
	from Corollary~\ref{principalpartsrepresentation}. A differential operator on $R$ is the same as an $R$-linear form on $	P^n_{R | K}$. This again is the same as an $R$-linear form $\varphi$ on $ \bigoplus_{  \mondeg {\lambda} \leq n } R e_{\lambda }$ (given by an $R$-tuple $\left(  a_\lambda \right)$), fulfilling $\varphi \circ { J_n^{ \text{tr} } } =  0$. This is equivalent with $ J_n \circ { \varphi^{ \text{tr} } } = 0$.
	
	In the notation of Lemma~\ref{principalpartdescription}, the universal operator $d^n:  R \rightarrow  P^n_{R{{|}} K}$ sends a monomial
	$x^\nu$ to
	\[	\begin{aligned} 1 \otimes x^\nu  & =   1 \otimes x_1^{\nu_1} \cdots x_k^{\nu_k} \\
	& =  \tilde{x}_1^{\nu_1} \cdots \tilde{x}_k^{\nu_k} \\
	& = (x_1+y_1)^{\nu_1} \cdots (x_k+y_k)^{\nu_k} \\
	& = \sum_{\lambda \leq \nu} \binom { \nu_1 } { \lambda_1} \cdots \binom { \nu_k } { \lambda_k} x_1^{\nu_1- \lambda_1} \cdots x_k^{\nu_k- \lambda_k} y_1^{\lambda_1} \cdots y_k^{\lambda_k} .
	\end{aligned} \]
	The composition with the linear form on $P^n_{R | K}$ given by $\left( a_\lambda \right)$ yields 
	\[ \sum_{\lambda \leq \nu} \binom { \nu_1 } { \lambda_1} \cdots \binom { \nu_k } { \lambda_k} x_1^{\nu_1- \lambda_1} \cdots x_k^{\nu_k- \lambda_k} a_\lambda \,  . \]
	This coincides with
	\[  \sum_\lambda a_\lambda \frac{ 1 }{ \lambda! }  \partial^\lambda \left( x^\nu \right).\]\qedhere
	\end{proof}

\begin{remark}
	It follows from the previous corollary that a differential operator $\delta$ on $K[x_1,\dots,x_k]$ of order $n$ descends to a differential operator on $K[x_1,\dots,x_k]/(f_1,\dots,f_m)$ if and only if $\delta(x^{\lambda} f_i) \in (f_1,\dots,f_m)$ for all $i$ and all $\lambda$ of degree at most $n-1$. This fact is known to experts, but we could not find a clear reference.
\end{remark}

For the universal differential operator $d^n$ we have a canonical lifting
\[ 
\xymatrix{
&  R \ar[dl]_-{d^{n'}} \ar[dr]^-{d^n} & \\
  \bigoplus\limits_{  \mondeg {\lambda} \leq n } R e_{\lambda } \ar[rr]   & & \ModDif{n}{R}{K} .
 }\]
An element $h$ is sent by $d^{n \prime}$ to $\sum_{ \mondeg {\lambda} \leq n} \frac{1}{\lambda!} \partial^\lambda (h) e_\lambda$. The commutativity follows from the proof of Corollary~\ref{JacobiTayloroperators}.

\begin{remark}
Let $R=K[x_1, \ldots, x_k]	/(f_1, \ldots , f_m)$ and let $W$ be a multiplicative subset of $R$.
Then by Proposition~\ref{diffmod-localize} we have $  P^n_{W^{-1}R|K}  \cong W^{-1} P^n_{R|K} \cong P^n_{R|K} \otimes_R  W^{-1}R $. Hence the representation for the module of principal parts given by the Jacobi-Taylor matrices given in Corollary~\ref{principalpartsrepresentation} can be used directly also for algebras essentially of finite type over $K$ and in particular for localizations. 
\end{remark}

\begin{lemma}
	\label{JacobiTaylorsymmetric}
	Let $f_1 , \ldots , f_m \in K[x_1 , \ldots , x_k]$ denote polynomials with residue class ring $R  =  K[x_1 , \ldots , x_k]/ \left( f_1 , \ldots , f_m \right)$. Let $\delta\in D^n_{R|K}$ be given by the $\lambda$-tuple $\left( a_\lambda \right), \deg(\lambda)\leq n$ in the kernel of the $n$-th Jacobi-Taylor matrix in the sense of Corollary~\ref{JacobiTayloroperators}.
	The image of $\delta$ under the natural $R$-linear map $D^{n}_{R|K} \rightarrow \Hom_R(\Sym^n(\Omega_{R|K} ) , R )$ is given by the restricted tuple $\left( a_\lambda \right)$, $ \mondeg {\lambda} = n $.
\end{lemma}
\begin{proof}
	From the presentation
	\[ \bigoplus_{i = 1}^m R e_i \stackrel{J^\text{tr} }{\longrightarrow} \bigoplus_{j = 1}^k R e_j \longrightarrow \Omega_{R|K} \longrightarrow 0 \]
	we get for the symmetric powers $\Sym^n (\Omega_{R|K} ) $ the presentation

	\[ \xymatrix{ \(\bigoplus\limits_{i = 1}^m R e_i \)  \tensor \Sym^{n-1}\( \bigoplus\limits_{j = 1}^k R e_j \) \ar[r]\ar[d]^{\cong}  & \Sym^n\( \bigoplus\limits_{j = 1}^k R e_j  \) \ar[r]\ar[d]^{\cong} & \Sym^n \(\Omega_{R|K} \) \ar[r]\ar[d]^{\cong} & 0 \\ 
	\bigoplus\limits_{i, \mondeg {\mu}  = n-1 } R e_{ \mu, i} \ar[r] & \bigoplus\limits_{  \mondeg {\lambda}  = n} Re_\lambda \ar[r] &  \Sym^n\(\Omega_{R|K} \) \ar[r] & 0 }	\]
	sending $e_\lambda \mapsto (dx)^\lambda$ and $e_{\mu, i} \mapsto  \sum_j \partial_j (f_i) e_{\mu + e_j} $. This last map is the matrix $T_n$ from Remark~\ref{JacobiTaylorrelation}. A linear form on $\Sym^n (\Omega_{R|K} )$ is the same as a linear form on $ \bigoplus_{ \mondeg {\lambda} = n} Re_\lambda $ annihilating $T_n$ from the left.
	
	We work with the commutative diagram
	\[\xymatrix{ \bigoplus\limits_{ \mondeg {\lambda} = n } R {e_\lambda }\ar[d] \ar[rr] &  & \bigoplus\limits_{ \mondeg {\lambda}  \leq n } R {e_\lambda }\ar[d] \\ \Sym^n(\Omega_{R|K} ) \ar[r] & \Delta^{n}/\Delta^{n+1} \ar[r] & \ModDif{n}{R}{K}  . }  \]
	A differential operator of order $\leq n$, considered as a linear form on $\ModDif{n}{R}{K}$ via the second row, induces a linear form on $\Sym^n(\Omega_{R|K} )$.
	If such a differential operator is given by a $\lambda$-tuple $\left( a_\lambda \right)$, $ \mondeg {\lambda} \leq n$, then both linear forms are given by sending $e_\lambda$ to $a_\lambda$. So the induced linear form is just given by the restricted tuple.
\end{proof}

%%%%%%%%%%%%%%%%%%%%%%%%%%%%%%%%%%%%%%%%%%%%%%%
\section{Differential powers and $D$-simplicity}\label{SecDiffPrimes}
%%%%%%%%%%%%%%%%%%%%%%%%%%%%%%%%%%%%%%%%%%%%%%
In this section we recall the definition of differential powers of ideals, and the related notion of $D$-ideals. These notions are essential to define the differential signature. We use these powers to give a criterion for the $D$-simplicity of $R$.

\subsection{$D$-ideals}

\begin{definition} We say that an ideal of $R$ is a \textit{$D_{R|A}$-ideal} if it is a $D_{R|A}$-submodule of $R$. We say that $R$ is \textit{$D_{R|A}$-simple} (or just \textit{$D$-simple}\index{D-simple} if no confusion is likely) if $R$ has no proper  {nonzero} $D_{R|A}$-ideals. Equivalently, $R$ is $D$-simple if it is simple as a $D_{R|A}$-module.
\end{definition}

We caution the reader that the property that the ring $D_{R|A}$ is a simple ring is also studied in the literature with similar nomenclature.

\begin{proposition}\label{PropDidealsLoc}
Suppose that $\ModDif{n}{R}{A}$ is finitely presented for all $n$.
Let $W\subseteq R$ be a multiplicative system. 
There is a natural bijection between \[\cA=\{I\subseteq R\ |\ I \ \text{is a  $D_{R|A}$-ideal and} \ I\cap W=\varnothing\}\]
 and 
\[\cB=\{ J\subseteq W^{-1}R \ |\  J \ \text{is a $D_{W^{-1}R | A}$-ideal}\,\}.\]
\end{proposition}
\begin{proof}
Let $\phi:\cA\to\cB$ given by $I\mapsto I \cdot W^{-1}R$.
Since $I$ is a $D_{R|A}$-ideal, $D_{R|A} I\subseteq I$. Then, 
 $W^{-1}R\otimes_R D_{R|A} I\subseteq I \cdot W^{-1} R$ by Proposition~\ref{localization2}.
Then, $\phi$ is well-defined. 
 
Let $\iota:R\to W^{-1}R$ denote the localization map, and  $\varphi:\cB\to \cA$ given by $J\mapsto \iota^{-1}(J).$ 
Let $\delta\in D_{R|A}$, and $f\in  \iota^{-1} (J)$. Then, $\delta \iota(f)=\iota(\delta f)\in J$ 
because $J\in \cB.$ As a consequence, $\delta f\in \iota^{-1}(J).$

Since $\varphi\circ \phi(I)=I$ and $\phi\circ\varphi(J)=J,$
we obtain the desired conclusion. 
\end{proof}

We use the previous proposition to obtain properties of $D$-ideals.

\begin{lemma}\label{PropMinimalPrime} 
Suppose that $\ModDif{n}{R}{A}$ is finitely presented for all $n$.
Every minimal primary component of a $D_{R|A}$-ideal is a $D_{R|A}$-ideal. In particular, the minimal primes of a radical $D_{R|A}$-ideal are $D_{R|A}$-ideals.
\end{lemma}
\begin{proof} 
 {
Let $I$ be a $D$-ideal of $R$ and $P$ a prime containing $I$. It follows from Proposition~\ref{PropDidealsLoc} that $I_P$ is a $D_{R_P|A}$-ideal, and that $I_P \cap R$ is a $D_{R|A}$-ideal. 
}
\end{proof}

\begin{remark}\label{rem-radicals-D-ideals}
It is not necessarily true that every minimal prime of a $D_{R|A}$-ideal is a $D_{R|A}$-ideal. For example, one can check that the $K$-linear endomorphism of $R=K[x]/(x^2)$ such that $\delta(1)=0$ and $\delta(x)=1$ is a $K$-linear differential operator of order $2$ (and of order 1 if $K$ has characteristic 2) and that $R$ is $D_{R|K}$-simple. Consequently, $(0)$ is a $D_{R|K}$-ideal but $\sqrt{(0)}=(x)$ is not. 
\end{remark}

We end this section with a property that  {relates} $D$-ideals of quotient with $D$-ideals of the original ring.

\begin{lemma}\label{lemma-D-ideal-from-quotient} Let $R$ be a ring and $A$ be a subring. Let $I\subseteq J$ be two ideals of $R$. Set $R'=R/I$, $A'=A/(A\cap I)$, and $J'$ to be the image of $J$ in $R'$. If $I$ is a $D_{R|A}$-ideal and $J'$ is a $D_{R' | A'}$-ideal, then $J$ is a $D_{R|A}$-ideal.
\end{lemma}
\begin{proof} If $I$ is a $D_{R|A}$-ideal, so that for every differential operator $\delta\in D^n_{R|A}$,  $\delta(I)\subseteq I$, then $\delta$ descends to a map on $R'$. It is immediate from the definitions that the map induced by $\delta$ on the quotient lies in $D^n_{R'|A'}$ so that there is a map $D_{R|A}\rightarrow D_{R'|A'}$ of filtered rings. If one also has that $J$ is not a $D_{R|A}$-ideal, there is some $g \in J$ and $\delta\in D_{R|A}$ such that $\delta(g)\notin J$. This then descends to a map $\bar{\delta}\in D_{R'|A'}$ with $\bar{\delta}(\bar{g})\notin J'$, so that $J'$ is not a $D_{R' | A'}$-ideal.
\end{proof}

\subsection{Differential powers and cores}
Motivated by Zariski's \cite{ZariskiHolFunct} study on symbolic powers for polynomial rings in characteristic zero, the differential powers were recently introduced \cite{SurveySP} to push this study to other rings.
  
\begin{definition}
Let $R$ be a ring and $A$ be a subring. Let $I$ be an ideal of $R$, and $n$ be a positive integer.
We define the \emph{$A$-linear $n$th differential powers of $I$}\index{differential power}\index{$I\dif{n}{A}$} by
$$
I\dif{n}{A}=\{f\in R \, | \, \delta(f)\in I \hbox{ for all } \delta\in D^{n-1}_{R|A}\}.
$$
\end{definition}

\begin{example}\label{diff-powers-regular} It follows from Example~\ref{example-regular-D} that $\m\dif{n}{K}=\m^n$ for 
an algebra with pseudocoefficent field $K$ that is smooth over $K$, 
$(R,\m,\kk)$. This follows from essentially the same argument  for polynomial rings \cite{SurveySP} applied in the more general setting of Example~\ref{example-regular-D}, but we reproduce it here for transparency. By Proposition~\ref{properties-diff-powers}~(ii) below, we have that $\m^n \subseteq \m\dif{n}{K}$. To see the other containment, let $x_1,\dots,x_d$ be a minimal generating set for $\m$, and pick $f\notin \m^n$. Write $f=g + \sum_{\alpha\in S} u_\alpha x^\alpha$, with $S$ a subset of $\NN^d$ consisting of elements with sum $<n$, $u_\alpha \notin \m$, and $g\in \m^n$. Since $f\notin \m^n$, $S$ is nonempty, so pick $\alpha\in S$ with $|\alpha|$ minimal. Then, in the notation of Example~\ref{example-regular-D}, the differential operator $D_\alpha$ has order $<n$ and is easily seen to satisfy $\partial^{\alpha}(f)\notin \m$, so $f\notin \m\dif{n}{K}$. Thus, $\m^n = \m\dif{n}{K}$.
\end{example}

If $f \notin \m\dif{n}{A}$, then there exists a $\delta \in D^{n-1}_{R|A}$ such that $\delta(f)$ is not in $\m$, hence it is a unit $u$. But then $u^{-1} \circ \delta$ is a differential operator of the same order with $(u^{-1} \circ \delta) (f)=1$.

We now recall a few properties of differential powers.

 {
\begin{proposition}[\cite{SurveySP}]\label{properties-diff-powers}
Let $R$ be a ring, $A$ be a subring,  $I,J_\alpha\subseteq R$ be ideals.
\begin{itemize}
\item[(i)] $I\dif{n}{A}$ is an ideal.
\item[(ii)] $I^{n}\subseteq I\dif{n}{A}$. 
\item[(iii)] $\left(\bigcap_{\alpha}J_\alpha\right)\dif{n}{A} 
=\bigcap_{\alpha }(J_\alpha)\dif{n}{A}$.
\item[(iv)] If $I$ is $\p$-primary, then $I\dif{n}{A}$ is also $\p$-primary .
\item[(v)] If $I$ is prime, then $I^{(n)} \subseteq I\dif{n}{A}$.
\end{itemize}
\end{proposition}
\begin{proof}
Parts (i)--(iv) are Proposition 2.4, Proposition 2.5, Exercise 2.13, and Propositions 2.6 of \cite{SurveySP}, respectively. The last part is a consequence of \mbox{(ii)} and~\mbox{(iv)}.
\end{proof}
}

%\begin{remark}\label{differential-filtrations-modules}
%One may also define various  differential filtrations on $R$-modules. For example, let $A$ be a subring of $R$ and $M,N$ be two $R$-modules. Let $N'$ be a submodule of $N$. One obtains a filtration on $M$ via $M_n=\{m \in M \ | \ \forall \delta\in D_{R|A}^{n-1}(M,N), \ \delta(m) \in N' \}$. We isolate the particular case where  $M$ is a module over a local ring $(R,\m)$ and we define $M\difM{n}{A}=M_n$ for $N=R$, $N'=\m$ in the setting of above.
%\end{remark}
We also note that if $A \subseteq B \subseteq R$, and $I$ is an ideal of $R$, then $I\dif{n}{A}\subseteq I\dif{n}{B}$ for all $n$.

Differential powers behave well with localization.

\begin{lemma}\label{diff-localize} 
 Let $W$ be a multiplicative set in $R$ and $I$ an ideal such that $W\cap I=\varnothing$. Suppose also that $\ModDif{n}{R}{A}$ is finitely presented for all $n$.	
 Then $I\dif{n}{A} = (W^{-1}I)\dif{n}{A} \cap R$.
\end{lemma}
\begin{proof}
$(\subseteq)$: It suffices to show that if $D^{n-1}_{R|A}\cdot f\subseteq I$, then $D^{n-1}_{W^{-1}R | A}\cdot(\frac{f}{1}) \subseteq W^{-1}I$. By Proposition~\ref{localization2}, for any $\delta\in D^{n-1}_{W^{-1}R | A}$, there exists $g\in W^{-1}R$ and $\eta\in D^{n-1}_{R|A}$ such that $\delta(\frac{f}{1})=g \frac{\eta(f)}{1}$ for all $f\in R$. The claim is then clear.

$(\supseteq)$: Suppose that $f\in R$, and $D^{n-1}_{W^{-1}R | A}\cdot(\frac{f}{1}) \subseteq W^{-1}I$. If $\delta\in D^{n-1}_{R|A}$, then it extends to a differential operator $\tilde{\delta}\in D^{n-1}_{W^{-1}R | A}$ such that $\tilde{\delta}(\frac{f}{1})=\frac{\delta(f)}{1}$. By hypothesis, this element is in $W^{-1}I \cap R =I$. Thus, $f$ lies in $I\dif{n}{A}$.
\end{proof}

\begin{lemma}\label{diff-localize2} Let $W$ be a multiplicative set in $R$ and $I$ an ideal. Suppose also that $\ModDif{n}{R}{A}$ is finitely presented for all $n$. Then $I\dif{n}{A} (W^{-1}R) = (W^{-1} I)\dif{n}{A}$.
\end{lemma}
\begin{proof}
 {For any $J\in W^{-1}R$, we have $J=(J \cap R) W^{-1}R$, so this follows from Lemma~\ref{diff-localize}.}
% $(\subseteq)$: We proceed  by induction on $n$. The case $n=1$ is trivial. Let $r\in I\dif{n}{A}$, $w\in W$, and $\delta\in D^{n-1}_{W^{-1}R | A}$. Then {\cb $\delta(\frac{r}{w})=\frac{1}{w}(\delta(r)-[\delta,w](\frac{r}{w}))$}. By the induction hypothesis, {\cb $[\delta,w](\frac{r}{w})\in W^{-1}I$}. Since we can write $\delta=\frac{1}{v} \otimes \delta'$ with $v\in W$ and $\delta'\in D^{n-1}_{R |A}$ by Proposition~\ref{localization2}, we also have that $\delta(r)\in W^{-1}I$ by hypothesis.

%$(\supseteq)$: By Proposition~\ref{localization2}, we know that $D^{n-1}_{R|A} \cdot a \subseteq D^{n-1}_{W^{-1}R | A} \cdot a$. But then $D^{n-1}_{W^{-1}R | A} \cdot a = D^{n-1}_{W^{-1}R | A} \cdot (\frac{as}{s}) = D^{n-1}_{W^{-1}R | A} \circ s \cdot (\frac{a}{s}) \subseteq D^{n-1}_{W^{-1}R | A} \cdot (\frac{a}{s}) $, since multiplication by an element does not increase the order of a differential operator.
\end{proof}

As a consequence of Proposition~\ref{diff-ops-completion}, differential powers of the maximal ideal commute with completion.

	\begin{lemma}\label{diff-powers-completion} Let  $(R,\m,\kk)$ be an algebra with pseudocoefficent field $K$. Then $\m\dif{n}{K}\widehat{R} = (\m\widehat{R})\dif{n}{K}$. Consequently, if $(S,\n,L)$ is a power series ring over $L$, then $\n\dif{n}{L}=\n^n$ for all $n$.
	\end{lemma}

\begin{proof} We first establish the equality $\m\dif{n}{K}\widehat{R} = (\m\widehat{R})\dif{n}{K}$. 

$(\subseteq)$:
	 Let $f\in R$, and assume $D^{n-1}_{R|K}\cdot f\subseteq \m$. If $\delta\in D^{n-1}_{\widehat{R}|K}$, then by Proposition~\ref{diff-ops-completion}, there is some $r\in \widehat{R}$ and $\eta\in D^{n-1}_{R|K}$ such that $\delta(f)=r\eta(f)$, which by hypothesis, lies in $\m\widehat{R}$.
	
$(\supseteq)$:
 Since $(\m\widehat{R})\dif{n}{K}$ is $\m$-primary, and every $\m$-primary ideal of $\widehat{R}$ is expanded from $R$, it suffices to show that if $f\in R$ and $D_{\widehat{R}|K}^{n-1}\cdot f \subseteq \m\widehat{R}$, then $D^{n-1}_{R|K}\cdot f\subseteq \m$. This is clear since every element of $D^{n-1}_{R|K}$ extends to $D_{\widehat{R}|K}^{n-1}$.

Now, suppose that $S$ is a power series ring over $L$. Write $S=L \llbracket x_1, \dots, x_d \rrbracket$ and $\n=(x_1,\dots, x_d)$. Set $R=L[x_1, \dots, x_d]_{\m}$, with $\m=(x_1,\dots, x_d)$. Since $S=\widehat{R}$ and $\n=\m\widehat{R}$, the second assertion of the Lemma follows from the first.
\end{proof}

We now introduce a differential version of the splitting prime for $F$-pure rings. However, this notion is valid in any characteristic. 

\begin{definition}
Let $(R,\m)$ be a local ring, $A$ be a subring, and let $J \subseteq R$ be an ideal.
We define the \emph{$A$-differential core of $J$}\index{differential core}\index{$\cP_A(J)$} by
$$
\cP_A(J) =\bigcap_{n\in\NN}J\dif{n}{A}.
$$
The \emph{$A$-differential core of $R$} is $\cP_A= \cP_A(\m)$.
\end{definition}

\begin{lemma}\label{LemmaDidealDifPower}
Let $R$ be a ring and $A$ be a subring.
Then,  $I$ is a $D_{R|A}$-ideal if and only if $I\dif{n}{A}=I$ for every integer $n\geq 1.$
\end{lemma}
\begin{proof}
We first show that if $I$ is a $D_{R|A}$-ideal, then $I\dif{n}{A}=I$ for every integer $n\geq 1.$ We note that  $I\dif{n}{A}\subseteq I$  because  if $f\in I\dif{n}{A},$ then $1\cdot f\in I$ as $1\in D^{0}_{R|A}.$
We now show the other containment.  If $f\in I $, then $D_{R|A} f\subseteq I$ as $I$ is an  $D_{R|A}$-ideal. Then, in particular, $D^{n-1}_{R|A} f\subseteq I$, and so, $f\in I\dif{n}{A}$. 

We now focus on the other implication. We suppose that  $I\dif{n}{A}=I$ for every integer $n\geq 1.$ Let $f\in I.$ Then, $f\in I\dif{n}{A}$ for every $n\in\NN$,  and so  $D^{n-1}_{R|A} f\subseteq I$ for every $n\in\NN.$  Hence,  $D_{R|A} f\subseteq I$, and so,   $D_{R|A} I\subseteq I$.
\end{proof}

\begin{remark}
 The previous proposition is not true if one replaces  the condition ``$I\dif{n}{A}=I$ for every integer $n\geq 1$'' with ``$I\dif{n}{A}=I$ for some integer $n>1.$'' For example, let $R=K[x^2,xy,y^2]$. Then, $D^1_{R|K}=R \cdot \langle 1, x \frac{\partial}{\partial x}, y \frac{\partial}{\partial x}, x \frac{\partial}{\partial y}, y \frac{\partial }{\partial y}\rangle$, and $D^1_{R|K}(\m)=\m$, so $\m\dif{2}{K}=\m$. However,  {$\frac{1}{2} \frac{\partial^2}{\partial x^2} \in D^2_{R|K}$,} so $x^2 \notin \m\dif{3}{K}$.
\end{remark}

We summarize a few properties of differential cores. In particular, we characterize $D$-simplicity using differential cores in Corollary~\ref{CorDifPrimeDsimple}.

\begin{proposition}\label{PropDiffPrime}
Let $(R,\m)$ be a local ring, $J$ an ideal of $R$, and $A$ be a subring.
Then, 
\begin{enumerate}
\item $\cP_A(J)$ is a $D_{R|A}$-ideal.
\item\label{cPAJ-contains} $\cP_A(J)$ contains every  $D_{R|A}$-ideal of $R$ contained in $J$.
\item\label{diff-prime-primary} $\cP_A(J)$ is a primary ideal if $J$ is prime. In particular, $\cP_A$ is primary.
\item $R/\cP_A$ is $D$-simple.
\end{enumerate}
\end{proposition}
\begin{proof}
We proceed by parts.
\begin{enumerate}
\item Let  {$f\in \cP_A(J)$}, and $\delta\in D^n_{R|A}.$
For every $\partial \in  D^m_{R|A},$  $\partial \delta\in D^{m+n}_{R|A}.$
Since $f\in J\dif{m+n}{A}$ for every $m\in\NN$, we have that $\partial\delta\cdot f\in J.$
Then, $\delta f\in J\dif{m}{A}$ for every $m\in\NN.$ Hence,  $\delta \cdot f\in \cP_A(J)$.
\item Let $I$ be a  $D_{R|A}$-ideal with $I\subseteq J$. We have that $I\dif{n}{A}\subseteq J\dif{n}{A}.$ Then,  
$$
I=\bigcap_{n\in\NN} I\dif{n}{A}\subseteq \bigcap_{n\in\NN} J\dif{n}{A}=\cP_A(J),
$$
where the first equality follows from Lemma~\ref{LemmaDidealDifPower}.
\item 
Let $\p$ be a minimal prime of $\cP_A(J)$. Since $J$ is prime, and $\cP_A(J) \subseteq J$, we have that $\p \subseteq J$. Let $\q$ be the $\p$-primary component of $\cP_A(J)$. We can write $\cP_A(J)$ as the intersection of $\q$ with some other primary ideals; in particular $\cP_A(J) \subseteq \q \subseteq J$. By Lemma~\ref{PropMinimalPrime}, $\q$ is a $D_{R|A}$-ideal, and by Part~(\ref{cPAJ-contains}), $\q \subseteq \cP_A(J)$. Thus, $\q = \cP_A(J)$.
\item Suppose on the contrary that there is an ideal of $R/\cP_A$ that is stable under its differential operators. By Lemma~\ref{lemma-D-ideal-from-quotient}, there would then exist a proper $D_{R|A}$-ideal contatining $\cP_A$, which would contradict Part~(\ref{cPAJ-contains}).\qedhere
\end{enumerate}
\end{proof}
\begin{corollary}\label{CorDifPrimeDsimple}
Let $(R,\m)$ be a local ring, and $A$ be a subring. Then,
 $R$ is a simple  $D_{R|A}$-module if and only if $\cP_A=0$.
\end{corollary}
\begin{proof}
This follows immediately from Part~(\ref{cPAJ-contains}) of Proposition~\ref{PropDiffPrime}.
\end{proof}

Thus, $D$-simplicity means that for all $f \neq 0$ in $R$ there exists a differential operator $\delta$ such that $\delta(f)$ is a unit $u$. By taking $u^{-1} \delta$, one also finds an operator sending $f$ to $1$.

An ongoing line of research is the comparison of $D_{R|K}$ and the ring generated by the derivations for finitely generated $K$-algebras. We now show that these algebras must differ for rings that are  $D_{R|K}$-simple, but not regular.

\begin{remark}\label{rem:der-simple}
	Let $K$ be a field of characteristic zero, and $R$ be essentially of finite type over $K$. We can consider the subalgebra $\mathscr{D}\subseteq D_{R|K}$ generated by $R$ and the $K$-linear derivations on $R$. The minimal primes of the singular locus of $R$ are stable under each derivation of $R$ \cite[Theorem~5]{Seidenberg}, hence are stable under the action of $\mathscr{D}$. It follows that if $R$ is $D$-simple and $\mathscr{D}= D_{R|K}$, then $R$ is regular. This is a special case of Nakai's conjecture that generalizes other known cases  \cite{Ishibashi}. This approach to Nakai's conjecture is employed in the work of Traves \cite{Traves}.
\end{remark}

\subsection{A Fedder/Glassbrenner-type criterion for $D$-simplicity}\label{FGCriterion}
In this subsection we introduce a similar  criterion for $D$-simplicity, which is
motivated by Glassbrenner's Criterion \cite{Glassbrenner} for strong $F$-regularity,

\begin{lemma}[{\cite[Lemma 1.6]{FedderFputityFsing}}]\label{Fedder} Let $A\subseteq B$ be Gorenstein rings, and suppose that $B$ is a finitely generated free $A$-module.
\begin{enumerate} 
\item[(i)] $\Hom_A(B,A) \cong B$ as $B$-modules.
\item[(ii)] If $\Phi$ is a generator for $\Hom_A(B,A)$ as a $B$-module, $\a \subseteq A$ and $\b \subseteq B$ are ideals, and $x\in B$, then $(x \Phi)(\b) \subseteq \a$ if and only if $x \in (\a B : \b)$.
\end{enumerate}
\end{lemma}

\begin{setup}\label{fin-type}
Let $K$ be a field, $K[x_1,\dots,x_d]$ be a polynomial ring, and $\m=(x_1,\dots,x_d)$.
Let $R=S/I$, where $S=K[x_1,\dots,x_d]_{\m}$.   {Then $\ModDif{}{S}{K} = W^{-1}K[x_1,\dots,x_d,\tilde{x}_1,\dots,\tilde{x}_d]$, where $W$ is the multiplicative set given by the product of $R\setminus \m$, and the corresponding set when $x_1,\ldots,x_d$ are replaced for $\tilde{x}_1,\ldots,\tilde{x}_d$}. Set $\Delta^{[n]}_{S|K}=((x_1-\tilde{x}_1)^n,\dots,(x_d-\tilde{x}_d)^n)$, and $\ModDif{[n]}{S}{K}=\ModDif{}{S}{K}/\Delta^{[n]}_{S|K}$. Then $\ModDif{}{R}{K}$ is naturally a quotient of $\ModDif{}{S}{K}$, and we write $\Delta^{[n]}_{R|K}, \ModDif{[n]}{R}{K}$ for the image of $\Delta^{[n]}_{S|K}, \ModDif{[n]}{S}{K}$ under this quotient map. By abuse of notation, we use $d$ for the universal differential in various of these settings.
\end{setup}

In the context of Setup~\ref{fin-type}, for an ideal ${J}$ of $R$ we define
\[J^{\Fdif{n}} := \{ r \in R \ | \ \delta(r)\in I \ \text{for all} \ \delta\in D^{[n]}_{R|K}\} \]\index{$I^{\Fdif{n}}$}
where $D^{[n]}_{R|K}$ is the set of $\Delta^{[n]}_{R|K}$-differential operators of $R$. This definition depends not only on ${J}$, but also on the presentation of $R$. We  revisit this definition in Section~\ref{five}.

\begin{proposition}\label{fedformula} In the context of Setup~\ref{fin-type}, let $K$ be a field and let $J$ be an ideal of $R$, and $J'$ be the preimage of $J$ in $S$. Then $$J^{\Fdif{n}}=\IM\left(\displaystyle \big( d(J') \ModDif{}{S}{K} + \Delta^{[n]}_{S|K} \big)  :_{\ModDif{}{S}{K}} \Big( \big( d(I)  \ModDif{}{S}{K} + \Delta^{[n]}_{S|K} \big) :_{\ModDif{}{S}{K}}   I  \ModDif{}{S}{K}   \Big)\right)$$
and
$$J\dif{n}{K}=\IM\left({\displaystyle \big( d(J') \ModDif{}{S}{K} + \Delta^{[n]}_{S|K} \big)  :_{\ModDif{}{S}{K}} \Big( \big( d(I)  \ModDif{}{S}{K} + \Delta^{[n]}_{S|K} \big) :_{\ModDif{}{S}{K}}   \big( I  \ModDif{}{S}{K} + \Delta^n_{S|K} \big)  \Big)}\right),$$ where the images are taken  in $R$.
\end{proposition}
\begin{proof}
For the first part, by Proposition~\ref{representing-differential}, we have that $\delta(f)\in J$ for all $\delta\in D^{[n]}_{R|K}$  if and only if  for every ${\phi\in \Hom_R( \ModDif{[n]}{R}{K}, R)}$, we have $\phi( {d(f)})\in J$. 
We can write 
$${\ModDif{[n]}{R}{K} = \ModDif{[n]}{S}{K} / \big( I\ModDif{[n]}{S}{K} + d(I)\ModDif{[n]}{S}{K}\big)}.$$ Since $\ModDif{[n]}{S}{K}$ is Gorenstein and free over $S$, Lemma~\ref{Fedder} applies. If $\Phi$ is a generator for $\Hom_S(\ModDif{[n]}{S}{K},S)$, then \[(r\Phi)\big( I\ModDif{[n]}{S}{K} + d(I)\ModDif{[n]}{S}{K}\big) \subseteq I\]
 if and only if 
\[ r \in \big(  I\ModDif{[n]}{S}{K} :_{\ModDif{[n]}{S}{K}} \big( I\ModDif{[n]}{S}{K} + d(I)\ModDif{[n]}{S}{K}\big) \big)=: W,\]
 so $\Hom_R( \ModDif{[n]}{R}{K}, R)$ consists of images of maps $r \Phi$, with $r \in W$. Then $\overline{(r \Phi)}(\overline{d(f)})\in J$ if and only if $\Phi(r \cdot d(f) ) \in J' $. 

Thus, $f\in J^{\Fdif{n}}$  is equivalent to $\Phi(d(f) \cdot  W)\subseteq J'$, which in turn is equivalent to $(d(f) \cdot \Phi)(W) \subseteq J'$. Applying Lemma~\ref{Fedder} again, $d(f)$ satisfies this condition if and only if $d(f) \in \big( J' \ModDif{[n]}{S}{K}:_{\ModDif{[n]}{S}{K}} W \big)$. By taking preimages in $\ModDif{}{S}{K}$, and simplifying, this occurs if and only if
\[ d(f) \cdot \Big( \big( I  \ModDif{}{S}{K} + \Delta^{[n]}_{S|K} \big) :_{\ModDif{}{S}{K}}   d(I)  \ModDif{}{S}{K}   \Big) \not\subseteq J' \ModDif{}{S}{K} + \Delta^{[n]}_{S|K}. \]
By switching the left inclusion $S \rightarrow \ModDif{}{S}{K}$ and the right inclusion $d: S \rightarrow \ModDif{}{S}{K}$, one obtains the statement of the theorem.

The proof of the second part proceeds similarly. We note that $\Delta^{[n]}_{S|K}$ is contained in $\Delta^{n}_{S|K}$, so every element of $H:=\Hom_S(\ModDif{n}{S}{K},S)$ is the image of some element of $\Hom_S(\ModDif{[n]}{S}{K},S)$. In particular, if $\Phi$ is a generator of the latter module, as in the proof of the first part,
\[(r\Phi)\big( I\ModDif{[n]}{S}{K} + d(I)\ModDif{[n]}{S}{K} + \Delta^n_{S|K}\ModDif{[n]}{S}{K} \big) \subseteq I\]
 if and only if 
\[ r \in \big(  I\ModDif{[n]}{S}{K} :_{\ModDif{[n]}{S}{K}} \big( I\ModDif{[n]}{S}{K} + d(I)\ModDif{[n]}{S}{K} + \Delta^n_{S|K}\ModDif{[n]}{S}{K}\big) \big).\]
The rest is analogous to the previous part.
\end{proof}

 {
We are now ready to state the  criterion for $D$-simplicity.
}

\begin{theorem}\label{criterion}  In the context of Setup~\ref{fin-type},  $R$ is $D_{R|K}$-simple if and only if for every $f \in S$ whose image is nonzero in $R$, there is some $n$ such that
\[ f \cdot \Big( \big( d(I)  \ModDif{}{S}{K} + \Delta^{[n]}_{S|K} \big) :_{\ModDif{}{S}{K}}   I  \ModDif{}{S}{K}   \Big) \not\subseteq \m^{[n]} \ModDif{}{S}{K} + d(\m)  \ModDif{}{S}{K} \quad \text{in}  \  \ModDif{}{S}{K}. \]
\end{theorem}
\begin{proof} We note first that $R$ is $D_{R|K}$-simple if and only if for every nonzero element $r$ of $R$, there is a differential operator $\delta$ such that $\delta(r)=1$; equivalently, $\delta(r)\notin \m$.
By Lemma~\ref{I-diff-ops}, $\delta\in \Hom_K(R,R)$ is an element of $D_{R|K}$ if and only if $\delta$ is a $\Delta^{[n]}_{R|K}$-differential operator for some $n$. Thus, $R$ is $D_{R|K}$-simple if and only if for every nonzero $f\in R$, there is some $n$ such that $f\notin \m R^{\Fdif{n}}$. The theorem then follows from Proposition~\ref{fedformula} and the observation that $d(\m)\ModDif{}{S}{K} + \Delta^{[n]}_{S|K} = \m^{[n]} \ModDif{}{S}{K} + d(\m)$.
\end{proof}

%%%%%%%%%%%%%%%%%%%%%%%%%%%%%%%%%%%%%%%%%%%%%%%%
\section{Differential signature}
%%%%%%%%%%%%%%%%%%%%%%%%%%%%%%%%%%%%%%%%%%%%%%%

In this section, we introduce our main object of study. As noted in the introduction, there are multiple definitions for the differential signature, that all agree in the case of an algebra with a pseudocoefficient field or complete local ring; each of these definitions provides different insights into this limit. Our first goal below is to establish the equivalence of the definitions. We then proceed to collect some of the basic properties of differential signature.

%%%%%%%%%%%%%%%%%%%%%%%%%%%%%%%%%%%%%%%
\subsection{Differential signature}

\begin{definition}
Let $(R,\m)$ be a local ring, $A$ be a subring, and let $d=\dim(R)$.
We define the \textit{differential signature}\index{differential signature}\index{$\dm{A}{R}$} of $R$ by
\[
\dm{A}(R)=\limsup\limits_{n\to\infty}\frac{\lambda_R(R/\m\dif{n}{A})}{n^d / d!} .
\]
\end{definition}

\begin{example}\label{reg-1}
Let $K$ be a field, and $(R,\m)$ be a graded or local ring of dimension $d$, essentially of finite type and smooth over $K$. Then, by Example~\ref{diff-powers-regular}, $\m\dif{n}{K}=\m^n$ for all $n>0$. Now, we compute
\[ \dm{K}(R)=\limsup_{n\rightarrow\infty}\frac{\lambda_R(R/\m\dif{n}{K})}{n^d/d!} = \limsup_{n\rightarrow\infty}\frac{\lambda_R(R/\m^n)}{n^d/d!} =e(R)= 1.\]
\end{example}

\begin{example}\label{cubic}
Let $R=\CC[x,y,z]/(x^3+y^3+z^3)$. Then $\m\dif{n}{\CC}=\m$ for all $n>0$. Indeed, one always has that $\m\dif{n}{\CC} \subseteq \m$  for each $n$. If $f\in \m$ is homogeneous, and $\delta\in D^{n-1}_{R|\CC}$ is homogenous, $\deg(\delta(f)) > \deg(f)>0$, by Example~\ref{example-BGG}, so $\delta(f)\in \m$. Again by Example~\ref{example-BGG}, $D_{R|\CC}$ is graded, so the previous computation implies that $\m\subseteq \m\dif{n}{\CC}$ as well.

Now, we compute
\[ \dm{\CC}(R)=\limsup_{n\rightarrow\infty}\frac{\lambda_R(R/\m\dif{n}{\CC})}{n^2/2!} = \limsup_{n\rightarrow\infty}\frac{\lambda_R(R/\m)}{n^2/2!}= \limsup_{n\rightarrow\infty}\frac{1}{n^2/2} =0.\]
\end{example}

In Theorem~\ref{Possiganeg},  we generalize this example to show that for cones over smooth curves of genus $\geq 1$ the differential signature is always $0$.

%%%%%%%%%%%%%%%%%%%%%%%%%%%%%%%%%%%%%%%%%%%%
\subsection{Principal parts signature}

We now proceed to define a signature in terms of the free ranks of the modules of principal parts. We recall that the \emph{free rank}\index{free rank} of a module $M$ is the maximal rank $a$ of a free summand in any direct sum decomposition $M=R^{a}\oplus M'$.

\begin{definition}
	Let $A\subseteq R$ be a map of rings that is essentially of finite type. The \emph{principal parts signature}\index{principal parts signature}\index{$\pps{K}(R)$} of $R$ over $A$ is
	\[\pps{K}(R):=\limsup_{n \rightarrow \infty} \frac{\frk_R(\ModDif{n}{R}{A})}{\mathrm{rank}_R(\ModDif{n}{R}{A})}.\]
\end{definition}

The rank of $\ModDif{n}{R}{A}$ may not always be defined; we say that the principal part signature is not defined in this case. Of course, this is not an issue when $R$ is a domain.

To work with this definition, we use the following two characterizations of free rank.

\begin{lemma}\label{freerankinterpretation}\label{freerank-conditions} Let $(R,\m,\kk)$ be local or graded and $M$ be a finitely generated $R$-module. 
	\begin{enumerate}
		\item
		Define a submodule
		\[ \NF{M} := \{ m \in M \ | \ \forall \phi \in \Hom_R(M,R), \ \phi(m) \in \m \}. \]\index{$\NF{M}$}
		Then, one has that $\mathrm{freerank}_R(M)=\lambda_R(M/\NF{M}) =\dim_{\kk} (M/\NF{M}) $.
		
		\item
		Consider the short exact sequence of $R$-modules,
		\[ 0 \longrightarrow \operatorname{Hom}_R(M, {\mathfrak m} ) \longrightarrow  \operatorname{Hom}_R(M,R )  \longrightarrow  Q \longrightarrow 0  . \]
		Then the free rank of $M$ is the same as the $\kk$-dimension of the quotient $Q$.
	\end{enumerate}
\end{lemma}
\begin{proof}
	The first part is well known {\cite[Discussion~6.7]{CraigSurvey}}. We now focus on the second part. The quotient $Q$ is a module over $\kk$. Let $M = F \oplus N$ with a free  module $F \cong R^s$. We have
	\[ \operatorname{Hom}_{ R } \left( M , \m \right) 	\cong \operatorname{Hom}_{ R } \left( F , \m \right) \oplus \operatorname{Hom}_{ R } \left( N , \m \right)  \]
	and
	\[ \operatorname{Hom}_{ R } \left( M , R \right) 		 \cong \operatorname{Hom}_{ R } \left( F , R \right) \oplus \operatorname{Hom}_{ R } \left( N , R \right)  .\]
	The quotient is
	\[	 \begin{aligned}
	Q & = \left( \operatorname{Hom}_{ R } \left( F , R \right) \oplus \operatorname{Hom}_{ R } \left( N , R \right) \right)/ \left( \operatorname{Hom}_{ R } \left( F , {\mathfrak m} \right) \oplus \operatorname{Hom}_{ R } \left( N , {\mathfrak m} \right) \right)  \\
	& \cong \operatorname{Hom}_{ R } \left( F , R \right) / \operatorname{Hom}_{ R } \left( F , {\mathfrak m} \right) \oplus \operatorname{Hom}_{ R } \left( N , R \right) /\operatorname{Hom}_{ R } \left( N , {\mathfrak m} \right)  \\
	& \cong  R^s/ {\mathfrak m}^{\oplus s} \oplus \operatorname{Hom}_{ R } \left( N , R \right) /\operatorname{Hom}_{ R } \left( N , {\mathfrak m} \right)  \\
	& \cong \kk^s \oplus \operatorname{Hom}_{ R } \left( N , R \right) /\operatorname{Hom}_{ R } \left( N , {\mathfrak m} \right).
	\end{aligned}
	\]
	So the $\kk$-dimension of $Q$ is at least $s$. If $F$ is a submodule where the free rank of $M$ is attained, then $N$ does not have a nontrivial free summand and there is no surjective homomorphism from $N$ to $R$. Hence
	$\operatorname{Hom}_{ R } \left( N , R \right)  =\operatorname{Hom}_{ R } \left( N , \m \right)$ and the right summand is $0$.
\end{proof}

\begin{definition}
	A \emph{unitary differential operator}\index{unitary differential operator} on a local or graded ring $(R,\m)$ is a differential operator $\delta :R \rightarrow R$, such that the image of $\delta$ is not contained in $\m$ {; in the graded case, we assume $\delta$ is graded.}
\end{definition}

 {Suppose that $K\subseteq R$.}
If $R$ contains only the constant units of the base field, this means that the partial differential equation $\delta (f)=1$ has a solution $f \in R$ for some  {$\delta \in D_{R\vert K}$}. For a local  {or graded} ring $R$ and a differential operator $\delta$ on $R$ we  denote in the following the induced differential operator with values in $R/{\mathfrak m}$ by $\delta'$.

\begin{lemma}\label{unitary}
	Let $(R,\m,\kk)$ be a local  {or graded} ring containing a  {coefficient field $K$, so that $K\cong \kk$}. Let $\delta$ be a $K$-linear differential operator from $R$ to $R$ of order at most $n$;  {in the graded case, we assume that $\delta$ is graded}. Let $\delta'$ be the induced operator $R \xrightarrow{\delta'} R \twoheadrightarrow R/\m=\kk$. Then the following are equivalent.
	\begin{enumerate}
\item The $R$-linear map $\delta :P^n_{R|K} \rightarrow R$ is surjective. 
\item There is a unit ( {of degree zero in the graded case}) inside the image of the differential operator $\delta$.
\item  $\delta'$ is surjective.
\item There exists a function $f \in R$ such that $\delta'(f)=1$.
\item  {$\delta$ is unitary.}
\end{enumerate}
\end{lemma}
\begin{proof}
	The equivalence of (2), (3), (4),  {and (5)} is clear. If (2) holds, let $f \in R$ with $\delta(f)=u \notin {\mathfrak m}$. Let $d^n(f) \in \ModDif{n}{R}{K}$ be the image of $f$ under the $n$th universal differential map. Then $\delta (d^n(f))$ is the unit $u$ and since $\delta$ is an $R$-linear form on the module of principal parts it must be surjective. Suppose now that (2) does not hold. Then the image of the differential operator $\delta$ is inside the maximal ideal $\mathfrak m$ of $R$. Since $\ModDif{n}{R}{K}$ is generated as an $R$-module by the images $d^n(f)$, $ f \in R$ \cite[Proposition~16.3.8]{EGAIV}, also the image of $\delta$ considered as a linear form on $\ModDif{n}{R}{K}$ lies inside the maximal ideal, and (1) does not hold.
\end{proof}

Note that there are many surjective differential operators from $R$ to $K$: For example, the $K$-valued derivation space $\operatorname{Der}_R(R,K)$ is under certain conditions just the tangent space. But such an operator is in general not a unitary differential operator, for which we require that it comes from an operator with values in $R$.

\begin{lemma}\label{unitarysystem}
	Let $(R,\m,\kk)$ be a local ring containing a field $K$. Let $\delta_1 , \ldots, \delta_t$ be differential operators from $R$ to $R$ of order at most $n$. Let $\delta'_i$ be the induced operators $R \xrightarrow{\delta_i} R \twoheadrightarrow R/\m=\kk$. Write $\delta_i=\phi_i \circ d^n$. Then the following are equivalent.
\begin{enumerate}
	\item The $R$-linear map $\Phi=(\phi_1, \ldots, \phi_t) :\ModDif{n}{R}{K} \rightarrow R^t$ is surjective.
\item The $\delta'_i$ are linearly independent over $\kk$, where the vector space structure is given by postmultiplication.
\end{enumerate}
\end{lemma}
\begin{proof}
We consider the $R$-linear map
	\[  \Phi':\ModDif{n}{R}{K} \xrightarrow{(\phi_1, \ldots, \phi_t)} R^t  \xrightarrow{\pi} \kk^t  .   \]
	
	If (1) holds, then $\Phi$ is surjective and hence also $\Phi'$ is surjective. The maps $\pi \circ \phi_i$ factor through $\kk$-linear maps $\phi'_i:\ModDif{n}{R}{K}\otimes_R \kk \rightarrow \kk$. The map $(\phi'_1,\dots,\phi'_t)$ is surjective, so the component maps are linearly independent. If the differential operators $\delta'_i$ were linearly dependent, since the image of $d^n$ generates $\ModDif{n}{R}{K}$ as an $R$-module, this would contradict the linear independence of the maps $\phi'_i$.
	
	Conversely, if (2) holds, then $\Phi'$ is surjective. By Nakayama's Lemma, $\Phi$ is surjective as well.
\end{proof}

\begin{definition}
	A set of differential operators forms an \emph{independent system of  unitary operators}\index{independent system of unitary operators} if it satisfies the equivalent conditions of the previous lemma.
\end{definition}

\begin{example}
	The differential operators $\frac{\partial}{\partial y}$ and $ x \frac{\partial}{\partial x} + \frac{\partial}{\partial y}$ on $K[x,y]_{(x,y)}$ show that the properties from Lemma~\ref{unitarysystem} are not equivalent to the property that the differential operators themeselves are $K$-linearly independent, even if they both are unitary. The two operators are independent over $K$, but as derivations to $K$ they are the same.
\end{example}

\begin{remark}
	\label{localunitaryoperators}
	We have the short exact sequences of $R$-modules
	\[ \xymatrix{ 0 \ar[r] & \Hom_R(\ModDif{n}{R}{K},\m) \ar[r] \ar[d]^{\cong} &   \Hom_R(\ModDif{n}{R}{K},R) \ar[r] \ar[d]^{\cong} & Q_n \ar[r] \ar[d]^{\cong} & 0 \\
		0 \ar[r] & D^n_{R|K}(R,\m) \ar[r] & D^n_{R|K} \ar[r] &  Q_n \ar[r] & 0 } \]
	which are identical by the universal property of the module of principal parts, and where the $K$-dimension of the quotient $Q_n$\index{$Q_n$} is the number of linearly independent unitary operators, which equals the free rank of $ P^n_{R|K}$ by Lemma~\ref{freerankinterpretation} (2).
\end{remark}

\begin{corollary}
	Let $(R,\m,\kk)$ be a local ring containing a field $K$. Then the free rank of $\ModDif{n}{R}{K}$ is equal to
	 the maximal number of independent unitary operators of order at most $n$.
\end{corollary}
\begin{proof}
This is clear from Lemma~\ref{unitarysystem}.
\end{proof}

\begin{lemma}
	\label{ppsignaturelocalize}
	Let $R$ be a domain essentially of finite type over a field $K$ and let $W \subseteq R$ be a multiplicative subset. Then we have the inequality 
	\[\pps{K}(R)  \leq \pps{K}(W^{-1} R)  . \]
	This holds in particular for a localization of a local $K$-algebra essentially of finite type.
\end{lemma}
\begin{proof}
By Proposition~\ref{diffmod-localize}, we have $P^n_{W^{-1}R|K} \cong W^{-1} P^n_{R|K} $.  Free ranks can only increase by localizing, and the rank is preserved under localization, since it coincides with generic rank.
\end{proof}

We also have the following weak equisingularity statement.
\begin{lemma}
	Let $R$ be a domain of finite type over a field $K$ and let $\p \subseteq R$ be a prime ideal. Then we have the equality 
	\[\pps{K}(R_\m) = \pps{K}( R_\p)  \,  \]
	for a very general maximal ideal $\m \in V(\p)$; i.e., there exists a countable union of closed subsets of $V(\p)$ of smaller dimension such that all maximal ideals outside of this set have this property.
	\end{lemma}
\begin{proof}
	For each $n$ we have a decomposition \[ P^n_{ R_{\p }|K } \cong (P^n_{R|K})_\p \cong R_p^{r_n} \oplus M ,\]
	where $r_n$ is the free rank of $P^n_{ R_\p|K} $. Since everything is finitely generated, there exists  $f_n \notin \p$ such that also $ P^n_{ R_{f_n}|K} \cong R_{f_n}^{r_n} \oplus N$ holds. Then $V(\p) \cap \bigcup_{n \in {\mathbb N } } V(f_n)$ describes the exceptional locus.
	\end{proof}

\subsection{Comparison of the two signatures}

We now proceed to relate the differential power signature and the principal parts signature.

\begin{proposition}\label{freerank}
	Let $(R,\m,\kk)$ be a local  {or graded} ring with dimension $d$.
	\begin{enumerate}
		\item If $(R,\m,\kk)$ is an algebra with pseudocoefficient field $K$, then
		\[ \lambda_R(R/\m\dif{n+1}{K}) =  \frk_R(\ModDif{n}{R}{K}). \]
		\item If $(R,\m,\kk)$ is a complete local ring with coefficient field $K\cong \kk$, then
		\[ \lambda_R(R/\m\dif{n+1}{K}) =  \frk_R(\wModDif{n}{R}{K})=\frk_R(\ModDif{n}{R}{K}).\]
	\end{enumerate}
\end{proposition}
\begin{proof} First, we consider Case~(1). Let 
\[\NF{\ModDif{n}{R}{K}}=\{ p \in \ModDif{n}{R}{K} \ | \ \phi(p)\in \m \quad \text{for all} \ \ \phi\in \Hom_R(\ModDif{n}{R}{K},R)\}\]
where we regard $\ModDif{n}{R}{K}$ as an $R$-module via the left factor. By Lemma~\ref{freerank-conditions} (1), 
\[{\frk_R(\ModDif{n}{R}{K})=\lambda_R(\ModDif{n}{R}{K} / \NF{\ModDif{n}{R}{K}})}.\] We analyze the map  {of $K$-vector spaces} $\bar{d}^n:R/\m\dif{n+1}{K}\rightarrow \ModDif{n}{R}{K} / \NF{\ModDif{n}{R}{K}}$ induced by $d^n:R\rightarrow \ModDif{n}{R}{K}$.

By Proposition~\ref{universaldifferential}, we have that for $r\in R$, $\delta(r)\in \m$ for all $\delta\in D^n_{R|K}$ if and only if $\phi(d^n(r))\in \m$ for all $\phi \in \Hom_R(\ModDif{n}{R}{K},R)$. That is, $r\in \m\dif{n+1}{K}$ if and only if $d^n(r)\in \NF{\ModDif{n}{R}{K}}$. Thus $\bar{d}$ is well-defined and injective.

Now, notice that $\m \ModDif{n}{R}{K} \subseteq \NF{\ModDif{n}{R}{K}}$, where, again, multiplication by elements of $R$ occurs via the left factor. We claim that the map induced by $d^n$ from $R$ to $S:=\ModDif{n}{R}{K}/\m \ModDif{n}{R}{K} \cong (\kk \otimes_K R)/\overline{\Delta^{n+1}_{R|K}}$ is surjective, where $\overline{\Delta^{n+1}_{R|K}}$ is the image of ${\Delta^{n+1}_{R|K}}$ in $\kk \otimes_K R$. It follows from the claim that $\bar{d}^n$ is surjective, which concludes the proof.

The claim is clear when $K=\kk$. In the general case, there exists a primitive element $u$ for $\kk$ over $K$; let $f$ be the minimal polynomial of $u$ over $K$. Setting $\delta=u\otimes 1 - 1\otimes u$, we have that $\kk \otimes_K R$ is generated as an algebra over $R$ by $\delta$, where we take $R$ to be the image of $1\otimes R$. Then, in $S$, $\delta^{n+1}=0$. By applying the Taylor expansion and the definition of~$f$, 
\[\begin{aligned}0=f(u\otimes 1)&=f(1\otimes u+\delta)=f(1\otimes u)+\delta \, f'(1\otimes u) + \delta^2 \, f_2(1\otimes u)+\cdots\\
&= \delta \, f'(1\otimes u) + \delta^2 \, f_2(1\otimes u)+\cdots + \delta^n \, f_n(1\otimes u),
\end{aligned}\] 
where $f_i=f^{(i)}/i!$, which is defined in any characteristic. By separability, $f'(1\otimes u)$ is a unit, so \[T=R[\delta]/(\delta^{n+1}, \delta+ \delta^2 (f_2/f')(1\otimes u) + \cdots + \delta^n (f_n/f')(1\otimes u))\]
 surjects onto $S$. But,  $T$ is in fact isomorphic to $R$ itself. Indeed, if not, let $t= r + r_a \delta^a+ \cdots + r_n \delta^n$ be a representative of an element of $T\setminus R$. Then, in $T$, this element is equal to $t-r_a \delta^{a-1}(\delta+ \delta^2 (f_2/f')(1\otimes u) + \cdots + \delta^n (f_n/f')(1\otimes u))$, which can be written as $r +  r'_{a+1}\delta^{a+1} + \cdots + r'_n \delta^n$. Applying this at most $n$ times gives a representative for $t$ in $R$. Thus, the map induced by $d^n$ gives a surjection from $R$ to $S$, which establishes the claim,  {so $R/\m\dif{n+1}{K}$ and $\ModDif{n}{R}{K} / \NF{\ModDif{n}{R}{K}}$ are $K$-vector spaces of the same (finite) dimension. Then, 
 \[ \lambda_R(R/\m\dif{n+1}{K}) = \frac{\dim_K(R/\m\dif{n+1}{K})}{\dim_K(\kk)} =  \frac{\dim_K(\ModDif{n}{R}{K} / \NF{\ModDif{n}{R}{K}})}{\dim_K(\kk)} = \lambda_R(\ModDif{n}{R}{K} / \NF{\ModDif{n}{R}{K}}).\]}

The argument for the first equality in Case~(2) is entirely analogous, with the use of Proposition~\ref{representing-differential} replaced by Proposition~\ref{represent-complete}. It remains to show that $\frk_R(\wModDif{n}{R}{K})=\frk_R(\ModDif{n}{R}{K})$ for each $n$. We now show that any free $R$-summand $F$ of  $\wModDif{n}{R}{K}$ is a free summand of $\ModDif{n}{R}{K}$ and vice versa.

Let $F$ be a free summand of $\wModDif{n}{R}{K}$. By Lemma~\ref{separatedquot}, $\wModDif{n}{R}{K}\cong \sep{\ModDif{n}{R}{K}}$, so the completion map $\ModDif{n}{R}{K}\rightarrow \wModDif{n}{R}{K}$ is surjective. It follows that the inclusion of $F$ into $ \wModDif{n}{R}{K}$ factors through $\ModDif{n}{R}{K}$, so $F$ is a free summand of $\ModDif{n}{R}{K}$.

Conversely, let $F$ be a free summand of $\ModDif{n}{R}{K}$. Since $F$ is a complete module, the splitting map from $\ModDif{n}{R}{K}$ to $F$ factors through $\wModDif{n}{R}{K}$, so $F$ is a free summand of $\wModDif{n}{R}{K}$.
\end{proof}

\begin{remark} \label{equalityvariant}	
In the case of a localization $R$ of a finitely generated $K$-algebra over a field $K$ at a maximal ideal with residue class field $K$, one may prove the previous proposition
slightly differently by showing that the natural map
\[D^{n}_{R|K} /D^n_{R|K} (R,\m)   \rightarrow \Hom_K(R/ \m\dif{n+1}{K}    , K) ,\, E \longmapsto \eta \circ E , \]
is an isomorphism of $K$-vector spaces. Here $\eta:R \rightarrow R/\m \cong K$ denotes the projection and the left hand side is $Q_n$ in the notation of Remark \ref{localunitaryoperators}, whose $K$-dimension is the free rank of the module of principal parts. The map is well defined, since $E \in D^{n}_{R|K}$ sends $ \m\dif{n+1}{K}  $ to $\m$ and $E \in D^n_{R|K} (R,\m)$ is sent to $0$. If $E \notin  D^n_{R|K} (R,\m)  $, then there exists $f \in R$ with $E(f) \notin \m$ and so $\eta \circ E \neq 0$, which gives the injectivity. To prove surjectivity, suppose that $U \subseteq 
\Hom_K(R/ \m\dif{n+1}{K}    , K)$ is the image space. Then there exists a subspace $W \subseteq R/ \m\dif{n+1}{K} $ such that 
\[U=W^\perp =\{ \varphi:R/ \m\dif{n+1}{K} \rightarrow K|_{\varphi(W)}=0 \}.\] Assume that $U$ is not the full space. Then $W \neq 0$ and there exists $h \in W$, $h \neq 0$. In particular, $h \notin  \m\dif{n+1}{K}$ and so there exists $E \in D^n_{R|K}$ with $E(h) \notin \m$. But then $\eta \circ E$ does not annihilate $W$ and so it does not belong to $ U$.	
\end{remark}

\begin{remark} The hypothesis that $(R,\m,\kk)$ is an algebra with pseudocoefficient field $K$ can be weakened slightly in Proposition~\ref{freerank}. Suppose that $R$ is a local $K$-algebra essentially of finite type, and the field extension $K \to \kk$ is algebraic and separably generated. The proof of Proposition~\ref{freerank}(1) goes through with only minor modifications: it is no longer true that there exists a primitive element $u$, but the argument of surjectivity of $\bar{d}^n$ follows in the same way.
	
	If $(R,\m,\kk)$ is a local ring essentially of finite type over a field of characteristic zero, then there exists a subfield $K$ of $R$ such that $K \to \kk$ is is algebraic and separably generated, and, hence, for which $\lambda_R(R/\m\dif{n+1}{K}) =  \frk_R(\ModDif{n}{R}{K})$ for all $n$.
	\end{remark}

The previous result allows us to give a different characterization of differential signature, which suggests that the module of principal parts works as a characteristic free analogue of $R^{1/p^e}$.

\begin{theorem}\label{ThmDiffSigRanks}
Let $(R,\m,\kk)$ be a domain that is an algebra with pseudocoefficient field $K$, and assume that $\Frac(R)/K$ is separable.
 Then, $\dm{K}(R)=\pps{K}(R)$.
\end{theorem}
\begin{proof}
We have that
\[
\dm{K}(R)=\limsup\limits_{n\to\infty}\frac{\lambda(R/\m\dif{n}{A})}{n^d / d!}
= \limsup\limits_{n\to\infty}\frac{\frk_R (\ModDif{n}{R}{K})}{n^d / d!},
\]
where the last equality follows  by Proposition~\ref{freerank}.
By Proposition~\ref{RankDiff}
$\Rank (\ModDif{n}{R}{K})=\binom{d+n}{d}$.
Since 
$\lim_{n\to \infty}\frac{n^d/d!}{\binom{d+n}{d}}=1,$
we obtain that 
\[
\limsup\limits_{n\to\infty}\frac{\frk_R (\ModDif{n}{R}{K})}{n^d / d!}=
\limsup\limits_{n\to\infty}\frac{\frk_R (\ModDif{n}{R}{K})}{\Rank{(\ModDif{n}{R}{K})} },
\]
which concludes the proof.
\end{proof}

We do not know whether the two definitions agree when one relaxes the assumption on existence of a pseudocoefficient field. There is, however, an inequality that holds under a much weaker assumption. We prepare for this with a straightforward lemma.

\begin{lemma}\label{lem:Lunz} Let $K$ be a field, and $(R,\m,\kk)$ be a local $K$-algebra essentially of finite type over a field $\kk$, and assume that $\kk$ is separable over $K$.
	
	Let $\lambda_1,\dots,\lambda_t$ be units of $R$ such that their images in $\kk$ form a separating transcendence basis for $\kk$ over $K$. Let $T=R[x_1,\dots,x_t]$, and $S=T_{\n}$, where $\n=\m+(x_1-\lambda_1,\dots,x_t-\lambda_t)$.
	
	Then, $(S,\n, \kk)$ is an algebra with pseudocoefficient field $K(x_1,\dots,x_t)$,  $${x_1-\lambda_1},\dots,{x_t-\lambda_t}$$ is a regular sequence on $S$, and 
	\[\sum_{s<n}\m\dif{n-s}{K} (x_1-\lambda_1,\dots,x_t-\lambda_t)^s \subseteq \n\dif{n}{K(\underline{x})}.\]
\end{lemma}

\begin{proposition}
	Let $K$ be a field, and $(R,\m,\kk)$ be a local $K$-algebra essentially of finite type over   {$K$}. Assume that $\kk$ is separable over $K$. Then, $\pps{K}(R) \leq \dm{K}(R)$.
\end{proposition}
\begin{proof}
	We use the notation of Lemma~\ref{lem:Lunz}. For the ring $S$, Proposition~\ref{freerank} applies, so $\frk_S(\ModDif{n}{S}{K(\underline{x})})=\ell_S(S/\n\dif{n}{K(\underline{x})})$. Since $S$ is obtained from $R$ by base change and localization, using Proposition~\ref{diffmod-localize}, we find that $\frk_S(\ModDif{n}{S}{K(\underline{x})})=\frk_R(\ModDif{n}{R}{K})$. Applying Lemma~\ref{lem:Lunz}, we obtain an inequality \[\frk_R(\ModDif{n}{R}{K})\leq \ell_S\left(\frac{S}{\sum_{s<n}\m\dif{n-s}{K} (x_1-\lambda_1,\dots,x_t-\lambda_t)^s}\right).\]
	The function $f(n)$ determining the RHS above is the $t$-iterated sum transform of the function $g(n)=\ell_R(R/\m\dif{n}{K})$. It is then an elementary analysis fact that one has $\limsup_n \frac{g(n)}{\binom{n+d}{n}} \geq \limsup_n \frac{f(n)}{\binom{n+t+d}{n}}$. The proposition follows.
\end{proof}

We now turn the case of complete local rings. We extend the definition of principal parts signature to this case.

\begin{definition} Let $(R,\m,\kk)$ be a complete local domain of dimension $d$ with coefficient field $K\cong \kk$. We define the  \emph{principal parts signature}\index{principal parts signature}\index{$\pps{K}(R)$} of $R$ over $K$ as
	\[\pps{K}(R):=\limsup_{n \rightarrow \infty} \frac{\frk_R(\ModDif{n}{R}{A})}{\binom{n+d}{n}}.\]
\end{definition}

 Using the second part of Proposition~\ref{freerank}, the same proof as Theorem~\ref{ThmDiffSigRanks} establishes the following.
 
 \begin{theorem}\label{ThmDiffSigRanksComplete} Let $(R,\m,\kk)$ be a complete local domain with coefficient field $K\cong \kk$. Then $\pps{K}(R) = \dm{K}(R)$.
 \end{theorem}

We mostly work in the situation of algebras with pseudocoefficient fields and complete local rings henceforth, in which case we freely use Theorems~\ref{ThmDiffSigRanks}~and~\ref{ThmDiffSigRanksComplete} to move between the two definitions.

\subsection{Basic properties of differential signature}

We first note that the differential signature is also bounded by one.
\begin{proposition}\label{leq-1}
If $(R,\m,\kk)$ is a domain with pseudocoefficient field $K$ and $\Frac(R) / K$ is separable,
then ${\dm{K}(R)\leq 1}$.
% Equality holds if $R$ is regular.
\end{proposition}
\begin{proof}
This follows immediately from Theorem~\ref{ThmDiffSigRanks}.
\end{proof}

Differential signature behaves well under completion.

\begin{proposition}\label{diff-sig-completion} Let $(R,\m,\kk)$ be a local algebra with coefficient field $K\cong \kk$; we may identify $K$ with a coefficient field for $\widehat{R}$. Then $\dm{K}{(R)}=\dm{K}{(\widehat{R})}$. The same equality holds when one replaces limits superior with limits inferior in the definition of differential signature.
\end{proposition}
\begin{proof}
This is immediate from Proposition~\ref{diff-powers-completion}.
\end{proof}

We now give a preparatory lemma  to show that the differential signature detects $D$-simplicity. We note that this is one of the properties that $F$-signature has.

\begin{lemma}\label{LemmaDimDiffPrimeNotDsimple}
Let $(R,\m)$ be reduced local  {or graded} ring, and  $A$ be a subring.
If $R$ is not a simple  $D_{R|A}$-module, then
$\dim (R/\cP_A)<\dim(R).$
\end{lemma}
\begin{proof}
If $R$ is a domain, the result follows from Corollary~\ref{CorDifPrimeDsimple}, because $\cP_A\neq 0.$
We suppose that $R$ is not a domain. The zero ideal is radical but not prime by hypothesis. Then, by Proposition~\ref{PropMinimalPrime}, the minimal primes of $R$ are $D_{R|A}$-ideals. Thus, the sum $I$ of any set of minimal primes is a $D_{R|A}$-ideal, and for such an ideal $\dim (R/I)<\dim(R)$. Then, since $\cP_A \supseteq I$ by Lemma~\ref{PropDiffPrime}, the inequality holds.
\end{proof}

The following result resembles one of the key features of $F$-signature in prime characteristic for $F$-pure rings. However,  {Theorem~\ref{ThmDifMultDsimple}} holds in characteristic zero and prime. 

\begin{theorem}\label{ThmDifMultDsimple}
Let $(R,\m)$ be reduced local  {or graded} ring, and  $A$ be a subring.
If $\dm{A}(R)>0,$ then  $R$ is a simple  $D_{R|A}$-module.
\end{theorem}
\begin{proof}
We set $d=\dim(R).$
We prove the equivalent statement:   if $R$ is not a simple  $D_{R|A}$-module, then $\dm{A}(R)=0$. 
Since $ \m^n\subseteq \m\dif{n}{A}$ and $\cP_A\subseteq  \m\dif{n}{A}$ for every $n\in\NN,$
we have that
\[
\begin{aligned}
\dm{A}(R)=\limsup\limits_{n\to\infty}\frac{\lambda_R(R/\m\dif{n}{A})}{n^d/d!}
&=\limsup\limits_{n\to\infty}\frac{\lambda_R(R/(\cP_A+\m\dif{n}{A}))}{n^d/d!} \\
&\leq \limsup\limits_{n\to\infty}\frac{\lambda_R(R/(\cP_A+\m^n))}{n^d/d!}.
\end{aligned}
\]
By Lemma~\ref{LemmaDimDiffPrimeNotDsimple}, we have that 
$\dim(R/\cP(A))<d,$ and so $\displaystyle \limsup\limits_{n\to\infty}\frac{\lambda_R(R/(\cP_A+\m^n))}{n^d/d!}$ is zero. Hence, $\dm{A}(R)=0$.
\end{proof}

We end this subsection by noticing that if one replaces the differential operators by Hasse-Schmidt  differentials in the definition of differential signature, a great deal of information is lost.

\begin{remark}
	Let $K$ be a field of characteristic zero, and $(R,\m)$ be an algebra with pseudocoefficient field $K$. One can define a ``Hasse-Schmidt signature'' $s^{HS}_{K}(R)$ by replacing the $n$-th differential power of $\m$ with the set of elements of $\m$ that are sent into $\m$ by every product of at most $n-1$ derivations. If $R$ is not regular, by Remark~\ref{rem:der-simple}, there exists an ideal stable under the action of all derivations. By an argument analogous to Theorem~\ref{ThmDifMultDsimple}, $s^{HS}_{K}(R)=0$ for any such $R$.
\end{remark}

\subsection{Some basic examples}

\begin{example}
	The inequality  {$\dm{K}(R)\leq 1$} in Proposition~\ref{leq-1} does not necessarily hold if $R$ is not a domain. For example, let $K$ be a field and $R=K[x]/(x^2)$. By Remark~\ref{rem-radicals-D-ideals}, we have $\m\dif{n}{K}=(0)$ for $n>2$. Thus, $\dm{K}(R)=2$. The modules of principal parts $\ModDif{n}{R}{K}$ are free $R$-modules of rank 2 for all $n\geq 3$.
\end{example}

We now prepare to show that $\dm{K}(R)=1$ does not imply that $R$ is regular, even if $R$ is a complete domain.

 {
\begin{lemma}\label{normalization-lemma} Let $(R,\m)\subseteq (S,\n)$ be local domains that are $K$-algebras, and suppose that $K$ is a coefficient field of each. If $R$ and $S$ have the same fraction field,  and there is some differential operator $\alpha\in D_{S|K}$ such that $\alpha(S)\subseteq R$ and $\alpha(1)=1$, then $\dm{K}(R)\geq \dm{K}(S)$.
\end{lemma}
}
\begin{proof}
	Let $\alpha\in D_{S|K}$ be as in the statement of the lemma, and suppose that $\alpha$ has order $\leq t$. If $x\in R\setminus \n \dif{n}{K}$, then there is some $\delta\in D^{n-1}_{S|K}$ such that $\delta(x)=1$. Then, $\alpha\circ\delta$ restricted to $R$ is a differential operator in $D_{R|K}$ of order at most $t+n-1$, and $\alpha\circ\delta(x)=1$. Therefore, $x\notin \m\dif{n+t}{K}$. Thus, $\m\dif{n+t}{K}\subseteq R\cap \n\dif{n}{K}$. Now, consider the short exact sequence
	\[ 0 \lra \frac{R}{R\cap \n\dif{n}{K}} \lra \frac{S}{\n\dif{n}{K}} \lra \frac{S}{R+\n\dif{n}{K}} \lra 0\]
	 {As $S/R$ has dimension strictly less than that of $S$,} and since the image of $\m^n$ is contained in the image of $\n\dif{n}{K}$ in $S/R$, by comparison with the Hilbert function one has that 
	\[\limsup_{n\rightarrow\infty}\frac{d!}{n^d}\; \lambda_R  \left(\frac{S}{R+\n\dif{n}{K}}\right)=0,\]
	so 
	\[\dm{K}(S)=\limsup_{n\rightarrow\infty}\frac{d!}{n^d}\; \lambda_R   \left( \frac{R}{R\cap \n\dif{n}{K}} \right) \leq \limsup_{n\rightarrow\infty}\frac{d!}{n^d}\; \lambda_R \left( \frac{R}{\m\dif{n+t}{K}} \right) =  \dm{K}(R),\]
	where the last equality follows from shifting indices by $t$ and ${\limsup\limits_{n\to\infty} (n+t)^d / n^d =1}$.
\end{proof}

\begin{corollary}
	If $(R,\m)$ is a local domain with a perfect coefficient field $K$, the normalization $R'$ of $R$ is local and regular, and $R'/R$ has finite length, then $\dm{K}(R)=1$.
\end{corollary}
\begin{proof}
	 {By Example \ref{reg-1} and Proposition~\ref{leq-1},} we have that $\dm{K}(R')=1$ and $\dm{K}(R)\leq 1$. It then suffices to show that $\dm{K}(R)\geq\dm{K}(R')$. By Lemma~\ref{normalization-lemma}, it suffices to show that there is some $\alpha\in D_{R'|K}$ sending $1$ to $1$ and with image in $R$. Under the hypotheses, $R'$ is differentially smooth over $K$  \cite[17.15.5]{EGAIV}. Given finitely many $K$-linearly independent elements $f_1,\dots,f_s$ of $R'$ and equally many elements $g_1,\dots,g_s$ of $R'$, there is some differential operator $\alpha$ such that $\alpha(f_i)=g_i$. In particular, if we choose $f_1,\dots,f_s$ whose images form a basis of $R'/R$, there is a differential operator that sends $1$ to $1$ and each $f_i$ to 0, and hence sends $R'$ into $R$.
\end{proof}

\begin{example}\label{example-equals-1}
	If $R=K\llbracket x^a \ | \ a \in \Theta \rrbracket \subseteq K\llbracket x \rrbracket$ for some numerical semigroup $\Theta\subseteq \NN$ and some perfect field $K$, then $\dm{K}(R)=1$. One may also compute this example explicitly for $\Theta=\langle 2,3\rangle$ using the description of the differential operators on this ring found in the work of Smith \cite{PaulSmith} and Smith-Stafford \cite{SmithStafford}.
\end{example}

\begin{example}
	If $R=K\llbracket x^2,x^3,y,xy\rrbracket\subseteq R'=K\llbracket x,y\rrbracket$, the normalization of $R$ is $R'$, and the quotient has length one. We have that $\dm{K}(R)=1$. We note that, for $K=\CC$, $D_{R|\CC}$ is not a simple ring \cite[0.13.3]{LS}. Thus positivity of differential signature does not imply simplicity of the ring of differential operators.
\end{example}

We do not know examples of normal rings with differential signature equal to one that are not regular. We thus pose the following question.

\begin{question} If $R$ is a normal domain with coefficient field $K$ and $\dm{K}(R)=1$, must $R$ be regular?
\end{question}

We do not know whether the differential signature exists as a limit rather than a limit superior. If the differential powers form a \emph{graded family}\index{graded family}, i.e., satisfy the containments $\m\dif{a}{K}\m\dif{b}{K}\subseteq \m\dif{a+b}{K}$ for all $a,b\in \NN$, and $R$ is reduced, then this follows from work of Cutkosky. Namely, as an immediate consequence of \cite[Theorem~1.1]{Cutkosky}, we have the following.

\begin{proposition}
	If $(R,\m)$ is local, the differential powers of $\m$ form a graded family, and the dimension of the nilradical of $R$ as a module is less than the dimension of $R$ (e.g., $R$ is reduced), then the differential signature of $R$ exists as a limit.
\end{proposition}

Alas, it is not always the case that differential powers form a graded family.

\begin{example}\label{example-not-D-graded}
	Let $R=K\llbracket x^2, x^3 \rrbracket$. By the description of the differential operators on this ring found in the work of Smith  \cite{PaulSmith} and Smith-Stafford \cite{SmithStafford}, one sees that $x^n \in \m\dif{n+1}{K} \setminus \m\dif{n+2}{K}$ for all $n>1$. In particular, $x^2 \in \m\dif{3}{K}$, but $x^4=(x^2)^2 \notin \m\dif{6}{K}$. Thus, $\{ \m\dif{n}{K} \}_{n\in\NN}$ does not form a graded family.
\end{example}

 {We do not know examples of normal domains whose differential powers do not form a graded family. Then, we  pose the following question.}

 {\begin{question} If $R$ is a normal domain with coefficient field $K$, do the differential powers $\{ \m\dif{n}{K} \}_{n\in\NN}$  form a graded family?
\end{question}
}

 {We give some positive results on convergence and rationality in Section~\ref{sec-duality}.}

%%%%%%%%%%%%%%%%%%%%%%%%%
\subsection{Algorithmic aspects}\label{SubAlg}
%%%%%%%%%%%%%%%%%%%%%%%%%

In this subsection we deal with the question how the free ranks of the modules of principal parts and the differential powers can be computed algorithmically. Throughout we work with a family of  polynomials $\{f_i, 1 \leq i \leq m\}$, in $K[x_1, \ldots , x_k]$ and with the Jacobi-Taylor matrices over the residue class ring $R=K[x_1, \ldots, x_k] /\left(f_1, \ldots, f_m\right)$ from Section~\ref{basics}, or localizations thereof.

\begin{corollary}
	\label{JacobiTaylorunitaryoperators}
	Let $f_1 , \ldots , f_m \in K[x_1 , \ldots , x_k]$ denote polynomials with residue class ring
	$R  =  K[x_1 , \ldots , x_k]/ \left( f_1 , \ldots , f_m \right)$.
Then a differential operator on $R$ of order $\leq n $ is unitary if and only
if the corresponding tuple (see Corollary~\ref{JacobiTayloroperators}) $\left( a_\lambda \right)$ in the kernel of the $n$-th Jacobi-Taylor matrix generates the unit ideal in $R$. In the graded case this is true if and only if one $a_\lambda$ is a unit.
\end{corollary}
\begin{proof}
A differential operator is unitary if and only if the corresponding linear form on $P^n_{R|K}$ is surjective by Lemma~\ref{unitary}. It is clear that a linear form on $P^n_{R|K}$ is surjective if and only if the corresponding tuple $\left(  a_\lambda \right) $ defines a surjection, and this is the same as the $a_\lambda$ generating the unit ideal. 
\end{proof}

Note that if $a_\nu$ is a unit, then $X^\nu$ is sent to a unit by the corresponding operator, since by Corollary~\ref{JacobiTayloroperators}
\[  \sum_\lambda a_\lambda \frac{ 1 }{ \lambda! } \partial^\lambda \left( x^\nu \right) =   a_\nu \frac{ 1 }{ \nu ! }  \partial^\nu \left( x^\nu \right) +  \sum_{\lambda < \nu} a_\lambda \frac{ 1 }{ \lambda! } \partial^\lambda \left( x^\nu \right)  \in a_\nu + (x_1, \ldots, x_k) \,  .\]
In the graded case, the induced operator with values in $K$ does only depend on those $a_\nu$ where $\nu$ has minimal degree.

\begin{example}
For $R=K[x,y,z]/(z^2-xy)$, the transposed second Jacobi-Taylor matrix is
\[ \begin{blockarray}{ccccc} & 1 & a & b & c \\
\begin{block}{c[cccc]}
 1 & 0 & 0 & 0 & 0 \\
  a & -y & 0 & 0 & 0 \\
   b & -x & 0 & 0 & 0 \\
    c & 2z & 0 & 0 & 0 \\
   a^2 & 0 & -y & 0 & 0 \\
   ab & -1 & -x & -y & 0 \\
     ac & 0 & 2z & 0 & -y \\
      b^2 & 0 & 0 & -x & 0 \\
       bc & 0 & 0 & 2z & -x \\
        c^2 & 1 & 0 & 0 & 2z \\
        \end{block} \end{blockarray} .\]
From this we get the unitary differential operators $(1,0,0,0 \ldots, 0,0)$ (which exists always and corresponds to the identity) and
\[ (0,1,0,0,4x,0,2z,0,0,y),  \, (0,0,1,0,0,0,0,4y,2z,x),\,   (0,0,0,1,0,4z,2x,0,2y,2z) . \]
The first of these corresponds to $\partial_x +2x \partial_x \circ \partial_x + 2z \partial_x \circ \partial_z + y \frac{1}{2} \partial_z \circ \partial_z $ and sends $x$ to $1$.
\end{example}

\begin{remark}
\label{JacobiTayloralgorithms}
For $R=K[x_1, \ldots, x_k]/\left(f_1, \ldots, f_m \right)$ we can compute the free rank of $P^n_{R|K}$ with the help of the Jacobi-Taylor matrix $J_n$. It is the maximal number $r$ of tuples
$\left (a_{\lambda}\right)_i $, $i=1, \ldots , r$, inside the kernel of the Jacobi-Taylor matrix such that there exists $\left (c_{\lambda}\right)_i$, $i=1, \ldots , r$, fulfilling the orthogonal relations $a_i \cdot c_j= \delta_{ij}$, since this relation describes the surjectivity of the map from $P^n_{R|K}$ to $R^r$.
	
For a localization $ ( S,\m) $ of $R$, we interpret the short exact sequence from Remark~\ref{localunitaryoperators} 
as
\[ 0 \longrightarrow \ker (J_n) \cap \m \,\left(\bigoplus_{ \mondeg {\lambda} \leq n} Re_\lambda\right) \longrightarrow \ker (J_n) \longrightarrow Q_n \longrightarrow 0 .\]
This provides a way to compute the free ranks of the modules of principal parts algorithmically.

This applies also to the differential powers $ \m\dif{n}{K}$. An element $h \in R$ belongs to $\m\dif{n}{K}$ if and only if the following hold: For all elements $\left( a_\lambda \right) $ in the kernel of the Jacobi-Taylor matrix $J_{n-1}$ we have $\sum_{ \mondeg {\lambda} \leq n-1} \frac{1}{\lambda!} \partial^\lambda (h) a_\lambda \in \m$. As the kernel is a finitely generated module, this is a finite test for one element $h$. In this, we only have to consider a maximal unitary system for the $\left( a_\lambda \right) $.

If we want to know for a fixed element $h$ whether there exists an $n$ such that ${h \notin \m\dif{n}{K}}$, the situation is more complicated. For $n $ large enough the terms 
$$\sum_{ \mondeg {\lambda} \leq n-1} \frac{1}{\lambda!} \partial^\lambda (h) e_\lambda$$ do not change anymore. However, the containment $\sum_{ \mondeg {\lambda} \leq n-1} \frac{1}{\lambda!} \partial^\lambda (h) a_\lambda \in \m$ has to be checked for all kernel elements of all higher Jacobi-Taylor matrices. 

The computation of $\m\dif{n}{K}$ for fixed $n$ is also more complicated. At least over a finite field this is possible. By Proposition~\ref{properties-diff-powers}~(ii) we know that $\m^n \subseteq \m\dif{n}{K}$, and since $R/\m^n$ is finite we can
check the containments for all $h$ separately.
\end{remark}

\begin{corollary}
Let $R= \left(  K[x_1, \ldots, x_k]	/(f_1, \ldots , f_m)\right)_\p $ be a local ring essentially of finite type. Then $R$ is $D$-simple if and only if for every $h \in R$, $h \neq 0$, there exists an element $(a_\lambda)$ in the kernel of some  Jacobi-Taylor matrix such that $ \sum_{\lambda} \frac{1}{\lambda!} \partial^\lambda (h)   a_\lambda $ is a unit.
\end{corollary}
\begin{proof}
This follows from Corollaries~\ref{CorDifPrimeDsimple} and \ref{JacobiTayloroperators} in connection with Remark~\ref{JacobiTayloralgorithms}. 	
\end{proof}

The Jacobi-Taylor matrices are given by finitely many data, all partial derivatives of the defining functions. Therefore it is reasonable to expect that there are certain patterns in them to get some finistic results, to put it optimistically: finite determination of differential signature, its rationality, that the limsup is in fact a limit.
A first result in this direction is the following.

\begin{lemma}
	\label{JacobiTaylorperiodic}
Let $f_1 , \ldots , f_m \in K[x_1 , \ldots , x_k]$ denote polynomials and let $\delta_j $ be the maximum of the exponents of $x_j$ in any monomial in any $f_i$. Set
\[\Lambda = \{\lambda \in \NN^k \ | \  \lambda_j \geq \delta_j \text{ for all } j \}.\]
Let $\sum_{\lambda \in \Lambda} a_\lambda C_\lambda=0$ (over the residue class ring) be a relation among the columns of a Jacobi-Taylor matrix for these data which involves only columns with indices from $\Lambda$.
Then for all $\beta \in \NN^k$ also $\sum_{\lambda \in \Lambda} a_\lambda C_{\lambda+ \beta}=0$, where the columns may refer to a sufficiently larger Jacobi-Taylor matrix.
\end{lemma}
\begin{proof}
The initial relation still holds after passing to a larger Jacobi-Taylor matrix. By induction it is enough to show the statement for $\beta=e_1$. We look at the row given by $(\mu',i)$. If $\mu'=\mu+e_1$, then 
\[ \begin{aligned} \sum_{\lambda \in \Lambda} a_\lambda C_{\lambda +e_1, (\mu',i) } & = \sum_{\lambda \in \Lambda} a_\lambda \frac{1}{ ( \lambda +e_1 - \mu' )!} \partial^{ \lambda +e_1 - \mu'  } (F_i) \\ &=
\sum_{\lambda \in \Lambda} a_\lambda \frac{1}{ ( \lambda - \mu )!} \partial^{ \lambda - \mu  } (F_i) \\ &=   0 . \end{aligned}   \]
If $\mu' \neq \mu+e_1$, then the first component of $\mu'$ is $0$. In this case
\[  \sum_{\lambda \in \Lambda} a_\lambda C_{\lambda +e_1, (\mu',i) }  = \sum_{\lambda \in \Lambda} a_\lambda \frac{1}{ ( \lambda +e_1 - \mu' )!} \partial^{ \lambda +e_1 - \mu'  } (F_i) =0 ,   \]
since the first component is always $\lambda_1+1 > \delta_1$ and so these differential operators annihilate $F_i$.
\end{proof}

\subsection{The graded case}\label{graded case}

We now consider the case of a standard-graded $K$-algebra $R$. In this setting, every differential operator has a decomposition into homogeneous differential operators, and the degree of a homogeneous operator $\delta$ is given as the difference $ \deg (\delta(f)) - \deg (f)$ for every homogeneous element $f$. For example, the degree of $x^\nu \partial^\lambda$ on the polynomial ring is $ \mondeg {\nu} -  \mondeg {\lambda} $. A unitary homogeneous operator sending $f$ to $1$ has degree $- \deg (f)$, and the (non)existence of operators of certain negative degrees imposes strong conditions on the differential signature.

\begin{lemma}
\label{gradedsignature}
	Let $R$ be a standard-graded ring over $K$ of dimension $d$ and multiplicity $e$. Suppose that there exists $\alpha \in {\mathbb R}_{\geq 0}$ such that $(D^n_{R|K} )_\ell = 0 $ for all $\ell < - \alpha n $.	Then $\dm{K}(R) \leq e \alpha^d $.
\end{lemma}
\begin{proof}
	We claim that
		 \[ R_{ >  \alpha n  }  \subseteq \m^{\langle n+1 \rangle} =\{f \in R\;|\; \delta(f) \in \m \text{ for all operators } E \text{ of order } \leq n \} . \] 
		 So let $f$ be a homogeneous element of degree $> \alpha n $. By assumption, every nonzero homogeneous operator $\delta$ of order $ \leq n$ has degree at least $ - \alpha n$. Therefore the degree of $ \delta (f) $ is $ > \alpha n - \alpha n = 0$ and so $\delta(f) \in \m$. It follows that we have a surjection
		 \[   R/  R_{ >  \alpha n  } = R_{\leq  \lfloor \alpha n    \rfloor }     \longrightarrow R/  \m^{\langle n+1 \rangle}  .    \]
		 Hence asymptotically
		 \[ \dim_K ( R/  \m^{\langle n+1 \rangle} )  \leq  \frac{ e}{d!} \alpha^d n^d     \]	 
		  and the result follows.	
\end{proof}

Compare also the proof of Theorem~\ref{ThmDirSumPos} for a bound from below with a similar shape.

\begin{corollary}
	\label{gradedsymdersignature}
	Let $R$ be a standard-graded ring over $K$ of dimension $d$ and multiplicity $e$. Suppose that there exists $\alpha \in {\mathbb R}_{\geq 0}$ such that $ \Hom_R (\Sym^n (\Omega_{R|K}), R )_\ell = 0 $ for all $\ell < - \alpha n $. Then $\dm{K}(R) \leq e \alpha^d $.
\end{corollary}
\begin{proof}
We have to show that the assumption implies that $(D^n_{R|K} )_\ell = 0 $ for all $\ell < - \alpha n $, then the result follows from Lemma~\ref{gradedsignature}. This we prove by induction on $n$. For $n=0$ the statement is true anyway since there is no multiplication of negative degree in a standard-graded ring. For the induction step we look at the short exact sequence
\[  0 \longrightarrow (D^{n-1}_{R|K}) _\ell \longrightarrow (D^{n}_{R|K}) _\ell \longrightarrow  \Hom_R (\Sym^n (\Omega_{R|K}), R )_\ell  \]
which we will discuss in detail in Section~\ref{Comparison}. The homogeneity of the map on the left follows from the beginning of the proof of Theorem~\ref{compareoperatorcomposition}. So suppose that $E$ is an operator of order $\leq n$ and of degree $\ell < - \alpha n$. If $E$ has order $\leq n-1$, then it is $0$ by the induction hypothesis. Hence $E$ does not come from the left and maps to a nonzero element on the right which contradicts the assumption.	
\end{proof}

\begin{remark}
	\label{jacobitaylorhomogeneous}	
	If $E$ is a homogeneous differential operator of degree $ \degoperator$ on a $\ZZ$-graded ring $R$ of finite type and given by a tuple $\left( a_\lambda \right)$ as in Corollary~\ref{JacobiTayloroperators}, then the $a_\lambda$ are homogeneous of degree
	$\deg (a_\lambda) =  \degoperator + \mondeg {\lambda} $. The operator can be decomposed as
	\[\delta = \sum_{\mondeg {\lambda} = u }  a_\lambda \frac{\partial^\lambda}{ \lambda!  }  + \sum_{ \mondeg {\lambda} = u+1 }  a_\lambda \frac{\partial^\lambda}{ \lambda!  } + \cdots   +  \sum_{\mondeg {\lambda} = v }  a_\lambda \frac{\partial^\lambda}{ \lambda!  } ,  \]
	where we suppose that the sums on the very left and on the very right are not $0$. The left sum determines the induced operator with values in $K$ alone, and this induced operator can only be nonzero if $u=- \degoperator $. The number $v$ is the order of the operator, and by Lemma~\ref{JacobiTaylorsymmetric} this last sum determines  the corresponding element in $\Hom (\Sym^v (\Omega_{R|K}) , R)_\degoperator $. 
\end{remark}

If $R$ is a normal standard-graded domain over a field of characteristic $0$ and $U \subseteq \Spec R=X$ is smooth and contains all points of codimension one, then we have
\[ \Hom_R(\Sym^n(\Omega_{R|K}),R ) \cong \Gamma(U, \Hom_X(\Sym^n(\Omega_{X|K}), {\mathcal O}_X ) ) \cong \Gamma(U, \Sym^n (\Der_K {\OO}_X)) . \]
This holds in every degree. If $R$ has an isolated singularity, then we can take $U$ to be the punctured spectrum and we can compute $ \Hom_R(\Sym^n(\Omega_{R|K}),R )_\ell$ on the smooth projective variety  $\Proj R$.

\begin{corollary}
	\label{gradedhypersurfacesymdersignature}
	Let $K$ be a field of characteristic $0$ and let $F \in K[x_1 , \ldots , x_{d+1}]$ ($d \geq 2$) be a homogeneous polynomial of degree $e$. Suppose that $R=K[x_1 , \ldots , x_{d+1}]/(f)$ has an isolated singularity and set $Y=\Proj R$.	Suppose that there exists $\alpha \in {\mathbb R}_{\geq 0}$ such that $ \Gamma(Y,  \Sym^\degsym (\Syz (\partial_1 F, \ldots , \partial_{d+1} F) ) (\degtotal) ) = 0 $ for all $\degtotal < (e- \alpha ) \degsym $. Then $\dm{K}(R) \leq e \alpha^d $.
\end{corollary}
\begin{proof}
	Since we have an isolated singularity we have on $Y$ short exact sequences of locally free sheaves of the form
	\[  0 \longrightarrow \Syz (\partial_1 f, \ldots , \partial_{d+1} f) (\degtotal)  \longrightarrow \bigoplus_{ d+1}    {\mathcal O}_Y (\degtotal-e+1)  \xrightarrow{\partial_1 f, \ldots , \partial_{d+1} f } {\mathcal O}_Y(\degtotal) \longrightarrow 0 \,  \]
	for all twists $\degtotal$. On the right we have the Jacobi matrix. Hence we get a correspondence	
	\[   \Der_K(R)_{m-e}   = \Gamma(Y,   \Syz (\partial_1 f, \ldots , \partial_{d+1} f) (\degtotal )      ) \,  \]
	because of the following: A global section of $ \Syz (\partial_1 f, \ldots , \partial_{d+1} f ) $ over $Y$ in total degree $\degtotal $ is a syzygy $(s_1, \ldots,s_{d+1})$ where the $s_i$ are homogeneous elements of degree $\degtotal -e+1$. This corresponds via $x_i \mapsto dx_i \mapsto s_i$ to a derivation on $R$ of degree $ \degtotal -e$.
	On a sheaf level we can write this as $ \widetilde {(\Der_KR)} ( \degtotal -e) \cong \Syz ( \partial_1 f, \ldots , \partial_{d+1} f )(\degtotal) $, where the tilde denotes taking the corresponding sheaf of a graded module. From this we get
	\[ \Sym^\degsym  \widetilde {(\Der_KR)}  \cong  \Sym^\degsym ( \Syz ( \partial_1 f, \ldots , \partial_{d+1} f ) (e)  ).   \]
	We translate this back to the punctured cone and  to get	
	\[  \Hom_R(\Sym^n(\Omega_{R|K}),R )_{ \degtotal -\degsym e}  = \Gamma(Y, \Sym^\degsym(\Syz( \partial_1 f, \ldots , \partial_{d+1} f     )     ) (\degtotal ) ), \]
because $ \Hom_R(\Sym^n(\Omega_{R|K}),R ) $ is reflexive.

	By assumption we know the nonexistence of nonzero global sections for the twists $\degtotal < (e- \alpha) \degsym $. Hence we deduce $  \Hom_R(\Sym^n(\Omega_{R|K}),R )_{ \ell } =0 $   for $\ell < - \alpha \degsym $ and Corollary~\ref{gradedsymdersignature} gives the result.	
\end{proof}

\begin{remark} \label{Symsyzcomputation}
The global sections of $\Sym^\degsym ( \Syz (\partial_1 f, \ldots , \partial_{d+1} f )) $ and its twists can be computed with the short exact sequences
{\small
\[ 0 \to \Sym^\degsym ( \Syz (\partial_1 f, \ldots , \partial_{d+1} f )) \to  \Sym^\degsym ( \bigoplus_{ d+1} {\OO}_Y (-e+1)  ) \to \Sym^{\degsym-1} ( \bigoplus_{ d+1} {\OO}_Y  (-e+1) ))  \to 0. \]}
With the identifications $ \Sym^\degsym( \OO_{Y}(-e+1)^{\oplus d+1} ) \cong   \OO_{Y}( \degsym (-e+1))^{\bigoplus \binom{\degsym+d}{d} } $, the map on the right hand side is given as $e_\nu \mapsto  \sum_j \partial_j f \, e_{\nu -e_j}$. A section in the symmetric power of the syzygy bundle (and its twists) is a tuple in the kernel of this map. The matrices describing these maps appear also as the submatrices $T_q^\text{tr}$ of the Jacobi-Taylor matrices $J_q$, see Remark~\ref{JacobiTaylorrelation} and Lemma~\ref{JacobiTaylorsymmetric}.
\end{remark}

\begin{example}
Let $F$ be a homogeneous polynomial of degree $3$ in $3$ variables over an algebraically closed field $K$ of characteristic $0$ such that $Y=\Proj K[x,y,z]/(f)$ is an elliptic curve. The Euler derivation determines because of $x\partial_1f +y\partial_2f+z\partial_3f = 0$ a short exact sequence
\[0 \longrightarrow \OO_Y \longrightarrow \Syz(\partial_1f,\partial_2f,\partial_3f) (3) \longrightarrow \OO_Y \longrightarrow 0 . \]
This does not split since the space of global sections in the middle has dimension $1$. Therefore the syzygy bundle is the bundle $F_2$ in Atiyah's classification of bundles on an elliptic curve \cite{atiyahelliptic}. Hence for the symmetric powers we get
\[ \Sym^\degsym( \Syz(\partial_1f,\partial_2f,\partial_3f) )  \cong \Sym^\degsym (F_2 (-3) ) \cong (\Sym^\degsym (F_2))(-3n) \cong  F_n (-3n)  ,  \]
where $F_n$ is again from Atiyah's classification, i.e. $F_n$ are the (semistable) bundles of rank $n$ which are the unique nontrivial extensions of $F_{n-1}$ by $\OO_Y$. For $\degtotal < 3 \degsym$ there are no global sections of 
$ \Gamma(Y,  \Sym^\degsym (\Syz (\partial_1 f,  \partial_{2} f, \partial_3 f) ) (\degtotal) ) \cong \Gamma(Y, F_\degsym( \degtotal -3 \degsym ) ) $. So 
this reproves known facts  \cite{DiffNonNoeth} mentioned in Example~\ref{example-BGG}, and Corollary~\ref{gradedhypersurfacesymdersignature} with $e=3$ and $\alpha =0$ shows that the differential signature is $0$.
\end{example}

The following theorem shows that positive differential signature is in the graded case related with many other relevant notions for a singularity.

\begin{theorem}
	\label{Possiganeg}
	Let $K$ be an algebraically closed field of characteristic zero and let $R$ be an $\NN$-graded $K$-algebra of dimension at least two that is generated in degree one. Assume that $R$ is a Gorenstein ring and has an isolated singularity at the homogeneous maximal ideal. Suppose that $R$ has positive differential signature. Then, the $a$-invariant of $R$ is negative.
\end{theorem}

\begin{proof}
	Let $ X = \operatorname{Proj} R $ be the smooth projective variety corresponding to $R$, let ${\mathcal O}_X(1)$ its very ample line bundle and let $\omega ={\mathcal O}( a ) $ be the canonical line bundle. For this interpretation of the $a$-invariant, see \cite[Section 3.6]{BrHe}.
	
	Assume that $a \geq 0$ and that the differential signature is positive. By Theorem~\ref{ThmDifMultDsimple}, the ring $R$ is simple as a module over the ring of differential operators. This implies that the tangent bundle $T_X$ is big \cite[Theorem 1.2]{Hsiaobigness}. We have $\bigwedge^{\dim(X)} T_X = {\mathcal O}_X(-a)$. On the other hand, if $a \geq 0$, then  $T_X$ is semistable with respect to ${\mathcal O}_X(1)$  \cite[Theorem 3.1]{Peternell}. Then also the  restriction of $T_X$ to a generic (smooth) complete intersection curve $C$ of sufficiently high degree is semistable  \cite{Flennerrestriction} and its degree is still $\leq 0$. Then also its symmetric powers are semistable  \cite[Theorem I.10.5]{ Hartshorneamplesubvarieties} and of nonpositive degree. 
	Then,   
	%applied to ${\mathcal O}_{ {\mathbb P}(T_X) } (1)$ on ${\mathbb P}(T_X)$
	the restriction of $T_X$ to $C$ is big \cite[Corollary 2.2.11]{LazBook1}. But this contradicts the Riemann-Roch theorem for curves.	
\end{proof}

This means also that the smooth projective variety corresponding to an isolated graded Gorenstein singularity is a Fano variety and has in particular negative Kodaira-dimension \cite[Definition V.1.1]{Kollarrationalcurves}. It follows for example that a graded hypersurface $R=K[x_1, \ldots , x_n]/(f)$ with an isolated singularity with positive differential signature must have degree $\operatorname{deg} (f) \leq \operatorname{dim} (R)$. We conjecture, in analogy with the situation in positive characteristic between $F$-regular, positive $F$-signature and negative $a$-invariant (\cite{Harainjectivity}, \cite[Section 5.3]{Hararational}, \cite[Theorem~0.2]{AL}), that the converse is true, but the first open case is already that of cubics in four variables.

The following corollary uses singularity notions from the minimal model program (see \cite{kollarsingularities}) and $F$-singularities which are explained in the next section.

\begin{corollary}\label{CorRedCharP}
	Let $K$ be an algebraically closed field of characteristic zero and let $R$ be a standard-graded normal $K$-domain with an isolated singularity and that is a Gorenstein ring. Suppose that $R$ has positive differential signature. Then $R$ has a rational singularity, it is log-terminal, and it is of strongly $F$-regular type. 
\end{corollary}
\begin{proof}
	The rationality of the singularity follows from work of Flenner-Watanabe \cite{Flennerrational, Watanaberational} which says that under the hypotheses negative $a$-invariant implies rationality. A Gorenstein rational singularity is also log-terminal. 
	Then, the rationality implies that $R$ has $F$-rational type \cite[Theorem~1.1]{HaraSing}.
	Under the Gorenstein condition this means that $R$ has strongly $F$-regular type, which again implies   that $R$ is log-terminal \cite[Theorem~5.2]{HaraSing}.
\end{proof}

We discuss the notions mentioned in Corollary \ref{CorRedCharP} and other connections to  $F$-singularities in the next section.

%%%%%%%%%%%%%%%%%%%%%%%%%%%%%%%%%%%%%%%%%%%%%%%%%%%%%%%%%%%%%%%%%%

%%%%%%%%%%%%%%%%%%%%%%%%%%%%%%%%%%%%%%%%%%
\section{Differential signature in prime characteristic}\label{five}
%%%%%%%%%%%%%%%%%%%%%%%%%%%%%%%%%%%%%%%%%%
%%%%%%%%%%%%%%%%%%%%%%%%%%%%%%%%%%%%%%%%%%

In this section we focus on positive characteristic. In particular, we compare the $F$-signature and differential signature in the case where both invariants can be defined.

%%%%%%%%%%%%%%%%%%%%%%%%%
\subsection{Differential Frobenius powers}
%%%%%%%%%%%%%%%%%%%%%%%%%

\begin{setup}
Unless specified, in this section $R$ denotes an $F$-finite Noetherian ring with prime characteristic $p>0$. 
\end{setup}

\begin{definition}
Let $R$ be a Noetherian  ring with prime characteristic $p>0$. 
\begin{enumerate}
\item[(i)]  We note that $R$ acquires an $R$-module structure by restriction of scalars via the $e$-th iteration of the Frobenius map, $F^e$. We denote this module action on $R$ by $F^e_* R$. 
 {To make explicit what structure is considered, we denote $F^e_* f$ for an element in $F^e_*$.} 
\item[(ii)] We say that $R$ is $F$-finite if  $F^e_* R$ is a finitely generated $R$-module.
\item[(iii)]  {If $R$ is $F$-finite, we say that $R$ is \textit{$F$-pure} if the natural map $R\to F^1_* R$ splits.}
\item[(iv)]  If $R$ is a domain,  we say that $R$ is \textit{strongly $F$-regular} if for every $r\in R$, $r \neq 0$, there exists $e\in\NN$ such that the map $\varphi:R\to F^e_* R$
defined by $1\mapsto F^e_* r$ splits.
\item[(v)] We denote  {$\End_{R^{p^e}}(R)$} by $D^{(e)}_R$.\index{$D^{(e)}_R$}
\item[(vi)]  An additive map $\psi:R\to R$ is a \textit{$p^{e}$-linear map} if
$\psi(r f)=r^{p^e}\psi(f).$ Let $\mathcal{F}^e_R$ be the set of all
the $p^{e}$-linear maps.
\item[(vii)] An additive map $\phi:R\to R$ is a \textit{$p^{-e}$-linear map} if
$\phi(r^{p^e} f)=r\phi(f).$ Let $\cC^e_R$\index{$\cC^e_R$} be the set of all the
$p^{-e}$-linear maps.
\end{enumerate}
\end{definition}

\begin{remark}\label{RemCartFrobDmod}
Let $R$ be a reduced $F$-finite ring.
We note that
$$
 \mathcal{F}^e_R \cong
\Hom_R(R, F^e_* R),\; \cC^e_R \cong
\Hom_R(F^e_* R,R),\;  \text{and} \; D^{(e)}_R\cong \Hom_R(F^e_* R,F^e_* R).
$$
If $R$ is  $F$-pure  and $\pi\in  F^e_* R\to R$ is a splitting of the inclusion, then map that sends  $\phi\in \Hom_R(F^e_* R,F^e_* R)$ to $\pi\circ\phi\in \Hom_R(F^e_* R,R)$ is a surjection. Furthermore, this surjection splits.
 \end{remark}

We now recall a definition of $F$-signature.  

 \begin{definition}[{\cite{SmithVDB,HLMCM,TuckerFSig,WatanabeYoshida}}]
 Let $(R,\m,K)$ be either a local ring or a standard graded $K$-algebra  {of dimension $d$}. Suppose that $R$ is $F$-finite. Let $a_e$ denote the biggest rank of an $R$-free direct summand of $F^e_* R$, and  {$\alpha=\hbox{log}_p[K:K^p]$}. Note that $a_e=0$ for all $e$ if $R$ is not $F$-pure. The \textit{$F$-signature} is defined by
 $$s(R)=\lim\limits_{e\to\infty}\frac{a_e}{p^{d(e+\alpha)}}.$$\index{$s(R)$}
\end{definition}

\begin{remark} [\cite{AE,YaoObsFsig}]
Let $(R,\m,K)$ be an $F$-finite $F$-pure ring   {of dimension~$d$}.
We define
\[
I_e=\{r\in R\;|\; \varphi(r)\in\m\;\; \forall \varphi\in\cC^e_R \}.
\]\index{$I_e$}
Then, if $R$ is $F$-finite, one has the equality $s(R)=\lim_{e\rightarrow\infty} \lambda(R/I_e)/p^{ed}$. In general, if $R$ is not $F$-finite, we define the $F$-signature by this formula.
\end{remark}

We recall a well-known description of the differential operators in prime characteristic. We include Part~(i) for comparison with Lemma~\ref{I-diff-ops}.

\begin{proposition}\label{RemYek}
Let $(R,\m,K)$ be an $F$-finite local ring of prime characteristic $p$. 
\begin{enumerate}
\item[(i)]\label{De-Delta-bracket} $D^{(e)}_{R|\ZZ}$ is the set of $\Delta_{R|\ZZ}^{[p^e]}$-differential operators of $R$.\index{$D^{(e)}_{R}$}
\item[(ii)] $D_{R|\ZZ}=\bigcup_{e\in\NN} D^{(e)}_R$.
\item[(iii)] Set $\mu=\dim_{K^p} (R/\m^{[p]})$. Then $D^{p^e-1}_{R|\ZZ} \subseteq D^{(e)}\subseteq D^{\mu(p^{e}-1)}_{R|\ZZ}$.
\item[(iv)] Suppose that $R$ is the localization of an algebra of finite type and let $\mu$ denote its global embedding dimension. Then $ D^{p^e-1}   \subseteq     D^{(e)}       \subseteq D^{\mu ( p^e -1)} $.
\end{enumerate}
\end{proposition}
\begin{proof}
For (i), (ii), (iii), we refer previous work  {\cite[Lemma~1.4.8,~Theorem~1.4.9]{Ye}}  {(see also \cite[Susbection~2.5]{SmithVDB})}.
(iv). We write $R=S_{\mathfrak m}$ with $S=K[X_1, \ldots, X_\mu]/ {\mathfrak a} $.
The ideal $\Delta$ in $ S \tensor_K S $ is generated by $ (X_1-Y_1,   ... ,X_ \mu-Y_\mu)$ and has thus $\mu$ generators. This is also true for $R$. As for any ideal in positive characteristic we have the containments
\[  \Delta^{ \mu (p^e-1)+1} \subseteq \Delta^{[p^e]}  \subseteq \Delta^{p^e} .\]
By looking at the $R$-linear forms on $R \tensor_K R$ modulo these powers we get the inclusions
\[  D^{p^e-1}   \subseteq     D^{(e)}  \subseteq D^{\mu (p^e -1)} . \]
\end{proof}

\begin{definition}
Let $R$ be an $F$-finite ring of characteristic $p>0$.
Let $I$ be an ideal of $R$, and $e$ be a positive integer.
We define the \textit{differential Frobenius powers} of $I$ by
\[ I^{\Fdif{p^e}}= \{f\in R \, | \, \delta(f)\in I \hbox{ for all } \delta\in \End_{R^{p^e}}(R)\}.\]\index{$I^{\Fdif{n}}$}
\end{definition}

This notion enjoys many of the nice properties that differential powers enjoy. For example:

\begin{lemma}\label{properties-Fdiff-powers} Let $R$ be a ring and  $I,J_\alpha\subseteq R$ be ideals.
\begin{itemize}
\item[(i)] $I^\Fdif{p^e}$ is an ideal;
\item[(ii)] $I^{[p^e]}\subseteq I^\Fdif{p^e};$
\item[(iii)] $\left(\bigcap_{\alpha}J_\alpha\right)^\Fdif{p^e}
=\bigcap_{\alpha }(J_\alpha)^\Fdif{p^e}.$
\item[(iv)] If $I$ is $\p$-primary, then $I^\Fdif{p^e}$ is also $\p$-primary.
\end{itemize}
\end{lemma}
 \begin{proof}
\begin{itemize} 
\item[(i)] This follows from the observation that if $\delta\in D^{(e)}$, then $\delta\circ f\in D^{(e)}$ for any $f\in R$.
\item[(ii)] This is immediate from $I^{[q]}\subseteq (I\cap R^q)R$ and the previous part.
\item[(iii)] Follows  from the definition.
\item[(iv)] This is analogous to the proof for differential powers  {\cite[Proposition 2.6]{SurveySP}.} \qedhere
\end{itemize}
\end{proof} 

We also have the following analogue of Lemma~\ref{diff-localize2}.

\begin{lemma}\label{Fdiff-localize} 
Let $W$ be a multiplicative set in $R$ and $I$ an ideal. Suppose that $R$ is $F$-finite. Then $I^\Fdif{p^e} (W^{-1}R) = (W^{-1} I)^\Fdif{p^e}$.
\end{lemma}
\begin{proof}
Since $R$ is $F$-finite, we have that 
\[D^{(e)}_{W^{-1}R}=\Hom_{W^{-1}R^{p^e}}(W^{-1}R, W^{-1}R)=W^{-1}\Hom_{R^{p^e}}(R, R)=W^{-1} D^{(e)}_R\] by the natural map. 

We first show $I^\Fdif{p^e} (W^{-1}R) \subseteq (W^{-1} I)^\Fdif{p^e}$.
 Let  {$r\in I^\Fdif{p^e}$}, $w\in W$, and $\delta\in D^{(e)}_{W^{-1}R}$. Write $\delta=\frac{1}{v} \cdot \eta$, with $v\in W$ and $\eta\in D^{(e)}_R$. Then $\delta(\frac{r}{w})=\frac{\eta(rw^{p^e-1})}{vw^{p^e}}$. Because $I^\Fdif{p^e}$ is an ideal containing $r$, we have $\eta(rw^{p^e-1})\in I$, and $\delta(\frac{r}{w})\in (W^{-1} I)^\Fdif{p^e}$. 

We now focus on the other containment. Let $\frac{r}{w}$ lie in $ (W^{-1} I)^\Fdif{p^e}$, fix $v\in W$ such that $v(W^{-1}I \cap R)\subseteq I$, and take some $\delta \in  D^{(e)}_R$. We have that $\delta(w^{p^e-1} r) = w^{p^e} \delta(\frac{r}{w}) \in W^{-1}I \cap R$. Then, $\delta(v^{p^e}w^{p^e-1} r)=v^{p^e} \delta(w^{p^e-1} r) \in I$. Since $\delta$ was arbitrary, $v^{p^e} w^{p^e-1} r \in I^\Fdif{p^e}$. Thus, $\frac{r}{w}\in I^\Fdif{p^e} W^{-1}R$, as required.
\end{proof}

We now give a result that is a key ingredient to compare both signatures. 

\begin{proposition}\label{PropDvsCartIdeals}
Let $(R,\m,K)$ be an $F$-finite $F$-pure local ring.
Then,  $\m^{\Fdif{p^e}}=I_e$.
\end{proposition}
\begin{proof}
We show the equivalent statement $R\setminus \m^{\Fdif{p^e}}=R\setminus I_e$.

Let $f\not\in \m^{\Fdif{p^e}}.$ Then, there exists $\delta\in D^{(e)}$ such that $\delta(f)=1.$ By Remark~\ref{RemCartFrobDmod}, there exist a map $\tilde{\delta}\in\Hom_R(F^e_*R,F^e_*R)$ such that $\tilde{\delta}(F^e_*f)=1.$
Let $\beta:F^e_*R\to R$ be a splitting. Then, $\beta(\tilde{\delta}(F^e_*f))=1$. 
Since $\beta\circ \tilde{\delta}\in \Hom_R(F^e_*R,R)$, there exists   {$\phi\in \mathcal{C}^e_R$}
such that $\phi(f)=1$ by Remark~\ref{RemCartFrobDmod}. Hence, $f\not\in I_e.$

 {Let $f\not\in I_e$.  Then, there exists $\phi\in \cC^{e}_R$ such that $\phi(f)=1.$ 
By Remark~\ref{RemCartFrobDmod}, there exists a map $\tilde{\phi}\in\Hom_R(F^e_*R,R)$ such that $\tilde{\phi}(F^e_* f)=1.$
Let $\iota:R\to F^e_*R$ be the inclusion. Then, $\iota(\tilde{\phi}(F^e_* f))=1$. 
Since $\iota\circ \tilde{\phi}\in \Hom_R(F^e_*R,F^e_* R)$, there exists  $\delta\in D^{(e)}_R$
such that $\delta\cdot f=1$ by Remark~\ref{RemCartFrobDmod}. Hence, $f\not\in  \m^{\Fdif{p^e}}.$
}
\end{proof}

%%%%%%%%%%%%%%%%%%%%%%%%%%%%%%%%%%%%%%%%%
\subsection{Differential signature and $F$-signature}
%%%%%%%%%%%%%%%%%%%%%%%%%%%%%%%%%%%%%%%%%%

Using Proposition~\ref{PropDvsCartIdeals}, we observe that the $F$-signature can be defined in terms of differential Frobenius powers.

\begin{corollary}[{see~\cite[Corollary 2.8]{AE}}]\label{RemFsigDifPowers}
Suppose that $(R,\m,K)$ is an $F$-finite $F$-pure  {local} ring. Let $d=\dim(R)$.
Then,  
\[s(R)=\lim_{e\rightarrow\infty} \frac{\lambda(R/\m^{\Fdif{p^e}})}{p^{ed}}.\]
\end{corollary}
\begin{proof}
This follows from the fact that 
$\m^{\Fdif{p^e}}=I_e$
for every $e\in\NN$ by Proposition~\ref{PropDvsCartIdeals}.
\end{proof}

\begin{remark}\label{analogy} Let $K$ be a perfect field of positive characteristic, and $(R,\m)$ be an algebra with pseudocoefficient field $K$. Set $\ModDif{[p^e]}{R}{K}=(R\otimes_K R)/\Delta^{[p^e]}_{R|K}$.
By the same argument in Proposition~\ref{freerank}, using Proposition~\ref{RemYek}~(i) and Proposition~\ref{representing-differential}, one can show that $\lambda_R(R/\m^{\Fdif{p^e}})=\frk_R(\ModDif{[p^e]}{R}{K})$. Thus, if $R$ is $F$-pure,
\[s(R)=\lim_{e\rightarrow\infty} \frac{\frk_R\big(\ModDif{[p^e]}{R}{K}\big)}{p^{ed}}.\]
We note also that 
\[\mu_R\big(\ModDif{[p^e]}{R}{K}\big)=\lambda_R\left( \frac{R/\m \otimes_K R}{\Delta^{[p^e]}_{R|K}}\right)=\lambda_R \big(R/\m^{[p^e]}\big),\]
so that $e_{HK}(R)=\lim_{e\rightarrow\infty} \mu_R\big(\ModDif{[p^e]}{R}{K}\big)/p^{ed}$. For comparison, \[\mu_R\big(\ModDif{n}{R}{K}\big) =\lambda_R\left( \frac{R/\m \otimes_R R}{\Delta^{n+1}_{R|K}}\right)=\lambda_R \big(R/\m^{n+1}\big),\]
where the right hand equation comes from \cite[Corollaire~16.4.12]{EGAIV}, so that $e(R)= d! \lim_{e\rightarrow\infty} \mu_R\big(\ModDif{n}{R}{K}\big)/n^d$. This motivates the analogy that the differential signature is to  the $F$-signature as Hilbert-Samuel multiplicity is to the Hilbert-Kunz multiplicity.

The function $\mu_R\big(\ModDif{n}{R}{K}\big)$ is studied by Kunz~\cite{KunzDiff} under the name of \emph{differential Hilbert series}, without the assumption that $R$ is algebra with a pseudocoefficient field.
\end{remark}

\begin{remark}
	Continuing with the previous remark, one may speculate what the analogue of the Hilbert-Kunz function for an $\m$-primary ideal $I$ or an Artinian $R$-module $M$ and what the analogue of tight closure in the setting of principal parts might be. Since the Hilbert-Kunz numerator function is given as $e \mapsto \lambda_R ( R/I^{[p^e]})= \lambda_R (R/I \tensor_R {}^e R)$, the analogue function is $n \mapsto \lambda_R\big( M \otimes_R \ModDif{n}{R}{K} \big) $ for an Artinian $R$-module $M$. Tight closure can be reduced to the tight closure of $0$ in an Artinian module $M$, by declaring $v \in 0^*$ if the normalized Hilbert-Kunz functions (divided by $p^{e d}$) of $M$ and of $M/(v)$ agree asymptotically (see \cite[Theorem 5.4]{CraigBookTC}). Hence the condition that $ \frac{\lambda_R\big( M \otimes_R \ModDif{n}{R}{K} \big) }{n^d} $ and  $ \frac{\lambda_R\big( M/(v) \otimes_R \ModDif{n}{R}{K} \big) }{n^d} $ coincide asymptotically defines a closure operation. If the differential signature is positive, then the substantial free part of $\ModDif{n}{R}{K} $ should imply that this closure is trivial for such rings.
\end{remark}

\begin{lemma} \label{LemmaCofinalDif}
 {Let $(R,\m,K)$ be a ring of positive characteristic $p$, and let $\mu$ be one of the numbers of Proposition~\ref{RemYek} (iii), (iv).} Then, for any ideal $I \subseteq R$ and any $e>0$ one has the containments
 { 
\[
I\dif{\mu (p^e-1)}{\ZZ} \subseteq I^{\Fdif{p^e}} \subseteq  I\dif{p^e-1}{\ZZ}.
 \]
 }
\end{lemma}
\begin{proof}
This follows directly  from 
the definition of $\ZZ$-linear differential powers, differential Frobenius powers, and Proposition~\ref{RemYek}.
\end{proof}

We are ready to compare both signatures. In particular, we obtain that, in the $F$-pure case, one is positive if and only if the other is. As a consequence, we have that the differential signature also detects strong $F$-regularity when the ring is $F$-pure.

\begin{lemma}\label{LemmaCompDifMultFsig}
Let $(R,\m,K)$ be  local $F$-pure ring of prime characteristic $p$ and let $\mu$ be one of the numbers of Proposition \ref{RemYek} (iii), (iv). Then,
$$
 s(R)\leq  \frac{\mu^d}{d!}  \dm{\ZZ}(R).
$$
\end{lemma}
\begin{proof}
We have, by Lemma~\ref{LemmaCofinalDif}, 
 {$\m\dif{\mu(p^e-1)}{\ZZ} \subseteq \m^{\Fdif{p^e}}$}.
As a consequence,
$$
 \lambda_R(R/ \m^{\Fdif{p^e}})\leq \lambda_R(R/\m\dif{\mu(p^e-1)}{\ZZ} ).
$$
By dividing by $p^{e\dim(R)},$ we obtain that
$ s(R)\leq \frac{\mu^d}{d!} \dm{\ZZ}(R).$
\end{proof}

\begin{remark}
	Under the same hypotheses, a similar argument yields the inequality 
	$\liminf\limits_{n\to\infty} \lambda_R(R/\m\dif{n}{\FF_p})/n^d \leq s(R).$ In particular, when the sequence defining differential signature converges, we have that  {$\frac{1}{d!} \dm{\ZZ}(R)\leq s(R)$}.
\end{remark}

\begin{theorem}\label{ThmFregPos}
Let $(R,\m)$ be an $F$-finite $F$-pure local ring of prime characteristic $p$.
Then,  $\dm{\ZZ}(R)>0$ if and only if $R$ is a strongly $F$-regular ring.
\end{theorem}
\begin{proof}
We first assume that  $\dm{\ZZ}(R)>0$. Since every $F$-pure ring is reduced, we can apply Theorem~\ref{ThmDifMultDsimple}. We conclude that $R$ is $D_{R|\ZZ}$-simple. Then, $R$ is strongly $F$-regular 
\cite[Theorem~2.2]{DModFSplit}.

We now assume that $R$ is strongly $F$-regular. Then,  $s(R)>0$  \cite[Theorem 0.2]{AL}.
Then, the claim follows from  Lemma~\ref{LemmaCompDifMultFsig}.
\end{proof}

The previous theorem does not hold without the assumption that the ring is $F$-pure, see Example~\ref{palmostfermat} below.

\begin{remark}
We point out that if $(R,\m,K)$ is a complete local ring with $K$ perfect, then
$D_{R|\ZZ}=D_{R|K}$, and the relation in Lemma~\ref{LemmaCofinalDif} still holds when 
$\ZZ$ is replaced by $K$ \cite{Ye}. As a consequence, we can replace 
$\dm{\ZZ}(R)$ for $\dm{K}(R)$ in the previous theorem.
\end{remark}

\begin{proposition}
	\label{sandwichpositive}
	Let $(R,\m,K)$ be a local ring.
	Let $(S,\n,K)$ be a regular ring of positive characteristic $p$ and let $\iota: S\subseteq R $ be a module-finite  extension of local rings. Suppose that there is an embedding $\rho:R \subseteq S^{1/p^t}$ of local rings such that the composition $\rho\circ \iota: S \to S^{1/p^t} $ is the inclusion. Then, $\dm{\ZZ}(R)>0.$
\end{proposition}

\begin{proof}
Let $d=\dim(R)=\dim(S)$.
Let $f \in R\setminus \{0\}$ such that 
$f^{p^t}\not\in \n^{{\Fdif{p^e}}}$.
Then, there exists $\delta\in \Hom_{S^{p^e}} (S,S)$ such that $\delta(f^{p^t})=1$. By extracting $p^t$-th roots, we get $\widetilde{\delta}\in \Hom_{S^{p^{e-t}}} (S^{1/p^t},S^{1/p^t})$ such that $\widetilde{\delta}(f)=1$.
Let $\beta: S^{1/p^t}\to S$ be an $S$-linear splitting.

Observe that $R^{p^e} \subseteq R \cap S^{p^{e-t}} \cap S$. Thus, $\beta,\iota,\rho$, and $\widetilde{\delta}$ are all $R^{p^e}$-linear, hence the map $\iota\circ\beta
\circ
\widetilde{\delta} \circ \rho: R\to R$ is as well. Since this composition sends $f$ to $1$, we conclude that $f\not\in \m^{{\Fdif{p^e}}}$.
Thus, \[\m^{{\Fdif{p^e}}} \subseteq  (\n^{{\Fdif{p^e}}})^{1/p^t} \cap R = (\n^{1/p^t})^{[p^e]} \cap R.\] It follows that there is a surjection \[R/\m^{\Fdif{p^e}} \twoheadrightarrow R/((\n^{1/p^t})^{[p^e]}\cap R).\] It follows $\lambda_R(R/\m^{\Fdif{p^e}}) \geq \lambda_R(R/((\n^{1/p^t})^{[p^e]}\cap R))$. Also, there is an injective map \[S/((\n^{1/p^t})^{[p^e]}\cap S)\hookrightarrow R/((\n^{1/p^t})^{[p^e]}\cap R),\] hence the inequality $\lambda_S(S/((\n^{1/p^t})^{[p^e]}\cap S)) \leq \lambda_R(R/((\n^{1/p^t})^{[p^e]}\cap R))$: we have used that the $R$-length of the latter module equals its $S$-length, since the residue fields agree. Now, $(\n^{1/p^t})^{[p^e]}\cap S = \n^{[p^{e-t}]}$ for $e\geq t$, so we obtain the inequality
\[\dm{\ZZ}(R)=\limsup_{e\to\infty}\frac{\lambda_R( R/ \m^{\Fdif{p^{e}}})}{p^{ed}}\geq \limsup_{e\to\infty}\frac{\lambda_S(S/  \n^{[p^{e-t}]})}{p^{ed}}=\limsup_{e\to\infty}\frac{p^{(e-t)d}}{p^{ed}}=\frac{1}{p^{td}}>0.\]
\end{proof}

%%%%%%%%%%%%%%%%%%%%%%%%%
\subsection{Characteristic zero and reduction to prime characteristic}\label{SubSecrelative}
%%%%%%%%%%%%%%%%%%%%%%%%%

We want to compare the differential signature in an algebra over a field of characteristic zero with the differential signature of reductions modulo a prime number and hence also with 
the $F$-signature of the reductions. We fix the following situation.

\begin{setup}
	\label{relativesetup}
Let $K$ be  a field of characteristic zero,  $X$ a scheme of finite type over $K$, and  $x$ a $K$-point in $X$. There exists a finitely generated $\ZZ$-subalgebra $A\subseteq K$ together with a scheme $X_A$ of finite type over $A$ and an $A$-point $x_A$ of $ X_A$ such that
$(X_A,x_A)\times_{\Spec A} K \cong (X,x)$. We can assume that $X_A$ and $x_A$ are flat over $\Spec A$ by generic freeness. For a closed point $ s \in \Spec A$ let $X_s$ denote the fiber of $X_A$ over $s$, which is a scheme of finite type over the residue field $\kappa(s)$ of $s$. We observe that $\kappa(s)$ is finite, and so $\kappa(s)$ and $X_s$ are $F$-finite. We take $x_s \in X_s$ to be the fiber of $x \in X$ over $s$. Under these conditions we say that  {$(X_A,x_A)$ is a \emph{model}} of $x\in X$.
\index{model}

 {
If $X=\Spec R$, 
$X_A$ and $X_s$ are affine. 
We denote the corresponding rings by $R_A$ and $R_s$ (or just $R_s$).} Moreover, $Q$  denotes the quotient field of $A$ and $R_Q$ the ring over $Q$. If $R$ is the localization of an algebra of finite type, then we can also find $A$ and a prime ideal $\m_A$ of $R_A$ which extends to the maximal ideal of $R$.
\end{setup}

We want to compare the free ranks of $P^n_{R_K|K} $, $P^n_{R_Q|Q} $, $P^n_{R_A|A} $, and $P^n_{R_{\kappa(s)} |\kappa(s)} $, and the differential signatures of these algebras. The basis for such comparison follows from  the fact that 
$P^n_{R'|A'} \cong P^n_{R|A} \tensor_R R' $ for any base change $A \rightarrow A'$ and $R'=R \otimes_AA'$  \cite[Proposition~16.4.5]{EGAIV}.

The following lemma implies that we can choose $A$ such that the free rank of $P^n_{R_K|K} $ equals the free rank of $P^n_{R_Q|Q} $.

\begin{lemma}
\label{freerankflatextend}
Let $(R,\m,\kk)$ and $(R',\m',\kk')$ be local rings and let $R \rightarrow R'$ be a flat ring homomorphism such that $\m R'=\m'$. Let $M$ be a finitely generated $R$-module and $M'=M \otimes_RR'$. Then
\[\frk_R M = \frk_{R'}M'  . \]
\end{lemma}
\begin{proof}
We look at the short exact sequence
\[0 \longrightarrow \Hom_R(M,\m) \longrightarrow   \Hom_R(M,R)  \longrightarrow N \longrightarrow  0 , \]
where the  {$\kk$-dimension} of $N$ gives the free rank of $M$ by  {Proposition~\ref{freerank-conditions}~(2).} We tensor with $R'$ and get by flatness
\[0 \longrightarrow \Hom_R(M,\m) \otimes_RR' \longrightarrow   \Hom_R(M,R)  \otimes_RR' \longrightarrow N \otimes_RR' \longrightarrow  0 . \]
By  the assumptions we have $\Hom_R(M,R) \otimes_RR' \cong \Hom_{R'} (M', R' ) $  \cite[Proposition 2.10]{Eisenbud}  and 
\[ \Hom_R(M,\m) \otimes_RR' \cong \Hom_R (M', \m \otimes_R R' )  \cong \Hom_{R'} (M',  \m R' )  \cong \Hom_{R'} (M', \m' )  . \]
So the quotient of $ \Hom_{R' }(M', \m' ) \rightarrow \Hom_{R'} (M', R' ) $ is $N \otimes_RR'$. Suppose that 
 {$N\cong \kk^r$}. Then,
\[ N \otimes_RR' = (R/\m)^r \otimes_RR' = (R/\m \otimes_RR' )^r = (R'/\m R')^r =\kk'^r .\]\qedhere
\end{proof}

\begin{lemma}
	\label{freerankrelative}
	Let $R$ be a local $K$-algebra essentially of finite type over a field $K$ of characteristic $0$ and let $R_A$ be a model for $R=R_K$ in the sense of Setup~\ref{relativesetup}. Let $M_K$ be a finitely generated $R_K$-module and $M_A$ be a finitely generated $R_A$-module with $M_A \otimes_A K=M_K$.
	Then for all points $s \in \Spec A$ in an open nonempty subset we have
	\[\frk M_K  =  \frk M_{{\kappa}(s)}  \, \]
\end{lemma}
\begin{proof}
	Let $Q=Q(A)$. The free ranks of  {$M_K$  and $M_Q$} coincide by Lemma~\ref{freerankflatextend}. Note that the condition $\m_Q R_K = \m_K$ can be obtained by enlarging $A$. By further enlarging $A$ the short exact sequence
	\[0 \longrightarrow \Hom_{R_Q} (M_Q,\m_Q) \longrightarrow \Hom_{R_Q} (M_Q,R_Q) \longrightarrow N \cong Q^r \longrightarrow 0 , \]
	which describes the free rank of $M_Q$, descends to a short exact sequence
	\[0 \longrightarrow \Hom_{R_A} (M_A,\m_A) \longrightarrow \Hom_{R_A} (M_A,R_A) \longrightarrow N_A \cong A^r \longrightarrow 0 \,  \]
	of $R_A$-modules. 	We tensor this sequence with $\otimes_A \kappa(s)$ (which is the same as $\otimes_{R_A} R_{\kappa(s)}$) and get a short exact sequence
	\[0 \longrightarrow \Hom_{R_A} (M_A,\m_A)\otimes_A \kappa(s) \longrightarrow \Hom_{R_A} (M_A,R_A)\otimes_A \kappa(s) \longrightarrow N_{\kappa(s)}  \cong \kappa(s)^r \longrightarrow 0 ,\]
	since $N_A \cong A^r $ is a flat $A$-module (see \cite[Observation 2.1.5]{HHCharZero}).
	Moreover, we have \[\Hom_{R_A} (M_A, L_A) \otimes_A \kappa(s) \cong \Hom_{R_{\kappa(s)} } (M_{\kappa(s)}, L_{\kappa(s)}) \]
	for $L_A=R_A,\m_A$   after enlarging $A$ again \cite[Theorem~2.3.5 (e)]{HHCharZero}). 
	Then,  the tensored sequence is the sequence which computes the free rank of $M_{\kappa(s)}$ to be $r$.
\end{proof}

\begin{corollary}
	\label{freerankprincipalrelative}
	Let $R$ be a local ring essentially of finite type over a field $K$ of characteristic $0$ and let $R_A$ be a model for $R=R_K$ in the sense of Setup~\ref{relativesetup}.
	Then for all points $s \in \Spec A$  {and for any fixed $n$}, in an open nonempty subset we have
	\[\frk P^n_{R_K|K}  = \frk P^n_{R_s|{\kappa}(s)}  \, \]
\end{corollary}
\begin{proof}
	We have $  P^n_{R_A|A} \otimes_{R_A} R_K   \cong  P^n_{R_Q|Q} \otimes_{R_Q} R_K \cong    P^n_{R_K|K} $ and  $P^n_{R_A|A} \otimes_{R_A} R_{\kappa(s)} \cong P^n_{R_{\kappa(s)}|\kappa(s) }$  \cite[Proposition 16.4.5]{EGAIV}. 
	So, this follows from Lemma~\ref{freerankrelative}.	
\end{proof}

If $A$ has dimension $1$, then the existence of the open subset means that the statement is true for all sufficiently large prime characteristics. In this situation we have
\[  \dm{\Q}(R_\Q) = \lim_{n \rightarrow \infty } \left( \lim_{p \rightarrow \infty} \frac{ \frk(P^n_{R_{\kappa(p)}|\kappa(p) } ) }{  n^{d}  /d! }   \right) .  \]

\begin{remark}
In Corollary~\ref{freerankprincipalrelative} the elements $f$ which describe the shrinking to $D(f)$ depend on $n$. It is not clear in general whether there exists an $f$ which works for all $n$. However, Lemma~\ref{JacobiTaylorperiodic} provides an instance where one unitary operator produces many unitary operators, so by extending the operator over $A_f$ extends all its companions as well.

In addition, if $R_K$ has positive differential signature, and if we can find $A \subseteq K$ such that also $R_A$ has positive differential signature, then we get that almost all reductions $R_s$ have positive differential signature bounded from below by $\dm{A}(R_A) $, since the estimate $\frk P^n_{R_{\kappa(s)|\kappa(s)} } \geq \frk P^n_{R_A|A} $ holds without further shrinking. Hence, under the assumption that the reductions are $F$-pure, they also have positive $F$-signature by Lemma~\ref{LemmaCofinalDif} with a common bound from below.
\end{remark}

We now present a corollary that gives an instance of how the differential signature in characteristic zero affects the behavior in varying positive characteristic.
\begin{corollary}	
\label{freerankprincipalrelativesequence}	
Suppose that $R_\ZZ$ is a generically flat $\ZZ$-algebra essentially of finite type of relative dimension $d$ such that $R_\QQ$ is local. Then there exists a sequence of prime numbers $p_n$, $n \in \NN$, such that
 {\[  \dm{\Q}(R_\Q) = \limsup_{n \rightarrow \infty } \frac{ \frk(P^n_{R_{\kappa(p_n)}|\kappa(p_n) } ) }{  n^{d}  /d! }     \]}
\end{corollary}
\begin{proof}
This follows from Corollary	\ref{freerankprincipalrelative}.
\end{proof}

\begin{definition}
We retain the situation described in Setup~\ref{relativesetup}.
	
We say that $x \in X$ is of {\it strongly $F$-regular type } (resp.\,\,{\it $F$-pure type}) if there exists a model of $x\in X$ over a finitely generated $\ZZ$-subalgebra $A$ of $K$ and a dense open subset  $U\subseteq \Spec A$ such that $x_s\in X_s$ is strongly $F$-regular (resp.\,\,$F$-pure) for all closed points $s\in U$.

We say that $x\in X$ is of {\it dense strongly $F$-regular type} or {\it dense $F$-pure type} if there exists a model and a dense (not necessarily open) set as above.
\end{definition}

We note that the previous definitions do not depend of the choice of the model (see \cite[Chapter~2]{HHCharZero} and \cite[Section~3.2]{MustataSrinivas})

Hara showed that strongly  $F$-regular type is equivalent to log-terminal singularities \cite[Theorem~5.2]{HaraSing} (see also \cite{KarenFrational}). Hara and Watanabe extended this result to dense strongly $F$-regular type \cite[Theorem~3.3]{HaraWatanabe}.

Aberbach and Leuschke \cite[Theorem 0.2]{AL}  established that $R$ is strongly $F$-regular if and only if $s(R)>0$. The following result gives a partial analogue for the differential signature in characteristic zero.

We recall that the \emph{anticanonical cover}\index{anticanonical cover} of a normal local ring is the symbolic Rees algebra of the inverse of the canonical module (in the class group of $R$). The condition that the anticanonical cover of $R$ is finitely generated (as an $R$-algebra) is a weakening of the condition that $R$ is $\QQ$-Gorenstein.

\begin{theorem}\label{ThmKLTPos}
Let $K$ be  a field of characteristic zero and  $X$ be a scheme of finite type over $K$.
Let $x\in X$ be a normal singularity defined over a field of characteristic zero $K$. Let $R=\cO_{X,x}$ be the germ of functions at $x$ and $\m$ its maximal ideal.
If $x\in X$ is of dense $F$-pure type, the anticanonical cover of $R$ is finitely generated, and $\dm{K}(R)>0$, then $x\in X$ is log-terminal.
\end{theorem}

\begin{proof}
Since  $x\in X$ is of  dense $F$-pure type   there exists a model of $x\in X$ over a finitely generated $\ZZ$-subalgebra $A$ of $K$ and a dense subset of closed points $W\subseteq \Spec A$ such that $x_s\in X_s$ is $F$-pure for all $s\in W$.

Since $\dm{K}(R)>0$,  $R$ is $D$-simple by Theorem~\ref{ThmDifMultDsimple}.
Then,
 {$ R_s:=\cO_{X_s,x_s}=(R_A)_\eta\otimes_A A/\m_S$} is a simple  $D_{R_s| \kappa(S)}$-module for every closed point $s\in U$, where $U$ is a dense open subset $U\subseteq \Spec A$ \cite[Theorem 5.2.1]{SmithVDB}.
Since $\kappa(s)$ is $F$-finite,
$R_s$ is strongly  $F$-regular for every $s\in U\cap W$ \cite[Theorem 2.2]{DModFSplit}. 
Since $U\cap W$ is dense, $x\in X$ is of dense strongly $F$-regular type. Hence, $x\in X$ is log-terminal \cite[Theorem~D]{CEMS}.
\end{proof}

	We recall a conjecture that relates $F$-purity with log-canonical singularities.
	Let $K$ be  a field of characteristic zero and $X$ be a scheme of finite type over $K$.
	Let $x\in X$ be an $n$-dimensional normal $\QQ$-Gorenstein singularity. Then $x\in X$ is log canonical if and only if it is of dense $F$-pure type.

	The direction ``$F$-pure implies log-canonical'' of this conjecture is already known \cite{HaraWatanabe}. There has been intense research regarding the other direction;  {see \cite{FujTak,TakagiANT,Her} for some positive results.} Assuming this conjecture, Theorem~\ref{ThmKLTPos} states that if $\dm{K}(R)>0$ and $x\in X$ is log-canonical, then $x\in X$ is log-terminal.

%%%%%%%%%%%%%%%%%%%%%%%%%%%%%%%%%%%%%%%%%%
%%%%%%%%%%%%%%%%%%%%%%%%%%%%%%%%%%%%%%%%%%
\section{Extending and restricting operators}
%%%%%%%%%%%%%%%%%%%%%%%%%%%%%%%%%%%%%%%%%%
%%%%%%%%%%%%%%%%%%%%%%%%%%%%%%%%%%%%%%%%%%

 We devote this section to establish concepts and results that allow us to compute the differential signature for  several examples. 
 
 \subsection{Definitions and examples}

\begin{definition}\label{def:extensible}$ $
	\begin{enumerate}
\item An inclusion of $A$-algebras $R \subseteq S$ is \emph{differentially extensible}\index{differentially extensible} over $A$ if for every $\delta \in D_{R|A}$ there exists an element $\tilde{\delta}\in D_{S|A}$ such that $\tilde{\delta}|_{R} = \delta$.
\item We say $R\subseteq S$ as above is \emph{differentially extensible with respect to the order filtration} or \emph{order-differentially extensible} if for  {every $n\in\NN$ and every}  $\delta \in D^n_{R|A}$ there exists an element $\tilde{\delta}\in D^n_{S|A}$ such that $\tilde{\delta}|_{R} = \delta$.
\item  If $A$ has characteristic $p>0$, we say $R\subseteq S$  is \emph{differentially extensible with respect to the level filtration} or \emph{level-differentially extensible} if  {for every $e\in\NN$ and every $\delta \in D^{(e)}_{R|A}$} there exists an element $\tilde{\delta}\in D^{(e)}_{S|A}$ such that $\tilde{\delta}|_{R} = \delta$.
\end{enumerate}
\end{definition}

Even though these notions are subtle, it includes many interesting examples, e.g., the inclusions of many classical invariant rings in their ambient polynomial rings. We discuss these examples further in this and the next section.

Clearly, if $R\subseteq S$ is differentially extensible with respect to the order or the level filtration, it is differentially extensible. However, there are inclusions that are level-differentially extensible but not order-differentially extensible.

\begin{example} Let $R=\FF_p[xy] \subseteq S=\FF_p[x,y]$. We have that $D^{(e)}_R=\End_{R^q}(R)$ and $D^{(e)}_S=\End_{S^q}(S)$. Since $R$ is free over $R^q$, any map in $D^{(e)}_R$ corresponds to a choice of images for $1, xy, \dots, (xy)^{q-1}$. As these elements form part of a free basis for $S$ over $S^q$, one can extend the operator to an element of $D^{(e)}_S$ by choosing arbitrary values for the rest of the free basis. However, the map $\delta\in D^{p-1}_{R|\FF_p}$ that sends $xy \mapsto 1$ and the rest of the free basis to zero does not extend to an element of $D^{p-1}_{S|\FF_p}$, since any such operator decreases degrees by at most $p-1$.
	
	We also note that in characteristic zero, the analogous inclusion $R=\CC[xy] \subseteq S=\CC[x,y]$ is not differentially extensible. Indeed, the derivation $\theta=\frac{d}{d(xy)}$ on $R$ does not extend to a differential operator of any order on $S$. To see this, observe that any extension of $\theta$ must be of the form  {$\theta=\sum_{a,b\geq 1} c_{a,b} x^{a-1} y^{b-1} \frac{1}{a!b!} \partial^{(a,b)}$ for some constants $c_{a,b}\in \CC$, with $c_{a,b}=0$ for all but finitely many pairs $a,b$. Plugging in $(xy)^n$ and extracting the $(xy)^{n-1}$ coefficient yields the equality $\sum_{a,b\geq 1} c_{a,b} \binom{n}{a}\binom{n}{b} =n$ for all $n$. But, for fixed integers $a,b\geq 1$, we have $n \mid \binom{n}{a}$ and $n\mid \binom{n}{b}$ in $\CC[n]$, so} the left-hand side is a polynomial divisible by $n^2$, so that the equation above for fixed constants yields a nonzero polynomial with roots for all integers $n$, a contradiction.
\end{example}

The notion of differential extensibility has been studied earlier in the literature, though not under this name
\cite{LS,Schwarz}.

\begin{remark} By Proposition~\ref{localization1}, localization maps are order-differentially extensible over a field $K$.
\end{remark}

Part of the following lemma is well-known; see, e.g., {\cite[Proposition~3.2]{Knop}}.
\begin{proposition}\label{prop-extensible-finite}
	Let $R$ and $S$ be algebras essentially of finite type over a field $K$. Suppose that both $R$ and $S$ are normal, and that $S$ is a module-finite extension of $R$, \'etale in codimension one, and split as $R$-modules. Then the inclusion of $R$ into $S$ is order-differentially extensible over $K$. If $K$ has characteristic $p>0$, then the inclusion of $R$ into $S$ is also level-differentially extensible over $K$. 
\end{proposition}
\begin{proof}
	By Lemma~\ref{etalemap} and the hypotheses, there is an $S$-module homomorphism $\alpha:S\otimes_R \ModDif{n}{R}{K} \to  \ModDif{n}{S}{K}$ that is an isomorphism in codimension one. Then, we obtain a map
	 {
	\[\Hom_R(\ModDif{n}{R}{K},S)\cong\Hom_S(S\otimes_R \ModDif{n}{R}{K},S) \leftarrow \Hom_S(\ModDif{n}{S}{K},S) \cong D^n_{S|K} \]
	}
	that is an isomorphism in codimension one. We observe that this map agrees with restriction of functions. We obtain that $D^n_{R|K}(S) \cong D^n_{S|K}$ \cite[{Lemma 0AV6 and Lemma 0AV9}]{stacks-project}. Since $R$ is a a direct summand of $S$ as an $R$-module, we find that $D^n_{R|K} \subseteq D^n_{R|K}(S)$.
	
	The statement about level-extensibility in positive characteristic follows in the same way.
\end{proof}

Much of the literature in terms of extending differential operators is in the context of invariant rings. If $S$ is a $K$-algebra with an action of a linearly reductive group $G$, any $G$-invariant differential operator on $S$ yields a differential operator on $S^G$. The question of whether this restriction homomorphism $\pi: (D_{S|K})^G \rightarrow D_{S^G|K}$ is surjective, and whether it is surjective for each filtered piece $\pi_n: (D^n_{S|K})^G \rightarrow D^n_{S^G|K}$ has been studied (see~\cite{LS, Schwarz} among others) in connection to the question of when rings of differential operators on invariant rings are simple algebras. For example, one has the following result.

\begin{theorem}[\cite{Musson}]\label{Musson}
Let $K$ be an algebraically closed field of characteristic zero, and $\Lambda$ be a semigroup of the form $L\cap \NN^d$ for some linear space $L\subseteq \RR^d$. Then the inclusion of $K[\Lambda]\subseteq K[\NN^d]$ is order-differentially extensible if and only if the following hold.
\begin{itemize}
\item[(i)] The spaces $L\cap \{x_i \geq 0\}$ are distinct for distinct $i$.
\item[(ii)] The spaces $L\cap \{x_i = 0\}$ are facets of $L\cap \RR_{\geq 0}^d$.
\item[(iii)] The image of $\Lambda$ under each coordinate function generates $\ZZ$ as a group.
\end{itemize}
\end{theorem}

For contrast, we note the following.

\begin{proposition}\label{charp-toric-extend} Let $K$ be a field of characteristic $p>0$, and $\Lambda$ be a semigroup of the form $L\cap \NN^d$ for some linear space $L\subseteq \RR^d$. Then $K[\Lambda]\subseteq K[\NN^d]$ is level-differentially extensible.
\end{proposition}
\begin{proof}
Note first that, from the $\NN^d$-grading,  $D^{(e)}_{K[\Lambda]}$ is generated by maps that send monomials to monomials, so it suffices to show that such a map extends. Let $\phi: K[\Lambda] \rightarrow K[\Lambda]$ be a $K[\Lambda]^q$-linear map that sends monomials to monomials. Define a map $\tilde{\phi}:K[\NN^d]\rightarrow K[\NN^d]$ as follows. For a monomial $x^\alpha$, set $\tilde{\phi}(x^\alpha)=x^{q\gamma}\phi(x^\beta)$ if $\alpha$ can be written as $\beta + q \gamma$ with $\beta \in \Lambda$ and $\gamma\in \NN^d$; set $\tilde{\phi}(x^\alpha)=0$; otherwise, extend by $K$-linearity. To see that this map is well-defined, write $\alpha=\beta+q\gamma=\beta' + q\gamma'$ for $\beta, \beta'\in \Lambda$ and $\alpha,\gamma,\gamma'\in\NN^d$. Then, $(q-1)\beta + \beta' = q(\beta + \gamma-\gamma')\in \Lambda \cap q\ZZ^d = q\Lambda$. Thus, 
\[	
\frac{x^{q\gamma}\phi(x^\beta) }{ x^{q\gamma'} \phi(x^{\beta'})} = \frac{x^{q(\beta + \gamma - \gamma')}\phi(x^{\beta}) }{ x^{q\beta} \phi(x^{\beta'})}  =\frac{ \phi(x^{q(\beta + \gamma - \gamma')+\beta})}{\phi(x^{q\beta +\beta'})} = \frac{
	\phi(x^{(q-1)\beta + \beta' +\beta}) }{\phi(x^{q\beta +\beta'})}=1.
\]
It follows from the definition that $\tilde{\phi}$ is $K[\NN^d]^q$--linear and agrees with $\phi$ on $K[\Lambda]$. 
\end{proof}

We pose the following dual condition to differential extensibility. We find it useful to pose the following in a bit more generality than the setting of Definition~\ref{def:extensible}.

\begin{definition}
	Let $R$ be an $A$-algebra, and $S$ be a $B$-algebra. Suppose we have a commutative diagram
	\[ \xymatrix{R \ar[r]^{\varphi} & S \\ A \ar[u]\ar[r] & B, \ar[u]}\]
	which we refer to as a map of algebras $(R,A)\to (S,B)$ for short. A map of algebras is \emph{differentially retractable}\index{differentially retractable} if for every $\delta \in D_{S|B}$ there exists an element $\hat{\delta}\in D_{R|A}$ such that $\delta \circ \varphi = \varphi \circ \hat{\delta}$; it is \emph{order-differentially retractable} if we require the order of $\hat{\delta}$ to be no greater than that of $\delta$.
\end{definition}

The following fact was used systematically by \`Alvarez-Montaner, Huneke, and the third author \cite{AMHNB} to obtain results on local cohomology and Bernstein-Sato polynomials of direct summands of polynomial rings.

\begin{proposition}[{\cite[Lemma~3.1]{AMHNB}}]
	Let $R\subseteq S$ be an inclusion of $A$-algebras such that $R$ is a ($R$-linear) direct summand of $S$. Then the inclusion $(R,A)\to (S,A)$ is order-differentially retractable.
\end{proposition}

\begin{lemma}
	Let $R$ be a polynomial ring over a field $K$, and $S=R/I$ for some ideal $I$. Then the surjection $(R,K)\to (S,K)$ is order-differentially retractable.
\end{lemma}
\begin{proof} This is a well-known fact; see \cite[Section 15.5.6]{MCR}
\end{proof}

\subsection{Applications to Bernstein-Sato polynomials and differential powers}

The property of differential extensibility has interesting consequences for differential invariants of rings. For example, this notion has strong implications for Bernstein-Sato polynomials. We first recall the following definition.

\begin{definition}\label{def-BS-poly}
	Let $R$ be a ring, $A$ a subring, and $f$ be an element of $R$. We say that the polynomial $b^{R|A}_{f}(s)\in A[s]$ is the \emph{Bernstein-Sato polynomial of $f$}\index{Bernstein-Sato polynomial}\index{$b^{R \vert A}_{f}(s)$}  {if $b_f$ is monic,} there exists a differential operator $\delta\in D_{R|A}$ such that, for $b(t)=b^{R|A}_{f}(t)$,
		\begin{equation}\tag{$\dagger$}\label{BS-equation} \delta\cdot f^{t+1}=b(t) f^t \quad \text{ for all } \quad t\in \ZZ,
		\end{equation}
 and for any pair $\delta$ and $b$ that satisfy~(\ref{BS-equation}), $b^{R|A}_{f}(s)$ divides $b$. If there do not exist $\delta$ and  {nonzero} $b$ that satisfy~(\ref{BS-equation}), we say that $f$ has no Bernstein-Sato polynomial.
\end{definition}

It is well-known that Bernstein-Sato polynomials exist for every element of $S$ when $A=K$ is a field of characteristic zero and $S=K[x_1,\dots,x_n]$ \cite{BernsteinPoly,SatoPoly}. Recently, it was shown that when $R$ is a direct summand of a polynomial ring $S$, then $b^{R|K}_{f}(s)$ exists and $b^{R|K}_{f}(s)\;\big|\;b^{S|K}_{f}(s)$ for any $f\in R$, and examples are given where  {$b^{R|K}_{f}(s)\neq b^{S|K}_{f}(s)$ \cite[Example~3.17]{AMHNB}}. The following result establishes a case where the equality is reached.

\begin{theorem}\label{BS-Thm} Let $A\subseteq R\subseteq S$ be rings, $f$ be an element of $R$, and suppose that $R$ is a direct summand of $S$ that is differentially extensible over $A$. If $b^{S|A}_{f}(s)$ exists, then $b^{R|A}_{f}(s)$ exists and $b^{R|A}_{f}(s)=b^{S|A}_{f}(s)$.
\end{theorem}
\begin{proof} We know that the Bernstein-Sato polynomial $b^{R|A}_{f}(s)$ exists and $b^{R|A}_{f}(s)\;\big|\;b^{S|A}_{f}(s)$ \cite[Theorem~3.14]{AMHNB}, so it suffices to show that $b^{S|A}_{f}(s)\;\big|\;b^{R|A}_{f}(s)$. Choose a differential operator $\delta\in D_{R|A}$ that satisfies the equation~(\ref{BS-equation}) for $b=b^{R|A}_{f}$ as an equation in $R$. Replacing $\delta$ by an extension to $D_{S|A}$, we may view this as an equation in $S$. It then follows from the defintion of $b^{S|A}_{f}(s)$ that $b^{S|A}_{f}(s)\;\big|\;b^{R|A}_{f}(s)$.
	\end{proof}

\begin{remark}\label{Rem-com-BS}
Let $K$ be a field of characteristic zero,  $S=K[x_1,\ldots,x_d]$, and $R\subseteq S$ be a $K$-subalgebra that is a direct summand of $S$. Suppose that the inclusion of $R$ into $S$ is differentially extensible over $K$. Given the previous result,  one can compute $b^{R|K}_f(s)$ by using the existing methods \cite{Noro,Oaku} to compute $b^{S|K}_f(s)$.
\end{remark}

The following propositions show the utility of the notion of extensibility in computing differential powers.

\begin{proposition}\label{behave-diff-powers} Let $A \subseteq R \subseteq S$ be rings. Let $I\subseteq R$ and $J\subseteq S$ be ideals.
\begin{enumerate}
\item[(i)] For any $R$-linear map $\pi:S\rightarrow R$, if $\pi^{-1}(I)\subseteq J$, one has $I\dif{n}{A} \subseteq J\dif{n}{A} \cap R$.
\item[(ii)] If the inclusion of $R$ into $S$ is order-differentially extensible and $I \supseteq J\cap R$, then $I\dif{n}{A} \supseteq J\dif{n}{A} \cap R$.
\end{enumerate}
\end{proposition}
\begin{proof}
\begin{enumerate}
\item[(i)] We show that if $f \in R\setminus J\dif{n}{A}$, then $f \notin I\dif{n}{A}$. Suppose that $f\in R$ is not in $J\dif{n}{A}$. Then there is a differential operator $\delta\in D^{n-1}_{S|A}$ such that $\delta(f)\notin J$. By precomposing with the inclusion and postcomposing with $\pi$, one gets a differential operator $\delta' \in D^{n-1}_{R|A}$; see \cite[Proof of Proposition 3.1]{DModFSplit}. Then, since $\delta'(f)=\pi(\delta(f))\notin I$, it follows that $f \notin I\dif{n}{A}$.
\item[(ii)] Suppose that $f \notin I \dif{n}{A}$. Then, there is a differential operator $\delta\in D^{n-1}_{R|A}$ with $\delta(f)\notin I$. Since the inclusion is order-differentially extensible over $A$, there is a differential operator $\delta'\in D^{n-1}_{S|A}$ with $\delta'(f)\in (R\setminus I) \subseteq (S\setminus J)$. Thus, $f\notin J\dif{n}{A}$.\qedhere
\end{enumerate}
\end{proof}

\begin{proposition}\label{intersect-maximal} Let $A \subseteq R \subseteq S$ be rings, with $(R,\m)$ and $(S,\n)$ local. Suppose that the inclusion of $R$ into $S$ is local, admits an $R$-linear splitting, and is order-differentially extensible over $A$. Then, $\m\dif{n}{A}=\n\dif{n}{A}\cap R$. If one removes the hypothesis that the inclusion is order-differentially extensible, then one still has that $\m\dif{n}{A}\subseteq \n\dif{n}{A}\cap R$.
\end{proposition}
\begin{proof}
If $\pi$ is a splitting of the inclusion and the inclusion is differentially extensible, the hypotheses of both parts of the previous proposition are satisfied for $I=\m$, $J=\n$. For the second statement, the first part of the previous proposition applies.
\end{proof}

Since direct summands of regular rings are strongly $F$-regular \cite{HHStrongFreg}, they also have positive $F$-signature \cite{AL}. We show a similar result for differential signature.

\begin{theorem}\label{ThmDirSumPos} Let $(R,\m,K)$ be an algebra with pseudocoefficient field $K$ that is a direct summand of a regular ring $(S,\n,K)$ that is an algebra with pseudocoefficient field $K$. Then $\dm{K}(R)>0$.
\end{theorem}
\begin{proof} By Proposition~\ref{intersect-maximal}, we have that $\m\dif{n}{K} \subseteq \n^n\cap R$. There is a constant $C$ such that $\n^n \cap R \subseteq \m^{\lfloor n/C \rfloor}$ \cite[Theorem~1]{Hubl-completions}. Then, setting $d=\dim(R)$,
\begin{align*}
\dm{K}(R)&=\limsup_{n\rightarrow\infty}\frac{\lambda_R(R/\m\dif{n}{K})}{n^d / d!}\\
&  \geq \limsup_{n\rightarrow\infty}\frac{\lambda_R(R/\m^{\lfloor n/C \rfloor})}{n^d / d!}\\
&  = \limsup_{n\rightarrow\infty}\frac{\lambda_R(R/\m^n)}{(Cn)^d / d!} = \frac{e(R)}{C^d},
\end{align*}
which is positive, as required.
\end{proof}

We conclude this section by noting that differential signature has a very simple interpretation for order-differentially extensible direct summands of polynomial rings.

\begin{theorem}\label{COMPUT-prop}
	Let $S$ be a standard graded polynomial ring over a field $K$, and $R$ be a graded subring such that $R\rightarrow S$ is split and differentially extensible. Then, the differential powers of $\m =R_+$ form a graded system and  $R\cong \bigoplus_{n\in \NN} \frac{\m\dif{n}{K}}{\m\dif{n+1}{K}}$. In particular, the gradings agree. Consequently, the differential signature of $R$ converges as a limit to a rational number, which is the degree of the graded ring $R$.
	
	If, in addition, $R$ is generated in a single degree $t$, then $\displaystyle \dm{K}(R)=\frac{e(R)}{t^{\dim(R)}}$.
\end{theorem}
\begin{proof}
	By Proposition~\ref{intersect-maximal}, we have that $(R_+)\dif{n}{K}=(S_+) \dif{n}{K}\cap R = [R]_{\geq n}$ for all $n$. Thus, $\cG_n=[R]_{\geq n} / [R]_{\geq n+1}$ for all $n$. Then, since $R$ is a Noetherian graded ring, its degree is a rational number.
	
	In the case $R$ is generated in a single degree $d$, the stated formula is the degree of $R$ as a graded ring.
\end{proof}

For the most general results on which inclusions of invariant rings are differentially extensible, we refer the reader to the work of Schwarz \cite{Schwarz}.

%%%%%%%%%%%%%%%%%%%%%%%%%%%%%%%%%%%%%%%
%%%%%%%%%%%%%%%%%%%%%%%%%%%%%%%%%%%%%%%
\section{Differential signature of certain invariant rings}
%%%%%%%%%%%%%%%%%%%%%%%%%%%%%%%%%%%%%%%
%%%%%%%%%%%%%%%%%%%%%%%%%%%%%%%%%%%%%%%

In this section, we apply Proposition~\ref{COMPUT-prop} to compute differential signature for certain rings of invariants.

\subsection{Formula for invariant rings under the action of a finite group}

We now compute the differential signature for rings associated to finite groups.
In this case, we obtain that $\dm{K}(R^G)=s(R^G).$

\begin{theorem}\label{group-formula} Let $G$ be a finite group such that $|G| \in K^\times$, and let $G$ act linearly on a polynomial ring $R$ over $K$. Suppose that $G$ contains no elements that fix a hyperplane in the space of one-forms $[R]_1$. Then $\dm{K}(R^G)=1/|G|$.
\end{theorem}
\begin{proof} The ramification locus of a finite group action corresponds to the union of fixed spaces of elements of $G$. Consequently, the assumption that no element fixes a hyperplane ensures that the extension is unramified in codimension one. The inclusion is order-differentially extensible over $K$ by Proposition~\ref{prop-extensible-finite}. Furthermore, since $|G| \in K^\times$, there exists a Reynolds operator that splits the inclusion map. By Theorem~\ref{COMPUT-prop}, the differential signature is just the degree, which, by Molien's formula, is
\[ \dm{K}(R^G)=\limsup_{n \rightarrow \infty} \frac{d!}{n^d}\; \lambda_{R^G}\left(  {R^G}/{[R^G]_{\geq n}}\right) = \frac{1}{|G|}. \]
	\qedhere
\end{proof}

\subsection{Volume formulas for toric rings}\label{SubSecToric}

We also obtain a formula for pointed normal affine semigroup rings. We  simply refer to such rings as \emph{toric} rings\index{toric ring}, although different authors mean different things by this term.

\begin{setup}\label{toric-setup}
Let $C$ be a pointed rational cone in $\RR^d$. That is, the rays of $C$ each contain a nonzero point in $\ZZ^d$, and $C$ contains no lines.
Let $K$ be a field, and $R$ be the semigroup ring $K[\ZZ^d \cap C]\subseteq K[\ZZ^d]=K[x_1^{\pm1} , \dots, x_d^{\pm 1}]$; that is, the subring of the  {Laurent polynomials} with $K$-basis given by monomials whose exponents lie in $\ZZ^d \cap C$.
Given such a realization, let $F_1,\dots,F_r$ be the facets. For each facet $F_i$ there exists a unique linear form $\ell_i$ with $F_i \subseteq \ker \ell_i$, with integral coprime coefficients and positive in the interior of the cone. Let $\ell =(\ell_1, \ldots, \ell_r)$ be the (injective) map from $\RR^d$ to $\RR^r$. Note that $\ell(C)=\RR_{\geq 0}^r\cap \ell(\RR^d)$. This gives an embedding of $R$ into $S=K[\NN^r]=K[y_1,\dots,y_r]$ and $R$ is the degree-$0$-part of $ K[y_1,\dots,y_r]$ under the grading given by the group $\ZZ^r/\ell(\ZZ^d)$.
\end{setup}

The grading group $\ZZ^r/\ell(\ZZ^d)$ is known to be the divisor class group of the monoid ring. We also use the standard-grading of $S=K[y_1,\dots,y_r]$, and a monomial $x^\mu$ will have in $S$ the degree  $ |\ell(\mu)|=\ell_1(\mu) + \cdots + \ell_r(\mu) $. The linear form $|\ell|=\ell_1 + \ldots+ \ell_r$ defines in $\RR^d$ the compact polytope
\[ \{P \in C\ |\ (\ell_1 + \cdots+ \ell_r ) (P) \leq 1 \} .\]
Its volume is directly related to the differential signature of the normal monoid ring $K[\ZZ^d \cap C]$.

\begin{lemma}[{\cite[Proof of Theorem~3.4]{Hsiao-Matusevich}},\cite{Musson}]\label{toric-extl} If $K$ is algebraically closed of characteristic zero, the inclusion of $R$ into $S$ is order-differentially extensible over $K$.
\end{lemma}

Hsiao and Matusevich gave a construction for an order-preserving lift of an arbitrary differential operator from $R$ to $S$ \cite[Proof of Theorem~3.4]{Hsiao-Matusevich}.

\begin{theorem}\label{ThmToric}
In the context of Setup~\ref{toric-setup}, if $K$ has characteristic zero, then $\dm{K}(R)= d! \ \vol  \{P \in C \ | \ (\ell_1 + \cdots+ \ell_r ) (P)  \leq 1 \}$.
\end{theorem}
\begin{proof} By extending the base field, we may assume that $K$ is algebraically closed. It follows from Lemma~\ref{toric-extl} that the inclusion of $R$ into $S$ is order-differentially extensible over $K$. Since $R$ is the degree-$0$-part under a grading (or the invariant ring by a linearly reductive group) of $S$, the inclusion map is split. We can then apply Theorem~\ref{COMPUT-prop} to get
\[ \dim_K(R/ (R_+)\dif{n+1}{K} ) = \dim_K(R/R_{\geq n+1})   = \# \{ \mu \in \ZZ^d \cap C\ |\ (\ell_1 + \cdots+ \ell_r )  (\mu ) \leq n \}            .     \]
This number, divided by $n^d$, converges to the volume of the polytope \[ \{P \in C|\, (\ell_1 + \cdots+ \ell_r ) (P)  \leq 1 \} .\qedhere\]
\end{proof}

We note that a simpler version of the argument for Theorem~\ref{ThmToric} can be used to obtain the description of the $F$-signature of toric varieties due to Watanabe--Yoshida, Von Korff \cite{WatanabeYoshida,VonKorff} (see also \cite{SinghSemigroup}). Our statement of their result, which is somewhat simpler than but closely parallels our description of differential signature, differs from that of the original sources. The key difference is that less specific realizations of toric rings as direct summands of polynomial rings suffice, so it is not necessary to appeal to Setup~\ref{toric-setup}.

\begin{theorem}[Watanabe--Yoshida, Von Korff, Singh]\label{WYVKS} Let $K$ be a field of positive characteristic, $L$ be a rational linear subspace of $\RR^r$, and $R=K[\NN^r \cap L] \subseteq S=K[\NN^r]$. Then $s(R)=\vol_L(\cube \cap L)$ where $\cube$ is the unit cube in $\RR^r$.
\end{theorem}
\begin{proof} The inclusion of $R$ into $S$ is split, so every element of $D^{(e)}_S$ restricts to an element of $D^{(e)}_R$ by composition with a retraction. By Proposition~\ref{charp-toric-extend}, every element of $D^{(e)}_R$ extends to an element of  $D^{(e)}_S$. The same proof as Proposition~\ref{intersect-maximal} shows that $\m^\Fdif{p^e}=\n^\Fdif{p^e}\cap R$, where $\m$ and $\n$ are the homogeneous maximal ideals of $R$ and $S$, respectively. As these ideals are generated by monomials, the length of $R/I_e=R/\m^\Fdif{p^e}$ may be computed as $\# (n\cube \cap L \cap \NN^r) = \# (\cube \cap \frac{1}{n}(L \cap \NN^r))$. The formula again follows by definition of the Riemann integral.
\end{proof}

This also shows that for a simplicial cone, i.e., a cone where the number of facets coincides with the dimension ($d=r$ in Setup~\ref{toric-setup}), the $F$-signature and the differential signature are the same. In fact, both are the inverse of the order of the finite group $D= \ZZ^d/\ell(\ZZ^d)$. This follows, since in the simplicial case the  determinant of the ray vectors normalized by the condition $|\ell|=1$ determines the volume of the relevant polytope as well as the order of this group. Since the $D$-grading of the polynomial ring corresponds to an action of the dual group, and the invariant ring is the degree-$0$-part, this also follows from Theorem~\ref{group-formula}.

We give some examples to illustrate how to use Theorem~\ref{ThmToric}.

\begin{example}
Let $R$ be the $d$-th Veronese subring of $K[x,y]$. One may realize this ring as $K[C \cap \ZZ^2]$, where $C$ is the cone bounded by the rays $\RR_{\geq 0} (1,0)$ and $\RR_{\geq 0} (1,d)$. Then, the linear forms are
$\ell_1=y$ and $\ell_2=dx-y$, and the generators of the monoid, namely $(1,0),(1,1), \ldots , (1,d-1), (1,d)$, are sent under $\ell$ to $(d-i,i)$, $i=0, \ldots, d$. The condition $dx=1$ determines the points $(\frac{1}{d},0)$ and $(\frac{1}{d},1 )$ on the rays and the area of the given triangle is $1/2d$, which gives $1/d$ when multiplied by $2$.
\end{example}

\begin{example} 
Let $R$ be the hypersurface given by the equation $ac - b^2d$. This can be realized as $K[C \cap \ZZ^3]$, where $C$ is the cone generated by $ (1,0,1), (1,0,0), (1,1,0)$ and $(0,1,1)$. The primitive linear forms for the facets of $C$ are $y,z, x-y+z $ and $x+y-z $. The sum of these linear forms is $2x+y+z$ and the condition $2x+y+z=1$ yields the points on the rays
$ ( \frac{1}{3},0,\frac{1}{3} ), (\frac{1}{2},0,0), (\frac{1}{3},\frac{1}{3},0)$ and $(0,\frac{1}{2},\frac{1}{2})$. 
We triangulate the polytope using these points and get $\dm{K}(R)=1/6$.
\end{example}

\subsection{Determinantal rings}\label{SubSecDet}
 We are now ready to compute the differential signature for rings obtained from matrices of variables. We point out that  the $F$-signature is not known in the following examples.
%\begin{setup}\label{determinantal} 
%\end{setup}

\begin{theorem}\label{ThmDet}
 Let $K$ be a field of characteristic zero, $Y=[y_{ij}]$ be an $m\times r$ matrix of indeterminates and $Z=[z_{ij}]$ be an $r\times n$ matrix of indeterminates, with $r \leq  m \leq n$. Let $X=YZ$ be the $m\times n$ matrix obtained as the product of $Y$ and $Z$. Let $S=\CC[Y,Z]$ be the polynomial ring in the entries of $Y$ and $Z$, and $R=K[X]$ be the subring of $S$ generated by the entries of $X$. We note that $R$ is isomorphic to a ring generated by an $m \times n$ matrix of indeterminates quotiented out by the ideal generated by the $(r+1) \times (r+1)$ minors.
Then,
% In the context of Setup~\ref{determinantal},
  \[\displaystyle\dm{K}(R)=\frac{1}{2^{r(m+n-r)}} \det\left[ \binom{m+n-i-j}{m-i} \right]_{i,j=1,\dots,r}.\]
\end{theorem}
\begin{proof} By base change, we can reduce to the case where $K=\CC$. Then,  the inclusion map of $R$ into $S$ is order-differentially extensible over $\CC$ \cite[Case~A, Main Theorem 0.3, 0.7]{LS}. As this is an invariant ring of a linearly reductive group, the inclusion splits. Theorem~\ref{COMPUT-prop} then applies. The ring $R$ is generated in degree 2, and the formulas for the dimension and multiplicity are classical {; see, e.g., \cite[Theorem~3.5]{HerzogTrung}.}
\end{proof}

\begin{example}
	\label{signaturesegre}
 As the special case where $r=1$, we obtain the analogue of Singh's formula \cite[Example~7]{SinghSemigroup} for the $F$-signature of the homogeneous coordinate ring for the Segre embedding of $\PP^{m-1} \times \PP^{n-1}$. The differential signature of this ring is $\displaystyle \frac{\binom{m+n-2}{m-1}}{2^{m+n-1}}$. For comparison, the $F$-signature is $\displaystyle \frac{A(m+n-1,n)}{n!}$, where $A$ denotes the Eulerian numbers.
\end{example}

%\begin{setup}\label{symmetric}  

%\end{setup}

\begin{theorem}
Let $K$ be a field of characteristic zero. $Y=[y_{ij}]$ be an $k\times n$ matrix of indeterminates, with $n > k$. Let $X=Y^t Y$, $S=\CC[Y]$ be the polynomial ring in the entries of $Y$, and $R=K[X]$ be the subring of $S$ generated by the entries of $X$. We note that $R$ is isomorphic to a ring generated by a $n \times n$ symmetric matrix of indeterminates quotiented out by the ideal generated by the $k+1\times k+1$ minors. Then,
% In the context of Setup~\ref{symmetric}, 
 \[\displaystyle\dm{K}(R)=\frac{1}{2^{kn-\binom{k}{2}}} \sum_{1\leq \ell_1 < \dots < \ell_k \leq n} \det\left[ \binom{n-i}{n-\ell_j}\right]_{i,j=1,\dots,k}.\]
\end{theorem}
\begin{proof} By base change, we can reduce to the case where $K=\CC$. Then, the inclusion map of $R$ into $S$ is order-differentially extensible over $\CC$ \cite[Case~B, Main Theorem 0.3, 0.7]{LS}. Since $R$ is the invariant ring of $S$ under a natural action of the orthogonal group  \cite[II~3.3]{LS},  the inclusion splits. Theorem~\ref{COMPUT-prop} then applies. The ring $R$ is generated in degree 2, and the multiplicity of these rings is known by the work of Conca  {\cite[Theorem~3.6]{ConcaSym}.}
\end{proof}

\begin{theorem}
Let $K$ be a field of characteristic zero. $Y=[y_{ij}]$ be an $2k\times n$ matrix of indeterminates, with $n > 2(k+1)$. Let $Q$ be the $2n \times 2n$ antisymmetric matrix $\begin{bmatrix} 0_{n \times n} & I_{n \times n} \\ -I_{n \times n} & 0_{n \times n} \end{bmatrix}$. Let $X=Y^t Q Y$, $S=\CC[Y]$ be the polynomial ring in the entries of $Y$, and $R=K[X]$ be the subring of $S$ generated by the entries of $X$. We note that $R$ is isomorphic to a ring generated by a $n \times n$ matrix of indeterminates quotiented out by the ideal generated by the $2(k+1) \times 2(k+1)$ minors. Then,
 %In the context of Setup~\ref{pfaffian}, 
 \[\displaystyle\dm{K}(R)=\frac{1}{2^{r(2n-2r-1)}} \det\left[ \binom{2n-4r-2}{n-2r-i+j-1}-\binom{2n-4r-2}{n-2r-i-j-1} \right]_{i,j=1,\dots,r}.\]
\end{theorem}
\begin{proof} By base change, we can reduce to the case where $K=\CC$. Then, the inclusion map of $R$ into $S$ is order-differentially extensible over $\CC$  \cite[Case~C, Main Theorem 0.3, 0.7]{LS}. Since $R$ is the invariant ring of $S$ under the action of the symplectic group  \cite[II~4.3]{LS},  the inclusion splits. Again, we apply Theorem~\ref{COMPUT-prop} then applies. The ring $R$ is generated in degree 2, and the multiplicity of these rings is known by the work of Herzog and Trung  {\cite[Theorem~5.6]{HerzogTrung}.}
\end{proof}

We end with one more related example.

\begin{example}
	\label{Grassmann}
	Let $\displaystyle R=\frac{\CC[u,v,w,x,y,z]}{(ux+vy+wz)}$. If each variable has degree two, this is isomorphic to the coordinate ring of the Grassmannian $\operatorname{Gr}(2,4)$, which is an invariant ring of an action by $SL_2(\CC)$. The inclusion map of this invariant ring is order-differentially extensible  \cite[Theorem~11.9]{Schwarz}. Since $R$ has multiplicity 2, we find $\dm{\CC}(R)=1/16$.
\end{example}

%%%%%%%%%%%%%%%%%%%%%%%%%
\subsection{Quadrics}\label{SubQuadrics}
%%%%%%%%%%%%%%%%%%%%%%%%%

We deal now with the quadric hypersurface \[ R=K[x_1, \ldots, x_{d+1}]/(x_1^2 + \cdots + x_{d+1}^2) .\] 
Over an algebraically closed field of characteristic $0$, all nondegenerate quadrics can be brought into this form. Nondegeneracy is equivalent to the property that $\Proj R$ is smooth. We show that in this case the differential signature is $ \left( \frac{1}{2} \right)^{d-1}$ provided $d \geq 2$. The arguments are quite involved and need several preparations. The first lemma gives explicitely the existence of sufficiently many unitary operators to deduce the estimate $\geq \left( \frac{1}{2} \right)^{d-1} $. For the other estimate we then study global sections of symmetric powers of syzygy bundles on the quadric $\Proj R$. The cases $d=2,3,5$ are covered by the examples done in previous subsections; see Theorem~\ref{group-formula}, Theorem~\ref{ThmToric}, Example~\ref{signaturesegre}, and Example~\ref{Grassmann}.

\begin{lemma} \label{quadriclemma}
	Let $f=x_1^2 + \cdots + x_{d+1}^2$, $d \geq 2$, and $R=K[x_1, \ldots, x_{d+1}]/(f)$ over a field $K$ of characteristic $0$. Then the following hold.
	\begin{enumerate}
		\item There exists a differential operator $\delta_1$ of order $2$ and homogeneous of degree $-1$ given by
		\[\delta_1 =(d-1) \partial_1 +x_1 \partial_1^2- \sum_{j \neq 1} x_1 \partial_j^2 + 2 \sum_{ j \neq 1} x_j \partial_1 \partial_j . \]
		
		\item
		A monomial $x^\lambda$ is sent by $\delta_1$ to 
		\[ \delta_1 (x^\lambda) =\lambda_1 \left( (d-1 ) + (\lambda_1-1) + 2 \sum_{j \neq 1} \lambda_j \right) x^{\lambda -e_1 } - \sum_{j\neq 1} \lambda_j ( \lambda_j-1) x^{ \lambda + e_1-2e_j} .\]
		
		\item
		We have the identity (as operators on the polynomial ring)
		\[\partial^\nu \circ \delta_1 =  \left( d-1+ \nu_1 +2 \sum_{j \neq 1} \nu_j \right) \partial^{\nu+e_1} - \nu_1 \sum_{j \neq 1 } \partial^{\nu -e_1+2e_j } + \theta ,  \]
		where $\theta$ is a sum of operators of the form $f_\lambda \partial^\lambda $ with $f_\lambda$ homogeneous of positive degree.
		
		\item
		For every monomial $x^\nu$ of degree $n$ with $\nu_{d+1} \leq 1$ there exists a differential operator $\xi_\nu$ of order $ \leq 2n$ and homogeneous of degree $-n$ of the form 
		\[\xi_\nu  = \frac{1}{\nu!} \partial^\nu + \zeta+\theta \, \] 
		where $\zeta$ is a linear combination of $\partial^\mu$ with $ \mu_{d+1} \geq 2 $ and where $\theta$ is a sum of operators of the form $f_\lambda \partial^\lambda $ with $f_\lambda$ homogeneous of positive degree.
		
		\item
		The induced $K$-valued operators $\tilde{\xi}_\nu$ have the property that
		\[\tilde{\xi}_\nu (x^\nu) =1 \text{ and } \tilde{\xi}_\nu (x^\mu) = 0 \text{ for all monomials }   \mu \neq \nu   \text{ with } \mu_{d+1} \leq 1 .  \]
	\end{enumerate}
\end{lemma}
\begin{proof}	
	\begin{enumerate}
		\item We claim that the tuple $a_\lambda$ indexed by monomials of degree $\leq 2$ given by
		\[ a_{ e_1 }  = d-1,\,   a_{ 2e_1} =2x_1,\, a_{2e_j }= -2x_1, a_{e_1+e_j} = 2 x_j  \text{ for } j \neq 1,  \]
		and all other entries $0$, gives a relation between the columns of the second Jacobi-Taylor matrix. From this relation, the corresponding differential operator arises via Corollary~\ref{JacobiTayloroperators}. To prove the claim we have to establish the relations \[\sum_\lambda a_\lambda \frac{\partial^{\lambda - \mu } }{ (\lambda - \mu)!}(x_1^2 + \cdots + x_{d+1}^2) =0\] in $R$ for all $\mu$ of degree $\leq 1$. For $\mu =0$ we get	
		\[ \begin{aligned}
		\sum_\lambda a_\lambda \frac{\partial^{\lambda   } }{ \lambda !}(f)	
		& =  (d-1) \partial^1 (f) + 2x_1 \frac{ \partial_1^2}{2}   (f) - 2 x_1 \sum_{j \neq 1} \frac{\partial_j^2 }{2}  (f)     \\
		& =  (d-1) 2x_1 + 2x_1 - \sum_{j \neq 1} 2x_1 \\
		& = 0 ,  \end{aligned}\]
		for $\mu =e_1$ we get	
		\[ \begin{aligned}
		\sum_\lambda a_\lambda \frac{\partial^{\lambda - \e_1  } }{ ( \lambda -e_1) !}(f)	
		& = (d-1) f +    2x_1  \partial_1 (f) + 2 x_j \sum_{j \neq 1} \partial_j (f)         \\
		& =4x_1^2 + \sum_{j \neq 1} 4x_j^2 \\
		& = 0  , \end{aligned}\]
		and for $\mu =e_k$, $k \neq 1$, we get	
		\[ \begin{aligned}
		\sum_\lambda a_\lambda \frac{\partial^{\lambda -e_k   } }{ ( \lambda-e_k) !}(f)	
		& = - 2x_1  \partial_k (f) + 2 x_k \partial_1 (f)     \\
		& = 0 .\end{aligned}\]
		\item This is a direct computation.
		
		\item
		A direct computation using $ \partial^\nu \circ x_1 \partial_j^2 = x_1 \partial^{\nu+2e_j} + \nu_1 \partial^{\nu-e_1+2e_j}  $ gives
		\[ \begin{aligned} \partial^\nu \circ \delta_1  & = { \partial^\nu  \left( (d-1) \partial_{1} +x_1 \partial_{1}^2 - \sum _{j \neq 1} x_1 \partial_{j}^2 + 2 \sum_{j \neq 1} x_j \partial_{1} \partial_{j} \right) } \\
		& =  (d-1) \partial^{\nu+e_1} + \nu_1 \partial_1^{\nu+e_1} - \nu_1 \sum_{j \neq 1} \partial^{\nu -e_1+2e_j } +2 \sum_{j \neq 1} \nu_j \partial^{\nu+e_1} + \theta \\
		& =  \left( d-1+ \nu_1 +2 \sum_{j \neq 1} \nu_j \right) \partial^{\nu+e_1} - \nu_1 \sum_{j \neq 1} \partial^{\nu -e_1+2e_j } + \theta  . \end{aligned} \]
		
		\item	
		We do induction on $n$. For $n=0$ the statements are clear and for  $n=1$ the operators $\delta_1, \ldots, \delta_{d+1}$ have the required properties. We construct the operators $\xi_\nu$ inductively using compositions of the $\delta_i$. So assume that we have already constructed the operators $\xi_\nu$ for all $\nu$ of degree $n$. 
		
		With the help of (3) we get	(with some $\theta$ as in (3))
	{	\[	 \begin{aligned}
			& \frac{ 1 }{ \nu! } \partial^\nu \circ \delta_1 - \frac{ \nu_1 }{ (\nu_2+1) \nu! } \partial^{\nu -e_1+e_2} \circ \delta_2 \\
		& =  \frac{ d-1+ \nu_1 + 2 \sum_{j \neq 1} \nu_j }{ \nu! } \partial^{\nu+e_1} - \frac{ \nu_1 }{ \nu! } \sum_{j \neq 1} \partial^ {\nu-e_1 +2e_j} \\  & \qquad - \frac{ \nu_1 (d-1+ \nu_2 +1 + 2 \sum_{j \neq 2} \nu_j -2 ) }{ (\nu_2+1) \nu! } \partial^{\nu-e_1+2e_2} \\  & \qquad + \frac{ \nu_1 }{ (\nu_2+1) \nu! }(\nu_2+1) \sum_{j \neq 2} \partial^{\nu-e_1 +2e_j} +\theta \\
		& =  \frac{ d-1+ \nu_1 + 2 \sum_{j \neq 1} \nu_j }{ \nu! } \partial^{\nu+e_1} - \frac{ \nu_1 (d-2+ \nu_2 + 2 \sum_{j \neq 2} \nu_j ) }{ (\nu_2+1) \nu! }\partial^{\nu-e_1+2e_2} \\
		& \qquad - \frac{ \nu_1 }{ \nu! } \partial^{\nu-e_1+2e_2} + \frac{ \nu_1 }{ \nu! } \partial^{\nu-e_1 +2e_1} +\theta \\
		& = \frac{ d-1+ 2 \sum_{j} \nu_j }{ \nu! } \partial^{\nu+e_1} - \frac{ \nu_1 (d-1+ 2 \sum_{j } \nu_j ) }{ (\nu_2+1) \nu! }\partial^{\nu-e_1+2e_2} +\theta \\
		& =  \left( d-1+ 2 \sum_{j} \nu_j \right) \left( \frac{ \nu_1 +1 }{ (\nu+e_1)! } \partial^{\nu+e_1} - \frac{ \nu_2+2 }{ (\nu-e_1+2e_2)! } \partial^{\nu-e_1+2e_2} \right) +H . \end{aligned}\]}
		From this we get operators
		\[	 \begin{aligned}
		\xi_\nu \circ \delta_1 &- \xi_ { \nu-e_1 + e_2 } \circ \delta_2 \\	& = \left( \frac{  \partial^\nu }{ \nu! } +\zeta_\nu +\theta_\nu \right) \circ \delta_1 -  \left(\frac{  \partial^{\nu-e_1 + e_2 } } { ( \nu-e_1 + e_2 )! } +\zeta_{\nu-e_1 + e_2 } +H_{\nu-e_1 + e_2 } \right) \circ \delta_2  \\ 
		& = a \partial^{\nu+e_1} +b \partial^{\nu-e_1+2e_2} +G+H
		\end{aligned}\]
		with certain coefficients $a,b \neq 0$. Summing up such operators we get for each $\mu$ of degree $n+1$ an operator of the form
		\[  \frac{ \partial^\mu}{\mu!} + c \partial^\lambda +\zeta+\theta  \]
		where $\lambda_j \leq 1$ for $j=2, \ldots, d+1$ (we shift as much as possible to the first exponent).
		If, say, $\lambda_2=1$, then $\xi_{\lambda-e_2} \circ \delta_2$ shows the existence of an operator as looked for: the summand $\zeta_{\lambda-e_2}$ gives the new summand of this kind; this does not work for $\lambda_{d+1}=1$. So, assume that $ \lambda_2=\cdots= \lambda_{d}=0$, $\lambda_{d+1} =0,1$ and $\lambda_1=n+1,n$.
		We have operators of the form
		$\partial_1^{n+1} +c \partial_{d+1}^{n+1} +\zeta+\theta$ or $\partial_1^{n+1} + c \partial_1 \partial_{d+1}^{n}+\zeta+\theta$ or $\partial_1^{n}\partial_{d+1} + c \partial_1 \partial_{d+1}^{n}+\zeta+\theta$  or $\partial_1^{n} \partial_{d+1} + c \partial_{d+1}^{n+1}+\zeta+\theta$. For $n \geq 2$ the second summand can be moved to the $\zeta$, for $n=1$ the operators can be given directly anyway. 
		
		\item
		This follows from (4). Indeed, the operator $\theta$ induces the zero operator to $K$ by base change, and the operator $\zeta$ annihilates all monomials $x^\mu$ with ${\mu_{d+1} \leq 1}$.\qedhere
	\end{enumerate}
\end{proof}

The previous lemma shows the existence of many unitary differential operators on a quadric. In order to get an upper bound for the differential signature we apply the methods from Subsection~\ref{graded case}, in particular Corollary~\ref{gradedhypersurfacesymdersignature} and Remark~\ref{Symsyzcomputation}. Note that in the quadric case the bundle $\Syz( \partial_1f, \ldots, \partial_{d+1} f)$ is (up to the scalar $2$) just the syzygy bundle $\Syz(x_1, \ldots,x_{d+1})$ on $Q_d=\Proj R$, which is the restriction of the cotangent bundle on $\PP^d$.

\begin{lemma}
	\label{quadricsymmetricsection}
	Let $K$ be an algebraically closed field of characteristic $0$ and let $f$ be an irreducible quadric in $d+1$ variables, $d \geq 2$. Then $\Sym^q(\Syz(x_1, \ldots, x_{d+1} )) (\degtotal )$ on $Q_d = \Proj K[x_1, \ldots,x_{d+1}]/(f)$ has no nontrivial section for $\degtotal < \frac{3}{2}q$.
\end{lemma}
\begin{proof}
	We do induction on the dimension $d$. For $d=2$, the quadric $Q$ is a projective line $\PP^1$ as an abstract variety, but embedded as a quadric. Let $\shL$ be the unique ample line bundle of degree $1$ on $Q$, so that ${\mathcal O}_Q(1) \cong \shL^2$. It is known that $\Syz(x_1,x_2, x_3)$ splits as $\shL^{-3} \oplus \shL^{-3}$ on $Q$. Hence $\Sym^q( \Syz(x_1,x_2, x_3)) \cong \bigoplus_{q+1} \shL^{-3q}$. So
	\[ \Sym^q( \Syz(x_1,x_2, x_3)) (\degtotal) \cong \bigoplus_{q+1} \shL^{-3q} (\degtotal) \cong \bigoplus_{q+1} \shL^{-3q+2\degtotal} \]
	has no nontrivial section for $\degtotal < \frac{3}{2}q$.
	
	So suppose now that $d \geq 3$ and that the statement is true for smaller $d$. A generic hyperplane section of a smooth quadric is again a smooth quadric. The restriction of the syzygy bundle $\shF_{d} = \Syz(x_1, \ldots, x_{d+1} )$ on $Q_d$ to $Q_{d-1}$ (say, given by $x_{d+1}=0$) is isomorphic to 
	\[\shF_{d} |_{Q_{d-1} }\cong \shF_{d-1} \oplus {\mathcal O}_{Q_{d-1} } (-1) . \]
	Therefore, the restriction of the symmetric powers of $\shF_d$, which are the symmetric powers of the restriction, $\Sym^q ( \shF_{d-1} \oplus  {\mathcal O}_{Q_{d-1} } (-1)   )$, is
	\[ \Sym^q (\shF_{d-1} ) \oplus \Sym^{q-1} (\shF_{d-1} ) (-1) \oplus \Sym^{q-2} (\shF_{d-1} ) (-2) \oplus    \cdots  \oplus {\mathcal O}_{Q_{d-1} } (-q)  . \]
	By the induction hypothesis, we have information about the global sections of the summand on the left, but not about the other summands. This decomposition is compatible with the decomposition of $\Sym^q\( \OO_{Q_d } (-1)^{\oplus d+1} \)$ coming from $  \OO_{Q_d }^{\oplus d+1} \cong ( \OO_{Q_{d} }^{\oplus d} ) \oplus \OO_{Q_d} $. Therefore, if a section of $\Sym^q(\shF_d) (\degtotal)$ on $Q_d$ is given as a tuple $(\alpha_\nu)$ in the kernel of the map $ \bigoplus \OO_{Q_d}(\degtotal -q) \rightarrow  \bigoplus \OO_{Q_d}(\degtotal -q+1) $, then its restriction to $Q_{d-1}=V_+(X_{d+1})$ is directly given with respect to the decomposition $\bigoplus_{k =0}^q \Sym^{q-k} ( \shF_{d-1}) (-k) $ as the family of the kernel elements of \[  \Sym^{q-k} \(\bigoplus_d \OO_{Q_{d-1} }(-1)\) \rightarrow  \Sym^{q-k-1} \(\bigoplus_d \OO_{Q_{d-1} }(-1) \) .\]
	
	So assume now that there is a nonzero section of $\Sym^q(\shF_{d}) (\degtotal)$ with $\degtotal <\frac{3}{2}q$ given by a tuple $(\alpha_\nu)$, which are homogeneous elements of $K[x_1, \ldots,x_{d+1}]/(f)$ of degree $ \degtotal -q $. We look at linear coordinate changes given by an invertible $(d+1) \times (d+1)$-matrix $M$ over $K$ (or field extensions of it) and giving rise to the commutative diagram
	\[ \xymatrixcolsep{5pc}\xymatrix{  \OO_{Q_{d} }(-1)^{\oplus d+1} \ar[r]^-{x_1, \ldots ,x_{d+1} }\ar[d]^-M &   \OO_{Q_{d} } \ar[d]^-{=}  \\  \OO_{Q_{d} }(-1)^{\oplus d+1} \ar[r]^-{ y_1, \ldots, y_{d+1} } &  \OO_{Q_{d} }, } \]
	and to
	\[ \xymatrix{ \Sym^q ( \OO_{Q_{d} }(-1)^{\oplus d+1} ) \ar[r] \ar[d]^{\Sym^q(M) } &  \Sym^{q-1} ( \OO_{Q_{d} }(-1)^{\oplus d+1} )  \ar[d]^{  \Sym^{q-1}(M) }  \\  \Sym^q ( \OO_{Q_{d} }(-1)^{\oplus d+1} ) \ar[r] &  \Sym^{q-1} (  \OO_{Q_{d} }(-1)^{\oplus d+1} ),  } \]
	and also $\degtotal$-twists thereof. Here $\Sym^q(M) $ is the $q$th symmetric power of $M$. Now, we look at the field extension $K \subseteq K'=K(t_{ij})$ (or its algebraic closure) with new algebraically independent elements $t_{ij}$, $1 \leq i,j \leq d+1$, and we consider the matrix $ M=(t_{ij} )$ which gives corresponding diagrams over $K'$. Corollary~\ref{matrixtranscendent} below applied to the vector space of forms of degree $\degtotal-q$  shows that in $\Sym^q(M) (\alpha_\nu)  $ all entries are nonzero. As the transformed representation is as good as the starting one, we may assume that all entries of $\alpha_\nu$ are nonzero. Now we restrict to $Q_{d-1}=V_+(L)$ for a linear form $L$. As the polynomials $\alpha_\nu$ only have finitely many linear factors altogether (since $K[x_1, \ldots, x_{d+1}]/(f) $ is factorial for $d \geq 4$, for $d=3$ the argument is slightly more complicated), we find $L$ such that the restrictions of all $\alpha_\nu$ to $V_+(L)$ are nonzero. So this produces a contradiction to the induction hypothesis.	
\end{proof}

\begin{lemma}
	\label{matrixsymmetric}
	Let $K$ be a field of characteristic $0$ and let $  M  = [t_{ij}]$ be an $n \times n$ matrix of indeterminates. Then the entries in the symmetric powers $\operatorname{Sym}^{ q } \left( M \right)  $ are linearly independent over $K$.
\end{lemma}
\begin{proof}
	The symmetric power  $\operatorname{Sym}^{ q } \left( M \right)$ of the matrix describes the induced map on the polynomial ring $K(t_{ij})[x_1 , \ldots , x_n]$ in degree $q$ given by the linear map $ x_i \mapsto \sum_{j = 1}^n t_{ji} x_j$. The entry in the $\mu$th row and the $\nu$th column is the coefficient of  $x^\mu$ of
	\[ \left( t_{11}x_1 + \cdots + t_{n1}x_n \right)^{\nu_1} \cdots \left( t_{1n}x_1 + \cdots + t_{nn} x_n \right)^{\nu_n} \,  . \]
	The $i$-th power is
	\[  \sum_{ \mondeg {\lambda_i } = \nu_i} { { \nu_i } \choose \lambda_i} \cdot t_i^{\lambda_i} x^{\lambda_i}   . \]
	To determine the coefficient of $x^\mu$ in the product, we have to consider the
	product of the form
	\[ t_1^{\lambda_1} x^{\lambda_1} \cdots t_n^{\lambda_n} x^{\lambda_n}  = t_1^{ \lambda_1} \cdots t_n^{\lambda_n} x^{ \lambda_1 + \cdots + \lambda_n} \]
	with $\lambda_1 + \cdots + \lambda_n  = \mu $. Since the $t_{ij}$ are variables, we get from the monomial  $t_1^{ \lambda_1} \cdots t_n^{\lambda_n} $ the multi-tuple $(\lambda_1 , \ldots , \lambda_n)$. This determines $\mu$ as the sum and it determines $\nu$ via $\nu_i =\mondeg {\lambda_i }$. This means that each  $t_1^{ \lambda_1} \cdots t_n^{\lambda_n} $ occurs only in one entry of the symmetric power matrix. Since the binomial coefficients are not $0$ in characteristic zero, the entries are linearly independent.
\end{proof}

\begin{corollary}
	\label{matrixtranscendent}
	Let $V$ be a finite dimensional $K$-vector space, let $t_{ij}$, $1 \leq i,j \leq n$, be variables with corresponding field extension $K \subseteq K'=K(t_{ij})$ and let $M=[t_{ij}]$. Let $ \Sym^q (M): V^{q+n-1 \choose n-1 } \tensor_K K' \rightarrow V^{ q +n-1\choose n-1 } \tensor_KK' $. Then every nonzero element in $ V^{ q +n-1\choose n-1 } $ is sent by this map to an element such that all its entries are nonzero.
\end{corollary}
\begin{proof}
	Let $V=K^m$ and let $\beta=(\alpha_\nu) =(\alpha_{\nu,j}) \neq 0$. Then $\alpha_{\nu,j} \neq 0 $ for some $\nu,j$. Assume that $(\Sym^q(M) (\beta))_\mu = 0$. Writing $  \Sym^q(M) = (c_{\mu , \nu} ) $, this means that $\sum_\nu  c_{\mu, \nu}   \alpha_\nu =0$ and in particular  $\sum_\nu  c_{\mu, \nu}   \alpha_{\nu,j} =0$ for all $j=1, \ldots, m$. But this means that there exists a nontrivial $K$-linear relation between the entries in the $\mu$-row of $\Sym^q(M)$ contradicting Lemma~\ref{matrixsymmetric}.
\end{proof}

\begin{theorem}
	\label{quadricsignature}
	Let $R=K[x_1, \ldots, x_{d+1}]/(x_1^2 + \cdots + x_{d+1}^2) $, $ d \geq 2$, over an algebraically closed field $K$ of characteristic $0$. Then the differential signature of $R$ is $\left( \frac{1}{2} \right)^{d-1}$.
\end{theorem}
\begin{proof}
	Lemma~\ref{quadriclemma}~(5) tells us that for every
	monomial $x^\nu$ of degree $n$ there exists an operator $F_\nu$ of order $\leq 2n$ and homogeneous of degree $-n$ sending  $x^\nu$ to a unit and sending the other monomials to $0$. This means that these operators form an independent system of unitary operators of order $2n$ and of cardinality 
	 {
	\[\sum_{j=0}^n \dim_K( R_j) = \dim_K R/\m^{n+1} =2 \frac{n^{d} }{d!}. \]
	}
	Therefore the quotient of unitary operators of order up to $2n$ compared with all operators of order up to $2n$ is $\geq \frac{2n^d/d!}{ (2n)^d/d!} $ and the differential signature $\geq (1/2)^{d-1}$.
	
	From Lemma~\ref{quadricsymmetricsection} and Corollary~\ref{gradedhypersurfacesymdersignature} we get (with $e=2$ and $\alpha =1/2$) $ \dm{K}(R) \leq  (1/2)^{d-1}   $.
	\end{proof}

\begin{remark}
	For $d=1$ the equation is $x^2+y^2=0$. In this case the situation is completely different. On one hand, there are no unitary operators beside the identity at all (the operators constructed in Lemma~\ref{quadriclemma} exist, but are not unitary). On the other hand, there are many global sections of $\Sym^q (\Syz(x,y) ) \cong \Sym^q ( \OO (-2)) $ of low degree.
\end{remark}

\begin{remark}
	The operators $\xi_\nu$ of order $\geq 2n$ from Lemma~\ref{quadriclemma} yield sections in $\Sym^{2n} (\Syz(x_1, \ldots, x_{d+1} )) (3n) $ on $Q_d $, hence Lemma~\ref{quadricsymmetricsection} is best possible. For example, $\delta_1$ yields a section in $\Sym^{2} (\Syz(x_1, \ldots, x_{d+1} )) (3) $ given by variables.
\end{remark}

\begin{remark}
	With the methods of this section we can also compute the differential powers of the maximal ideal of a quadric, the result is \[ \m\dif{2n-1}{K}= \m\dif{2n}{K} = \m^n .\]
	We restrict ourselves to monomials. If $x^\nu \notin \m^n$, then the degree of $x^\nu$ is $ \leq n-1$ and then there exists by Lemma~\ref{quadriclemma} a unitary operator of order $\leq 2n-2$ sending it to a unit. Hence $x^\nu \notin \m\dif{2n-1}{}$. If $x^\nu \in \m^n$, then its degree is $\geq n$. Lemma~\ref{quadricsymmetricsection} and the proof of Theorem~\ref{quadricsignature} shows that there does not exist an operator of order $< 2n$ sending $x^\nu$ to a unit. Hence $x^\nu \in \m\dif{2n}{K}$.	
\end{remark}

\begin{remark}
	The limit of the $F$-signature of the quadrics $R_{d,p}=\FF_p[x_1, \ldots, x_{d+1}]/(x_1^2+ \cdots + x_{d+1}^2)$ as $p$ goes to infinity can be computed via \cite[Example~2.3]{WatanabeYoshida} from results of Gessel and Monsky \cite[Theorem~3.8]{GesselMonsky}. The result is that the limit is one minus the coefficient of $z^d$ in the power series expansion of $\operatorname{tan}(z) + \operatorname{sec}(z)$. This gives the values
	
	\
	
	\begin{center}
		\begin{tabular}{|c|c|c|c|c|c|c|} \hline $ d $ & 2 & 3 &4&5&6&7  \\ \hline $ \lim_{p \rightarrow \infty} s(R_{d,p} )$ & $\frac{ 1 }{ 2 }$ & $ \frac{ 2 }{ 3 }$ &$ \frac{ 19 }{ 24 }$  &$\frac{ 13 }{ 15 }$ &$ \frac{ 659 }{ 720 }$ & $\frac{ 298 }{ 315 }$ \\ \hline \end{tabular}
	\end{center}
	\
	
	\noindent which look much wilder than the differential signature.
\end{remark}

\begin{remark}
	For $d=2c+1$ odd, and $K$ algebraically closed of characteristic zero, the quadric hypersurface $R=K[x_1, \ldots, x_{d+1}]/(x_1^2+ \cdots + x_{d+1}^2) \cong K[y_1, \ldots, y_{d+1}]/(y_1 y_2+ y_3 y_4 + \cdots + y_d y_{d+1})$ can be realized as a ring of invariants of an action of $\mathrm{SL}_{c}(K)$: if $V$ is the standard representation, then the invariant ring of the representation $V^{\oplus c-1} \oplus V^*$ is isomorphic to $R$ \cite[\S6,\S14]{Weyl}. However, the inclusion of this invariant ring into the ambient polynomial ring is not differentially extensible \cite[Theorem~11.15]{Schwarz}, so the methods of the previous subsections of this section do not apply.
\end{remark}

%%%%%%%%%%%%%%%%%%%%%%%%%%%%%%%%%%%%%%%%%%%%%
\section{Comparison with differential symmetric signature}\label{Comparison}
%%%%%%%%%%%%%%%%%%%%%%%%%%%%%%%%%%%%%%%%%%%%%

In this section we compare the differential signature and the differential symmetric signature recently introduced by Caminata and the first author \cite{SymSig,BCHigh}. Before recalling the definition of this signature, we make a few observations.

We have the short exact sequence
\[  0 \longrightarrow \Delta^{n}/\Delta^{n+1} \longrightarrow P^n_{R|A} = (R \tensor_A R)/\Delta^{n+1}  \longrightarrow P^{n-1}_{R|A} = (R \tensor_A R)/\Delta^{n}   \longrightarrow 0 \]
of $R$-modules. The direct sum $ \gr_{\bullet}(\ModDif{\bullet}{R}{A})=\bigoplus _{n \in \NN}  \Delta^n/\Delta^{n+1}$ is the \emph{graded associated ring} (for the diagonal embedding). There is a surjective graded $R$-linear map
 \cite[16.3.1.1]{EGAIV}:
 \[ \bigoplus_{n \in \NN} \Sym^n (\Omega_{R|A} ) \longrightarrow  \bigoplus_{n \in \NN} \Delta^{n}/\Delta^{n+1}.    \]
 The algebra on the left is called the \emph{tangent algebra} as its spectrum gives the tangent scheme over $\Spec R$. This map is induced in degree one by the identity $\Omega_{R|A} \rightarrow \Delta/\Delta^2$. If $R$ is differentially smooth in the sense of \cite[D{e}finition~16.10.1]{EGAIV}, then by definition
$\Omega_{R|A}$ is locally free and this canonical homomorphism is an isomorphism. In this case also the modules of principal parts are locally free. So in the affine smooth case we have a decomposition
\[ P^n_{R|A} = \bigoplus_{k \leq n}  \Sym^k (\Omega_{R|A} )  . \]

The \emph{symmetric signature}\index{symmetric signature}  \cite{SymSig,BCHigh} is defined as the limit (if it exists) for $n \rightarrow \infty$ of
\[ \frac{ \frk \left( \bigoplus_{k \leq n} \Sym^k (\Omega_{R|A} )^{**} \right) }{ \Rank \left( \bigoplus_{k \leq n} \Sym^k (\Omega_{R|A} ) \right) }  =  \frac{ \frk \left( \bigoplus_{k \leq n} \Sym^k (\Omega_{R|A} )^{**} \right) }{ \binom{d+n}{n} } , \]
where $\, ^{**}$ denotes the double dual functor $\Hom_R (\Hom_R (-,R),R)$. This gives the reflexive hull of the module, which is also the evaluation of the (sheaf) module over an open subset $U$ containing all points of codimension one. If $U$ is also smooth, and such subsets exist in the normal case, then there is an exact sequence
\[ 0 \longrightarrow  \Sym^n (\Omega_{R|K})|_U  \longrightarrow  P^n_{R|K}|_U  \longrightarrow  P^{n-1}_{R|K}|_U \longrightarrow 0 \,  \]
of locally free sheaves, which does not split in general. We describe a situation where the module of principal parts on $U$ splits and is isomorphic to the direct sum of the symmetric powers of the K\"ahler differentials.

\begin{theorem}
	\label{differentialsymmetriccompare}
	Let $S=K[x_1, \ldots, x_d]	$ be a polynomial ring and $G$ a finite group acting on $S$, with order coprime to the characteristic of the field $K$, with invariant ring $R=S^G$. Suppose that $G$ contains no elements that fix a hyperplane in the space of one-forms $[S]_1$. Let $U$ denote the smooth locus of $\Spec R$. Then there exists the following commutative diagram   
	\[ \xymatrix{ 
		\Sym^n (\Omega_{R|K}) \ar[r]\ar[dr] &  \Delta^{n}/\Delta^{n+1} \ar[r]\ar[d] & \ModDif{n}{R}{K} \ar[r]\ar[d] & \ModDif{n-1}{R}{K} \ar[r]\ar[d] & 0 \\
		0 \ar[r] & \Sym^n (\Omega_{R|K}) (U) \ar[r]\ar[d]^-{\cong} & \ModDif{n}{R}{K}(U) \ar[r]\ar[d]^-{\cong} \ar@/_/@{-->}[l] & \ModDif{n-1}{R}{K}(U) \ar[r]\ar[d]^-{\cong}\ar@/_/@{-->}[l] & 0\\
		0 \ar[r] & (\Sym^n (\Omega_{S|K}))^G \ar[r] & (\ModDif{n}{S}{K})^G \ar[r]\ar@/_/@{-->}[l] & (\ModDif{n-1}{S}{K})^G \ar[r]\ar@/_/@{-->}[l] & 0
	 }\]
	where the dotted arrows indicate splittings. The free rank of $P^n_{R|K} $ and of $(P^{n}_{ S|K} )^G$ coincide, and they equal the sum of the free ranks of $\Sym^k (\Omega_{R|K}) (U)$ for $k \leq n$. In particular, the differential signature equals the symmetric signature, namely $1/ |G|$.	
\end{theorem}
\begin{proof}	
	The first row exists for every $K$-algebra $R$. The first downarrows on the right are the restrictions for the open subset $U$. The exact row in the middle comes from the smoothness of $U$ (without the splittings). The first downarrows on the left are induced by the exactness we have so far.
	
	The smooth locus $U$ contains by the smallness assumption on the group action all points of codimension one and the same is true for its preimage $V \subseteq {\mathbb A}^d_K$. Hence we have natrual maps	
	$	 \ModDif{n}{R}{K}(U) \rightarrow  \ModDif{n}{S}{K}(V) =  \ModDif{n}{S}{K}$
	where the identity comes from the freeness of $ \ModDif{n}{S}{K} $ and the codimension property of $V$. The image lies in the invariant subspace of the induced action on $P^n_{S|K}$. This gives the second downarrows.
	
	For the polynomial ring we have $P^n_{S|K} = \bigoplus_{k \leq n} \Sym^k(\Omega_{S|K} ) $ and hence the invariants of the induced action on $P^n_{S|K}$  are
	$(P^n_{S|K})^G = \bigoplus_{k \leq n} ( \Sym^k(\Omega_{S|K} ))^G $.
	Hence the splitting in the last row is clear. The induced map $V \rightarrow U=V/G$ is \'{e}tale, hence the second downarrows are isomorphisms, as they are locally isomorphisms on the affine smooth (invariant) subsets. Therefore we get the splitting in the second rows.
	
	The differential operators on $R$ correspond to the invariant differential operators on $S$. This is true for the quotient fields $Q(R) \subseteq Q(S)$ (which is a Galois extension) and so it is also true for the rings as every operator on $R$ has an extension to $S$ by Proposition \ref{prop-extensible-finite} which must be the invariant one.	
	
	A free summand of $P^n_{R|K}$ is the same as a surjection $P^n_{R|K} \rightarrow R $ which gives also a surjection $P^n_{R|K} (U) \rightarrow R$. On the other hand, such a map corresponds to a differential operator $\delta$ on $U$ and on $R$. Let $\tilde{\delta}$ be the corresponding invariant differential operator on $S$. Suppose now that $P^n_{R|K} (U) \cong (P^n_{S|K} )^G \rightarrow R$ is surjective. Then also
	 $\tilde{\delta}: P^n_{S|K} \rightarrow S$ is surjective and so there exists $ f \in S$ such that $\tilde{\delta} (f) =1$. Then by the invariance of the operator $\tilde{\delta}$, $\tilde{\delta} \left(\sum_{\varphi \in G} \varphi(f) \right) = |G| $, which is a unit, and since $\sum_{\varphi \in G} \varphi(f)  \in R$, also the operator $\delta$ is unitary. Hence $\delta $ defines a surjection $P^n_{R|K} \rightarrow R$ by Lemma~\ref{unitary}. This argument works also for a family of unitary operators and shows that the free rank of $P^n_{R|K}$ and of $P^n_{R|K} (U) = (P^{n}_{ S|K} )^G $ coincide.
	 
	 Therefore the free rank of $P^n_{R|K}$ equals by the splitting of the second row the sum of the free ranks of $\Sym^k (\Omega_{R|K}) (U)$ for $k \leq n$. By the codimension property of $U$, $\Sym^k (\Omega_{R|K}) (U)$ is the reflexive hull of $ \Sym^k (\Omega_{R|K}) $ and the sum of its free ranks enters as denominators the definition of the symmetric signature. Hence the signatures must be the same in the current setting. The symmetric signature was alredy known \cite[Theorem 2.8]{BCHigh} and the differential signature was computed in Theorem~\ref{group-formula}.
\end{proof}

\begin{example}
The first downarrows in Theorem	\ref{differentialsymmetriccompare} are not isomorphisms, not even for $n=1$. We consider the invariant ring $R=K[x,y,z]/(xy -z^2)\cong K[s^2,t^2,st] \subseteq K[s,t]$ with the group action of $\ZZ/2$ given by sending the variables to their negatives. The module $\Omega_{R|K}$ is generated by $dx=2sds$, $dy=2tdt$ and $dz=sdt+tds$, whereas $(\Omega_{S|K})^G$ contains also $sdt$ and $tds$.
\end{example}

Theorem~\ref{differentialsymmetriccompare} can not be extended to more general situations.

\begin{example}
 For the toric (and determinantal) hypersurface given by the equation $ux-vy $, the symmetric signature is $0$  \cite[Example 3.9]{BCHigh}, but the differential signature is $1/4$ by Example~\ref{signaturesegre}.
\end{example}

We expect that, at least under some conditions, the symmetric signature gives a lower bound for the differential signature. The following considerations deal with this point. See also Example~\ref{palmostfermat} below for what can go wrong.

\begin{lemma}
	Let $(R,\m,\kk)$ be a local $K$-algebra essentially of finite type over a field $K$. Suppose that $R$ is an isolated singularity and that $\operatorname{Hom}_R(\Sym^n( \Omega_{R|K}),R)$ has depth $\geq 3$
	for all $n$. Then the natural map $D^n_{R|K} \rightarrow \operatorname{Hom}_R(\Sym^n( \Omega_{R|K}),R)$ is surjective.	
\end{lemma}
\begin{proof}
	Let $U=\Spec R \setminus \{ {\mathfrak m} \}$.
	We proof by induction the statement that the map is surjective and
	that $H^1(U, D^n_{R|K} )=0$. For $n=1$ the statement is clear since $D^1_{R|K}= R \oplus \operatorname{Der}_{R|K}$ and $H^1(U,R)=H^2_{\mathfrak m}(R)=0$ due to the depth assumption (for $n=0$). Let now the statement be known for $n-1$ and look at the commutative diagram

	\[ \xymatrix{ 0 \ar[r] & D^{n-1}_{R|K} \ar[r]\ar[d]^-{\cong} &  D^{n}_{R|K} \ar[r]\ar[d]^-{\cong} & \Hom_R(\Sym^n( \Omega_{R|K}),R) \ar[d]^-{\cong} &  \\
		0 \ar[r] & D^{n-1}_{R|K}(U) \ar[r] &  D^{n}_{R|K}(U) \ar[r] & \Hom_R(\Sym^n( \Omega_{R|K}),R) (U)\ar[r] & H^1(U,D^{n-1}_{R|K}).
	 } \]
	
	The downarrow maps are isomorphisms because of reflexivity. On the smooth locus $U$ we have a short exact sequence of sheaves and so the second row is exact. By the induction hypothesis, $H^1(U, D^{n-1}_{R|K} )=0$, and hence the map is surjective. The second statement follows from
	\[ \cdots \rightarrow   H^1(U, D^{n-1}_{R|K})  \rightarrow   H^1(U, D^n_{R|K}) \rightarrow H^1(U,\operatorname{Hom}_R(\Sym^n( \Omega_{R|K}),R)     )  \rightarrow \cdots \,  \]
	and the depth assumption.
\end{proof}

\begin{remark}
There are many results on depth properties for $\Sym^n(\Omega_{R|K})$ and on conditions for
$ \Sym^n(\Omega_{R|K}) \rightarrow \Delta^n/\Delta^{n+1}  $ to be a bijection in the literature.  For instance, the K\"ahler differentials in a complete intersection ring  has projective dimension $\leq 1$ and one can deduce that the depth of $\Sym^n(\Omega_{R|K})$ is $ \geq \dim(R)-1$ \cite[Proposition 3 (3)]{Avramovcomplete}; see also
 \cite[Propositions~2.10 \&~3.4]{SimisUlrichVasconcelostangentalgebrastangentstar}.
It is however more difficult to find depth conditions for $\operatorname{Hom}_R (\Sym^n (\Omega_{R|K} ),R )$ and even for $\Der_{R|K} $. If $\Omega_{R|K}$ itself is a maximal Cohen-Macaulay module, then for Gorenstein rings \cite[Proposition 3.3.3]{BrHe}  
%\cite[Theorem 21.21 (a)]{Eisenbud} 
also the dual is a maximal Cohen-Macaulay module. 
This can be applied to certain determinantal rings \cite[Proposition 14.7]{BrunsVetter}. In addition,  the derivation module for Pl\"ucker algebras of Grasmannians $\neq G(2,4)$ has depth $\geq \dim(R) -2$  \cite[Proposition 3.4]{ChristophersenIlten}. It would be interesting to know whether these results extend to depth conditions on $\operatorname{Hom}_R (\Sym^n (\Omega_{R|K} ) ,R)$.
\end{remark}

\begin{lemma}
	\label{differentialdualsymmetric}
	Let $(R,\m,\kk)$ be local and essentially of finite type over a field $K$. Suppose that the natural maps $D^n_{R|K} \rightarrow \operatorname{Hom}_R(\Sym^n( \Omega_{R|K}),R)$ are surjective and that the free ranks of $ P^n_{R|K}$ and of $D^n_{R|K}$ conincide.
	Then the sum of the free ranks of $(\Sym^k( \Omega_{R|K}))^{**}$ for $k \leq n$ is bounded above by the free rank of $ P^n_{R|K}$, and the symmetric signature is bounded above by the principal parts signature. 
\end{lemma}
\begin{proof}
	Suppose by induction that we have already a free summand $N$ of $D^{n-1}_{R|K}$ of rank equal to $\sum_{k = 0}^{n-1} \frk (  (\Sym^k( \Omega_{R|K}))^{**} )$. By assumption, there exists a free summand $V$ of $P^{n-1}_{R|K}$ such that $N=\operatorname{Hom}_R(V,R)$. As $V$ is also a free summand of $P^{n}_{R|K}$, also $N$ is a free summand of $D^n_{R|K}$.
	
	Let $M$ be a free direct summand of $(\Sym^n( \Omega_{R|K}))^{**} $. This defines a corresponding free direct summand of the dual of it, which is isomorphic to $(\Sym^n( \Omega_{R|K}))^{*} $. 
	By assumption we have the short exact sequence
	\[ 0  \longrightarrow  D^{n-1}_{R|K} \longrightarrow D^n_{R|K} \longrightarrow  \operatorname{Hom}_R(\Sym^n( \Omega_{R|K}),R)  \longrightarrow 0 \]
	hence we get a free direct summand $M$ of $ D^n_{R|K}$. As the free summand $N$ of $D^n_{R|K}$ maps to $0$, we have $N \cap M= 0$. Hence $N\oplus M$ is a free summand of  $D^n_{R|K}$.	
\end{proof}

The following theorem says that a significant part of $\operatorname{Hom}_R ( \operatorname{Sym}_R^n( \Omega_{R|K} ),R)$ is always inside the image of the map from $D^n_{R|K}$.

\begin{theorem}
	\label{compareoperatorcomposition}
	Let $K$ be a field, $R$ be a $K$-algebra, and let  $\delta_1 , \ldots , \delta_n$ denote derivations. Then the composition $\delta_n \circ \cdots \circ \delta_1$ is mapped under the natural mapping \[  D^n_{R|K} \longrightarrow \operatorname{Hom}_R ( \operatorname{Sym}_R^n(\Omega_{R{{|K}} } ),R) \] to the image of the symmetric product
	$\delta_n \cdots \delta_1$ under the natural map \[ {\operatorname{Sym}_R^n( \operatorname{Der}_{R|K}) \cong \operatorname{Sym}^n_R (\operatorname{Hom}_R ( \Omega_{R{{|K}} },R)) \longrightarrow \operatorname{Hom}_R ( \operatorname{Sym}_R^n(\Omega_{R{{|K}} } ),R) } . \]
\end{theorem}
\begin{proof}
	The homomorphisms in $\operatorname{Hom}_R ( \operatorname{Sym}_R^n(\Omega_{R{{|K}} },R)) $ are determined on the symmetric products of the differential forms $df$, as they generate this module. Let $ f_1 , \ldots , f_n \in R$. The $df_i \in \Omega_{R|K} \cong \Delta/\Delta^2$ are $f_i \otimes 1 -1 \otimes f_i$ and their product $f_1 \cdots f_n$ is sent to $  \sum_{I \subseteq \{1 , \ldots , n\} } (-1)^ { \#( I ) } \left( \prod_{i \notin I} f_i \right) \otimes \left( \prod_{i \in I} f_i \right) $ in $P^n_{R|K}$. Under a differential operator $\eta$ this is sent to
	\[ \sum_{I \subseteq \{1, \ldots, n\} } (-1)^{ \#( I ) } \left( \prod_{i \notin I} f_i \right) \eta \left( \prod_{i \in I} f_i \right) . \]
	In the case $\eta= \delta_n \circ \cdots \circ \delta_1$ we have for $I=\{i_1, \ldots, i_m\}$
	\[ (\delta_n \circ \cdots \circ \delta_1) \left( \prod_{i \in I} {f_i} \right) = \left( \sum_{ \{1 , \ldots , n \} = A_1 \uplus \cdots \uplus A_m } \delta_{A_1} (f_{i_1}) \cdots \delta_{A_m} (f_{i_m}) \right) ,  \]
	where $\delta_{A_j}(f_{i_j})$ denotes the composition of the derivations given by $A_j$ in the given order applied to $f_{i_j}$ and where the sum runs over all ordered partitions of $\{1, \ldots, n\}$.
	Hence the evaluation yields
	\begin{eqnarray*} 
		& \, &  \sum_{I \subseteq \{1 , \ldots , n\} } (-1)^{ \#( I ) } \left( \prod_{i \notin I} f_i \right) (\delta_n \circ \cdots \circ \delta_1 ) \left( \prod_{i \in I} f_i \right) \\
		& =& \sum_{I \subseteq \{1 , \ldots , n\} } (-1)^{ \#( I ) } \left( \prod_{i \notin I} f_i \right) \left( \sum_{ \{1 , \ldots , n \} = A_1 \uplus \cdots \uplus A_m } \delta_{A_1} (f_{i_1}) \cdots \delta_{A_m} (f_{i_m}) \right) \\
		& =& \sum_{I \subseteq \{1, \ldots, n\} } (-1)^{ \#( I ) } \sum_{ \{1 , \ldots , n \} = B_1 \uplus \cdots \uplus B_n \text{ with } B_i = \emptyset \text{ for } i \notin I } \delta_{B_1} (f_{1}) \cdots \delta_{B_n} (f_{n}) \\
		&= & \sum_{ \{1 , \ldots , n \} = B_1 \uplus \cdots \uplus B_n } \left( \sum_{I \subseteq I(B) } (-1)^{ \#( I ) } \right) \delta_{B_1} (f_{1}) \cdots \delta_{B_n} (f_{n}) ,
	\end{eqnarray*} 
	where here $I(B)$ denotes for an ordered partition $B=(B_1 , \ldots , B_n)$ the set of indices $i$ for which $B_i$ is empty. Note that in the first equation we can omit the summand corresponding to $ I  =  \emptyset$, since every derivation annihilates $1$. 
	
	For $ I(B)   \neq   \emptyset$ the inner sum is $0$, and for $I(B)  =  \emptyset $ the inner sum is $1$. Hence only those partitions are relevant, where no subset is empty, thus all subsets contain just one element. These correspond to the permutations  on $ \{ 1 , \ldots , n \}$, so this is the same as
	\[ \sum_{\pi \in S_n} \delta_{\pi (1)}(f_1) \cdots \delta_{\pi (n)}(f_n). \]
	
	The symmetric product
	\[\delta_n \cdots \delta_1 \in \operatorname{Sym}_R^n( \operatorname{Der}_{R|K}) \cong \operatorname{Sym}^n_R (\operatorname{Hom}_R ( \Omega_{R{{|K}} },R)) \]
	is sent under the natural map $\operatorname{Sym}^n_R (\operatorname{Hom}_R ( \Omega_{R{{|K}} },R)) \rightarrow \operatorname{Hom}_R ( \operatorname{Sym}_R^n(\Omega_{R{{|K}} } ),R)  $ to
	\[ \omega_1 \cdots \omega_n \longmapsto \sum_{\pi \in S_n} \delta_{\pi(1)} (\omega_1) \cdots \delta_{\pi(n)} (\omega_n) .\]
	For $ \omega_i =  df_i $ this coincides with the above result.
\end{proof}

This theorem says that we have a  commutative diagram
 \[ \xymatrix{  & \operatorname{Sym}^n(\Der_{R|K} ) \ar[d] \\
\mathscr{D}_n \ar[d]\ar[r] & \operatorname{im} \big(\operatorname{Sym}^n(\Der_{R|K}) \rightarrow \Hom_R(\Sym^n(\Omega_{R|K}),R) ) \big) \ar[d] \\
D^n_{R|K} \ar[r] & \Hom_R(\Sym^n(\Omega_{R|K}),R) }\]
where the downarrows in the second row are injective, $ \mathscr{D}_n$ denotes the submodule generated by the composition of $n$ derivations as in Remark~\ref{rem:der-simple}, and the first horizontal map is surjective.

\begin{example}
\label{palmostfermat}
Fix a prime number $p$, let $K$ be a field of characteristic $p$ and consider $f=x^p+y^{p+1}+z^{p+1}$ and the ring $R=K[x,y,z]/(f)$. As $f$ is irreducible, $R$ is a domain. The partial derivatives are $ \frac{\partial f}{\partial x} =0$, $ \frac{\partial f}{\partial y} =y^p$, and $ \frac{\partial f}{\partial z} =z^p$. Hence, in the singular locus $y$ and $z$ vanish, and then also $x$ has to vanish, so we have an isolated singularity and $R$ is a normal domain.

Because of $x^p=-yy^p -zz^p$ we have $x^p \in (y^p,z^p)$, so $x$ is in the Frobenius closure of the ideal $y,z$, but $x \notin (y,z)$. Hence $R$ is not $F$-pure and thus not strongly $F$-regular. Then,  the $F$-signature of $R$ is $0$ \cite[Theorem~0.2]{AL}. 

We compute the other signatures considered in this paper. We first show that the differential signature is positive. As a consequence, Theorem~\ref{ThmFregPos} does not hold for non $F$-pure rings.
We have the following sandwich situation
 {\[ K[y,z] \subseteq R \subseteq K[y^{1/p},z^{1/p}]\cong K[y,z]  , \]
where $R$ is a free module over $K[y,z]$ of rank $p$.} In this situation it follows from Proposition~\ref{sandwichpositive} that $R$ has positive differential signature. The ratios start in characteristic $2$ with $1/1$, $2/3$, $4/6$, $7/10$, but we do not know the value of the signature.

The module of K\"ahler differentials is given by the exact sequence
\[0 \longrightarrow R \stackrel{(0,y^p,z^p)}{\longrightarrow} R^3 \longrightarrow \Omega_{R|K} \cong R(dx,dy,dz)/df \longrightarrow 0 .\]
Hence \[\Omega_{R|K} \cong R \oplus R^2/(y^p,z^p) \cong R \oplus I , \]
where $I=(y^p,z^p)$. The second isomorphism comes from the fact that $I$ is a parameter ideal.
Hence the symmetric powers of the K\"ahler differentials itself are
\[\Sym^n(\Omega_{R|K} ) \cong R \oplus I \oplus I^{\otimes 2} \oplus \cdots \oplus I^{\otimes n} , \]
with just one free summand.

The derivation module $\Der_{R|K} \cong \Hom_R (\Omega_{R|K},R)$ is free (so this is another example showing that for Zariski-Lipman we need characteristic $0$, see also \cite[Section~7]{LipmanfreeDerivation}). A basis for the derivations is given by $\delta= \frac{ \partial}{\partial x}$ and $\epsilon = z^p \frac{ \partial}{\partial y}-y^p \frac{ \partial}{\partial z}$. The two derivations commute, and $\delta$ is a unitary derivation but $\epsilon$ is not. From that we get that $\Sym^n(\Der_K(R,R)) \cong R^{n+1}$ with the basis $\delta^i\epsilon^j$, $i+j=n$. Therefore also the double duals $(\Sym^n(\Omega_{R|K} ))^{**} $ are free and hence the symmetric signature is $1$, though $R$ is normal and not regular.

To set up the Jacobi-Taylor matrices, only the following entries are relevant (and those with $z$ instead of $y$).
\[ \frac{1}{p!} \left( \partial_x \right)^p (f) =1,\,  \partial_y (f) = y^p   ,\,
\frac{1}{p!} \left( \partial_y \right)^{p} (f) =y,\, 
 \frac{1}{(p+1)!} \left( \partial_y \right)^{p+1} (f) =1 .  \]

For the element $x$ the unitary derivation $\delta$ sends $x$ to $1$. But for $x^2$ we have to go in characteristic $2$ up to order $8$ to find an operator sending $x^2$ to $1$. A computation with the Jacobi-Taylor matrices yields
\[ a^2+b^3+yb^4+y^3a^2b^3 +y^3a^4+y^4a^4b+y^5a^4b^2+y^7a^6b+y^{9} a^8 . \]
This operator is homogeneous of degree $-6$ and involves only partial derivatives with respect to $x$ and $y$.

We claim that
$\Delta^n/\Delta^{n+1} \cong \bigoplus_{\ell =0 }^n I^\ell$. This rests on the fact that the matrices $T_n$ in the sense of Remark~\ref{JacobiTaylorrelation} have the block matrix form
\[ T_n = \begin{pmatrix} 0 & 0& 0 & \cdots \\ M_1  & 0 &0 & \cdots \\ 0  & M_2 &0 & \cdots \\ 0& 0 & M_3 & \cdots \\ \vdots & \vdots & \vdots & \ddots  \end{pmatrix}  ,\, \text{ where }
 M_\ell = \begin{pmatrix} y^p & 0& \cdots& 0  \\ z^p  & y^p & 0 & \vdots \\ 0  & z^p & y^p & \vdots \\ \vdots   & \ddots & \ddots & \vdots \\ 0& 0 & z^p & y^p \\ 0 & \cdots & 0 & z^p  \end{pmatrix} \]
with $\ell +1$ rows. The $T_n$ and hence the Jacobi-Taylor matrices define injective maps, and so, 
$\Delta^n/\Delta^{n+1}$ is the cokernel of the $T_n$. The cokernel of every matrix $M_\ell$
is $I^\ell =(y^{\ell p}, y^{(\ell -1)p } z^p, \ldots , z^{\ell p})$.
The natural surjection
\[ \Sym^n( \Omega_{R|K}) \cong \bigoplus_{\ell =0}^n  I^{\tensor \ell}  \longrightarrow   \Delta^n/\Delta^{n+1} \cong \bigoplus_{\ell =0 }^n I^\ell    \]
is naturally given by $ I^{\tensor \ell}  \rightarrow I^\ell    $.

From this it also follows that in the exact complex
\[ 0 \longrightarrow D^{n-1}_{R|K} \longrightarrow D^n_{R|K} \longrightarrow \Hom_R (\Delta^n/\Delta^{n+1}, R ) = \Hom_R (\Sym^n( \Omega_{R|K}) ,R  ) \cong R^{n+1} \]
the last map is not surjective. The relation $(1,0, \ldots ,0)$ for the rows of the matrix $T_n$ can not for $n \geq p$ be extended to a relation on $J_n^\text{tr} $.
\end{example}

\section{Duality and convergence}\label{sec-duality}

In this section we discuss the differential signature for rings such that the associated graded of $D_{R|K}$ is a finitely generated $R$-algebra. Our main goal is to show that the differential signature is a limit rather than a limsup and that it is a rational number.

 {Our approach involves the approach of Matlis duality for $D$-modules. This idea appears first in work of Yekutieli \cite{YekDuality}, and is developed further developed by Switala~\cite{Switala} and Switala and Zhang~\cite{SwiZha}. We recall some facts about this duality; see \cite[\S 3 and \S 4]{Switala}. For an algebra $(R,\m,K)$ with coefficient field $K$, and $R$-modules $M,N$, we use the notation $\Hom^{\m\mathrm{-cts}}_{K}(M,N):=\varinjlim \Hom_K(M/\m^n M,N)$.\index{$\Hom^{\m\mathrm{-cts}}_{K}(M,N)$} We denote by $E$ the injective hull of the residue field.}

\begin{proposition}\label{MatlisD}
	Let $(R,\m,K)$ be a complete or graded ring with coefficient field $K$.
	\begin{enumerate}
		\item There is an exact functor $(-)^{\vee}$\index{$M^{\vee}$} from $R$-modules to $R$-modules that sends left $D$-modules to right $D$-modules and vice versa, such that $(-)^{\vee}$ agrees with Matlis duality up to $R$-isomorphism for $R$-modules that are Noetherian or Artinian.
		\item For $M$ Noetherian, one has $M^{\vee}=\Hom^{\m\mathrm{-cts}}_{K}(M,K)\cong\Hom_R(M,E)$. The last isomorphism comes from composition with a fixed $K$-linear projection onto the socle.
		\item The right $D$-action on $E=R^{\vee}=\Hom^{\m\mathrm{-cts}}_{K}(R,K)$ comes from precomposition with a differential operator.
	\end{enumerate}
\end{proposition}

\begin{remark}\label{RemJ}
We set $\cJ_{R|K}=\{\delta \in D_{R|K} \;|\; \delta(R)\subseteq \m\}= \bigcup_{n \in \NN} D^n_{R|K} (R,\m) $, i.e., the collection of all nonunitary operators.\index{$\cJ_{R \vert K}$}
Then, $\cJ_{R|K}$ is a right ideal of $D_{R|K}$ and $\m D_{R|K}\subseteq \cJ_{R|K}$.
\end{remark}

The following lemma is an immediate consequence of Proposition~\ref{MatlisD}.

\begin{lemma}\label{LemmaESimple}
Let $(R,\m,K)$ be a complete or graded ring with coefficient field $K$, and $E$ be the injective hull of $K$.
Suppose that $R$ is a simple $D_{R|K}$-module. Then, $E\cong R^\vee$ is a simple right $D_{R|K}$-module.
\end{lemma}

\begin{setup}\label{Setup9}
Let $(R,\m,K)$ be a complete or graded ring with coefficient field $K$,   and $E$ be the injective hull of $K$. As in Proposition~\ref{MatlisD}, we identify $E=\Hom^{\m\mathrm{-cts}}_{K}(R,K)$ and pick a generator $\eta\in \Hom^{\m\mathrm{-cts}}_{K}(R,K)$ for its socle. Let $\displaystyle G_{R|K}=\bigoplus_{n=0}^{\infty} \frac{D^n_{R|K}}{D^{n-1}_{R|K}}$ be the associated graded ring of $D_{R|K}$ with respect to the order filtration.
\end{setup}

We now present a few preparation lemmas in order to reduce, in some cases,  the study of differential signature to the classical Hilbert-Samuel theory.

\begin{lemma}\label{LemmaLenDeta}
 {In the situation of Setup~\ref{Setup9},
we} have the equality
$\lambda_R(R/\m\dif{n}{K})=\lambda_R(D^{n-1}_{R|K}\cdot \eta)$.
\end{lemma}
\begin{proof}
	Let $\eta:R\rightarrow K$ be the quotient map. We identify $\eta$ as a generator of $R^{\vee}$. Evidently, $f\in \m$ if and only if $\eta(f)=0$. We claim that $f\in \m\dif{n}{K}$ if and only if $\mu(f)=0$ for all $\mu\in D^{n-1}_{R|K}\cdot \eta$. Indeed, this is immediate from $D^{n-1}_{R|K}\cdot \eta = \{ \eta \circ \delta \ | \ \delta \in D^{n-1}_{R|K}\}$.
	
	To conclude the proof of the lemma, it suffices to show that, given a finite length submodule $N\subseteq R^{\vee}=\Hom^{\m\mathrm{-cts}}_{K}(R,K)$, the ideal $I=\{r\in R \ | \ \psi(r)=0 \ \text{for all} \ \psi \in N\}$ satisfies $\lambda_R(N)=\lambda_R(R/I)$. To see this, write $\widetilde{N}$ for the image of $N$ in $\Hom_R(R,E)$ via Proposition \ref{MatlisD}(2), and set $J=\{r\in R \ | \ \rho(r)=0 \ \text{for all} \ \rho \in \widetilde{N} \}$. It is evident that $J\subseteq I$. If $r\notin J$, there is some $\rho\in \widetilde{N}$ and $\theta \in E\setminus\{0\}$ with $\rho(r)=\theta$. Since $E$ is divisible, there is some $s\in R$ such that $s\theta$ is nonzero in the socle. Then $s\rho$ is a map in $\widetilde{N}$ that corresponds to a map in $N$ that sends $r$ to a nonzero element, so $r\notin I$. Thus $I=J$, so $\lambda_R(R/I)=\lambda_R(R/J)=\lambda_R(\widetilde{N})=\lambda_R(N),$
	where the middle equality is a standard fact from Matlis duality.
\end{proof}

\begin{remark}
 {In the situation of Setup~\ref{Setup9},
	the} cyclic $D^{n-1}_{R|K}$-module $D^{n-1}_{R|K} \cdot \eta \subseteq \Hom_K(R,K) $ is isomorphic to $D^{n-1}_{R|K}/D^{n-1}_{R|K}(R,\m)  $. Therefore, if $R$ is essentially of finite type over $K$ with residue class field $K$, the equality of Lemma \ref{LemmaLenDeta}
follows also directly from Proposition~\ref{freerank} or Remark~\ref{equalityvariant}. 
\end{remark}

\begin{lemma}\label{LemmaGradedE}
Suppose that $R$ is $D_{R|K}$-simple. Then, the map $\psi:D_{R|K} \to E$ defined by $\psi(\delta)=\eta\circ\delta$
is a surjective morphism of right $D_{R|K}$-modules with kernel $\cJ_{R|K}$.
As a consequence, $$
\bigoplus_{n=0}^{\infty} \frac{D^n_{R|K}\cdot \eta}{D^{n-1}_{R|K}\cdot \eta}\cong \bigoplus_{n=0}^{\infty} \frac{D^n_{R|K}}{ \cJ_{R|K} \cap D^{n}_{R|K}+D^{n-1}_{R|K}}
$$
as graded $G_{R|K}$-modules.
\end{lemma}
\begin{proof}
Since $R$ is a simple left $D_{R|K}$-module, we have that  $E$ is a simple $D_{R|K}$-module by Lemma~\ref{LemmaESimple}.
Since $\eta\neq 0,$ we have that $E$ is generated by $\eta$ as $D_{R|K}$-module. Then, $\psi$ is a surjective map.

We now show that $\Ker(\psi)=\cJ_{R|K}.$
We have that
\begin{align*}
\delta \in \Ker(\psi) &\Longleftrightarrow \eta\circ \delta=0\\
&\Longleftrightarrow \eta( \delta(f))=0\; \;\forall f\in R\\
&\Longleftrightarrow \delta(f)\in\Ker(\eta)\;\;\forall f\in R\\
&\Longleftrightarrow  \delta(f)\in \m\; \;\forall f\in R.
\end{align*}

We conclude that $E$ is isomorphic to $D_{R|K}/\cJ_{R|K}$.
The last claim follows from giving to $D_{R|K}/\cJ_{R|K}$ the filtration by the image of the filtration $\{D^n_{R|K}\}$ and passing to the associated graded.
\end{proof}

The following theorem presents the existence and rationality of differential signature for rings such that $G_{R|K}$ is a finitely generated $R$-algebra.

\begin{theorem}\label{existenceandraionality} Let $(R,\m,\kk)$ be an algebra with coefficient field $K$, and suppose that  $G_{R|K}$ is a finitely generated $R$-algebra. Then the sequence $\displaystyle \frac{\lambda_R(R/\m\dif{n}{K})}{n^d/d!}$ converges to $\dm{K}(R)=\dm{K}(\widehat{R})$, and the limit is rational.
\end{theorem}
\begin{proof}
%	Let $L$ be a coefficient field of $\widehat{R}$ that contains $K$. By Proposition~\ref{diff-sig-completion}, $\dm{K}(R)=\dm{L}{(\widehat{R})}$, and the same equalities hold for the limits inferior. Additionally,  we have that $G_{\widehat{R}|L} = G_{\widehat{R}|K} \cong \widehat{R}\otimes_R G_{R|K}$ is a finitely generated $\widehat{R}$-algebra from Proposition~\ref{diff-ops-completion}. If $R$ is not $D_{R|K}$-simple, then $\dm{K}(R)=0$. 

 {By Proposition~\ref{diff-ops-completion},   $G_{\widehat{R}|K} \cong \widehat{R}\otimes_R G_{R|K}$ is a finitely generated $\widehat{R}$-algebra. If $R$ is not $D_{R|K}$-simple, then $\dm{K}(R)=0$. }
	
Now, we can assume that $R$ is a complete local ring and a simple  $D_{R|K}$-module.	
Let $\gr(E)=\bigoplus_{n \in \NN } \frac{D^n\cdot\eta}{D^{n-1}\cdot\eta}.$
We note that $D^n_{R|K}/\m D^n_{R|K}$ surjects onto $D^n_{R|K}/(\cJ_{R|K}\cap D^n_{R|K}).$
Then, $\displaystyle G_{R|K}/\m G_{R|K}\to \gr(E)$ is a surjection of graded $G_{R|K}$-modules.
As a consequence, we have that 
$\gr(E)$ is a cyclic. Therefore, it is a finitely generated, graded $G_{R|K}/\m G_{R|K}$-module.
Then, 
$$
\dm{K}(R)=\lim\limits_{n\to\infty}\frac{\lambda_R(R/\m\dif{n}{K})}{n^d/d!}
=\lim\limits_{n\to\infty}\frac{\dim_K \gr(E)_{\leq n-1} }{n^d/d!}
$$
by Lemmas~\ref{LemmaLenDeta} and \ref{LemmaGradedE}. The convergence and rationality statements follow from the last description by standard Hilbert function theory.
\end{proof}

\section{Applications to symbolic powers}

In this section, we discuss some connections between differential operators, symbolic powers, and singularities. A classical theorem of Zariski and Nagata \cite{ZariskiHolFunct,Nagata} characterizes the symbolic powers of primes in $\CC[x_1,\dots,x_n]$ as differential powers: $\p\dif{n}{\CC}=\p^{(n)}$. More generally, if $K$ is a perfect field, and  $R=K[x_1,\ldots,x_n]$, and $I\subseteq R$ is a radical ideal, then $I\dif{n}{K}=I^{(n)}$. We point out that there is a  recent extension of this result to mixed characteristic using $p$-derivations  \cite{DSGJ}.

With the general notion of differential powers, one may ask to what extent the Zariski-Nagata theorem holds over other $K$-algebras $R$. It turns out that this is very closely tied to the singularities of $R$. The first two results below show that in reasonably geometric situations, the Zariski-Nagata theorem actually characterizes smoothness.

\begin{proposition}\label{ZN-reg-primes} Let $K$ be a perfect field, and $R$ be a ring essentially of finite type over $K$. Let $\p \in \Spec(R)$ be a prime such that $R_{\p}$ is regular. Then, $\p\dif{n}{K}=\p^{(n)}$.
\end{proposition}
\begin{proof} 
	By Proposition~\ref{properties-diff-powers}, we have that $\p\dif{n}{K} \supseteq \p^{(n)}$,
	and that both ideals are $\p$-primary. It suffices to check the equality after localizing at $\p$ and completing.
	Then, by Lemma~\ref{diff-localize2} and Lemma~\ref{diff-powers-completion}, 
	$\p\dif{n}{K} \widehat{R_{\p}} = (\p \widehat{R_{\p}} )\dif{n}{K}$.
	
	Since $K$ is perfect, the residue field of $R_{\p}$ is separable over $K$. Then, there exists a coefficient field $L$ for $\widehat{R_{\p}}$ containing $K$ \cite[Theorem~28.3~(iii,iv)]{MatsumuraRing}.
	Since $R_{\p}$ is regular, we have $\widehat{R_{\p}}\cong L\llbracket y_1,\dots, y_e \rrbracket$ for some $e$. Under this isomorphism, $\p=(y_1,\dots, y_e)$. By Lemma~\ref{diff-powers-completion}, it follows that $(\p \widehat{R_{\p}} )\dif{n}{L} = \p^n \widehat{R_{\p}}$. We obtain containments
	\[ \p^n \widehat{R_{\p}} = (\p \widehat{R_{\p}})^n \subseteq (\p \widehat{R_{\p}} )\dif{n}{K} \subseteq (\p \widehat{R_{\p}} )\dif{n}{L} = \p^n \widehat{R_{\p}}, \]
	so that equality must hold throughout.
\end{proof}

For maximal ideals in algebras with pseudocoefficient fields, the converse holds. We thank Mel Hochster for helping us complete the proof below.

\begin{theorem}\label{Mel}
	Let $(R,\m,\kk)$ be a domain with pseudocoefficient field $K$ such that $\Frac(R)$ is separable over $K$.
	Then, $\m\dif{n}{K}=\m^n$ for some $n \geq 2$ if and only if $R$ is smooth over $K$. Furthermore, if $\kk$ is perfect, then the previous statements are equivalent to the property  that  $\ModDif{n}{ {R}}{K}$ is free for some $n \geq 1$.
\end{theorem}
\begin{proof}
	Since we already know that if $R$ is regular, then  $\m\dif{n}{K}=\m^n$, we focus on the other implication.
	We have that $\lambda_R(R/\m^n) =\lambda_R(R/\m\dif{n}{K})\leq \binom{n+d}{d}$ by Proposition~\ref{RankDiff}.
	We now show that this inequality forces $R$ to be regular. For this purpose, after a flat base change, we can reduce to the case where $K$ is infinite and perfect. Then there exists a $d$-generated minimal reduction of $\m$, say $J=(x_1,\dots,x_d)$. Let $G=\gr_{\m}(R)$ be the associated graded ring of $R$. Set $T=K \otimes_R \gr_{J}(R) \cong K[\overline{x_1},\dots,\overline{x_d}]$; since $J$ is generated by a system of parameters, this is a $d$-dimensional polynomial ring over $K$. We claim that $T$ is a graded subring of $G$. Since $J$ is a minimal reduction of $\m$, we have that $J^n \cap \m^{n+1} = \m J^n$ \cite[Corollary~8.3.6]{SwansonHuneke}.
	
	Suppose that $R$ is not regular. Then $\dim_K G_1 >d$. Since $\dim_K G_i \geq \dim_K T_i = \binom{i+d-1}{d-1}$ for every $i$, we get since $n \geq 2$
	\[ \lambda_R(R/\m^n) = \sum_{i=0}^{n-1} \dim_K G_i \geq \sum_{i=0}^{n-1} \dim_K T_i + 1 > \binom{n+d}{d},\]
	contradicting the inequality above. Thus, $R$ is regular.
	The last statement follows from Propositions~\ref{freerank} and \ref{RankDiff}.
\end{proof}

The previous result was independently and simultaneously proven for hypersurfaces by  Barajas and Duarte \cite{BarajasDuarte}.

	We also have an analogue of the Zariski-Nagata Theorem/Proposition~\ref{ZN-reg-primes} that describes the differential Frobenius powers of primes outside of the singular locus.
	
	\begin{proposition}\label{ZN-frobenius} Let $K$ be a perfect field, and $R$ be a ring essentially of finite type over $K$. Let $\p \in \Spec(R)$ be a prime such that $R_{\p}$ is regular. Then $\p^\Fdif{p^e}$ is the $\p$-primary component of $\p^{[p^e]}$. In particular, if $R$ is regular, $\p^\Fdif{p^e}=\p^{[p^e]}$.
	\end{proposition}
	\begin{proof} The proof is the same as that of Proposition~\ref{ZN-reg-primes}, using Lemmas~\ref{properties-Fdiff-powers}~and~\ref{Fdiff-localize}. The second claim follows from the first because $\p^{[p^e]}$ is $\p$-primary in a regular ring. 
	\end{proof}

	\begin{remark}\label{Kunz}
		A result of Kunz~\cite{KunzReg} combined with Propositions~\ref{ZN-frobenius}~and~\ref{PropDvsCartIdeals} gives a characterization of regularity for a local $F$-pure $F$-finite ring $(R,\m)$. Namely, the following are equivalent:
		\begin{itemize}
			\item $R$ is a regular ring;
			\item $\m^\Fdif{p^e}=\m^{[p^e]}$ for every $e\in \NN;$
			\item $\m^\Fdif{p^e}=\m^{[p^e]}$ for some $e\in\NN.$
		\end{itemize}

		Then, we can think of the Zariski-Nagata theorem and Theorem~\ref{Mel} as a differential version analogue of Kunz's Theorem. We point out that unlike Kunz's result, Theorem~\ref{Mel} is characteristic-free.
		
	\end{remark}

The comparison between symbolic powers and differential operators also reflects finer qualities of singularities, beyond smoothness versus nonsmoothness. In strongly $F$-regular rings, the Zariski-Nagata Theorem fails, but the topologies defined by symbolic powers and differntial powers are linearly equivalent.

\begin{theorem}[Linear Zariski-Nagata Theorem]
	Let $R$ be an $F$-finite $F$-pure $K$-algebra, where $K$ is a perfect field. Then, for any $\p\in \Spec(R)$,
	\[ R_{\p} \text{ is strongly $F$-regular } \Longleftrightarrow \exists C>0 : \p\dif{Cn}{K}\subseteq \p^{(n)}\text{ for all }n>0. \]
\end{theorem}

\begin{proof}
	Since the ideals $\p^{(n)}$ and $\p\dif{n}{K}$ are $\p$-primary for all $n$, by Proposition~\ref{diff-localize}, the condition on the right-hand side is equivalent to that for some $C$, $(\p R_{\p})\dif{Cn}{K}=\p\dif{Cn}{K} R_{\p} \subseteq \p^{n}R_{\p}$ for all $n>0$. Thus, since $F$-finiteness and $F$-purity localize, we can assume that $(R,\m)$ is local and $\p=\m$.
	
	If $R$ is not strongly $F$-regular, then $R$ is not $D$-simple  \cite{DModFSplit}, so $\cP_K\neq 0$ by Corollary~\ref{CorDifPrimeDsimple}. Since $\bigcap_{n\in \NN} \m^n =0$, the condition on the right-hand side fails.
	
	Now, assume that $R$ is strongly $F$-regular. For an integer $n$, set $l(n)=\lceil \log_p(n) \rceil$: this is the smallest integer $e$ such that $p^e\geq n$. Observe that  $n \leq p^{l(n)} \leq pn$. If $\mu$ is the embedding dimension of $R$, then $D^{(e)}_{R | K} \subseteq D_{R|K}^{\mu(p^e-1)}$.  We obtain that $\m\dif{\mu p^e}{K} \subseteq  I_e(R)$.	There is a constant $e_0$ such that $I_{e+e_0}(R)\subseteq \m^{[p^e]}$ for all $e$ \cite{AL}. Put together, we obtain
	
	\[ \m\dif{\mu p^{e_0+1} n}{K} \subseteq \m\dif{\mu p^{l(n)+e_0}}{K} \subseteq I_{l(n) + e_0}(R) \subseteq \m^{[p^{l(n)}]} \subseteq \m^n, \]
	so the constant $C=\mu p^{e_0+1}$ suffices.
\end{proof}

This result should be compared with the linear comparison between
ordinary and symbolic powers \cite{SwansonLinear}; see also \cite{HKV,HKV2}. Connections between strong $F$-regularity and symbolic powers are not new; we refer the reader to \cite{HH-Symb,GrifoHuneke,Smolkin,JavierSmolkin}.

We end this section with an algorithm to compute symbolic powers for radical ideals. For algorithmic aspects of the computations of differential powers for other rings, see Remark~\ref{JacobiTayloralgorithms}. 

\begin{proposition}\label{PropAlgSymbPowers}
	Let $S=K[x_1,\ldots,x_d]$  is a polynomial ring over a field $K$, and $J$ an ideal.
	Let $T=K[x_1,\ldots,x_d,\tilde{x}_1 , \ldots , \tilde{x}_d]\cong \ModDif{}{S}{K}$, and $\Delta=(x_1-\tilde{x}_1,\dots,x_d-\tilde{x}_d)$.
	Then,
	$$
	J\dif{n}{K}=\Big( \tilde{J} + \Delta^n \Big)\cap S,
	$$
	where $\tilde{J}$  denotes the ideal in $T$ generated by the elements in $J$ written in the variables $\{\tilde{x}_i\}$. As a consequence,
	if $K$ is perfect and $I$ is radical, then 
	$$J^{(n)}=\Big( \tilde{J} + \Delta^n \Big)\cap S.$$
\end{proposition}
\begin{proof} Let $\tilde{S}$ be the polynomial subring of $T$ generated by the variables $\{\tilde{x}_i\}$.
	We know that $T\cong {\ModDif{}{S}{K}}$ and that $d$ corresponds to the inclusion $S\to T.$
	We note that 
	$T/\Delta^n \cong {\ModDif{n-1}{S}{K}}$. 
	We have then that
	$T/\Delta^n$ is a free $S$-module. As a consequence, $\Big( \tilde{J} + \Delta^n \Big)=\bigcap_{\phi} \phi^{-1}(\tilde{J})$, where $\phi$ runs over all $\tilde{S}$-module morphisms $\phi:S/\Delta^n\to  \tilde{S}/\tilde{J}$.
	Then, by Proposition~\ref{universaldifferential}, we have 
	$J\dif{n}{K}=\Big( \tilde{J} + \Delta^n \Big)\cap S$.
	The claim about symbolic powers follows from the characterization of differential powers in this case \cite{SurveySP}.  
\end{proof}

The formula in Proposition \ref{PropAlgSymbPowers} is similar in spirit to the characterization of symbolic powers in terms of joins of ideals \cite{Sullivant}. The key advantage of the formula above is that the computation involves only twice (not three times) as many variables.

\section*{Acknowledgments}

  We thank Josep \`Alvarez Montaner, Alessio Caminata, Elo\'isa Grifo,  Daniel Hern\'andez, Mel Hochster, Craig Huneke, Luis Narv\'aez-Macarro, Claudia Miller, Mircea Musta\c{t}\u{a}, Julio Moyano, Eamon Quinlan, Anurag K.~Singh, Ilya Smirnov, Karen E.~Smith, Jonathan Steinbuch, Robert Walker, Wenliang Zhang,  {and the anonymous referee} for helpful conversations and comments. 

We are very thankful for the work of  Gennady Lyubeznik, which introduced us to $D$-module theory.
%%%%%%%%%%%%%%%%%%%%%%%%%%%%%%%%%%%%%%%%%%%%%%%%%%%%%%%%%%%%%

\printindex

%%%%%%%%%%%%%%%%%%%%%%%%%%%%%%%%%%%%%%%%%%%%%%%%%%%%%%%%%

\bibliographystyle{alpha}
\bibliography{References}

\end{document}